\documentclass[11pt]{amsart}

\usepackage{upgreek}
 \usepackage{mathpazo}
\usepackage{ytableau}

\usepackage{amsmath, amsxtra, amsthm, amsfonts,amssymb,mathtools,bm}
\usepackage{mathrsfs}
\usepackage[normalem]{ulem}
\usepackage[mathscr]{euscript}
\usepackage{graphicx}
\usepackage{url}
\usepackage{color}
\usepackage{bbm}
\usepackage{tikz-cd}
\usepackage{tikz}

\usepackage[T1]{fontenc}

\usepackage{enumitem}
\usepackage{comment}
\usepackage[pdftex,hidelinks,backref=page]{hyperref}
\hypersetup{
    colorlinks,
    citecolor=magenta,
    filecolor=magenta,
    linkcolor=blue,
    urlcolor=black
}
\usepackage{cleveref}

\renewcommand{\emptyset}{\varnothing}
\newcommand{\NN}{\mathbb N}

\newcommand{\RR}{\mathbb R}
\newcommand{\CC}{\mathbb C}
\newcommand{\PP}{\mathbb P}
\newcommand{\ZZ}{\mathbb Z}

\theoremstyle{definition}
\newtheorem{thm}{Theorem}[section]
\newtheorem{cor}[thm]{Corollary}
\newtheorem{lem}[thm]{Lemma}

\newtheorem{prop}[thm]{Proposition}
\newtheorem{defn}[thm]{Definition}

\newtheorem{conj}[thm]{Conjecture}
\newtheorem{eg}[thm]{Example}
\newtheorem{rem}[thm]{Remark}

\newtheorem{question}[thm]{Question}
\newtheorem{maintheorem}{Theorem}	\newtheorem{fact}[thm]{Fact}

\numberwithin{equation}{section}

\newcommand{\ar}[1]
{{\xrightarrow{#1}}}

\newcommand{\indexedforests}[1][]{ 
{\ifx&#1&%
    \operatorname{\mathsf{For}}
\else
    \operatorname{\mathsf{For}}^{#1}
\fi}
}
\newcommand{\forestpoly}[2][]{{\mathfrak{P}_{#2}^{\underline{#1}}}} 

\newcommand{\qsym}[2][]{
{\ifx&#1&%
  {\operatorname{QSym}_{#2}}
\else
  {{}^{#1}\!\operatorname{QSym}_{#2}}
\fi}
} 

\newcommand{\qseq}[2][]{
{\ifx&#1&%
  {\operatorname{QSeq}_{#2}}
\else
  {{}^{#1}\!\operatorname{QSeq}_{#2}}
\fi}
}

\newcommand{\qsymide}[2][]{
{\ifx&#1&%
  {\operatorname{QSym}_{#2}^+}
\else
  {{}^{#1}\!\operatorname{QSym}_{#2}^+}
\fi}
} 
\newcommand{\sym}[1]{\operatorname{Sym}_{#1}} 
\newcommand{\symide}[1]{\sym{#1}^+} 
\newcommand{\lcode}[1]{\operatorname{lcode}(#1)} 
\newcommand{\compatible}[2][]{
{\ifx&#1&%
  {\mathcal{C}(#2)}
\else
  {\mathcal{C}^{m}(#2)}
\fi}
} 
\newcommand{\internal}[1]{\operatorname{IN}(#1)} 
\newcommand{\suchthat}{\;|\;}
\newcommand{\nvect}{\mathsf{Codes}}

\newcommand{\poly}{\operatorname{Pol}} 
\newcommand{\schub}[1]{\mathfrak{S}_{#1}} 
\newcommand{\des}[1]{\operatorname{Des}(#1)} 
\date{}

\newcommand{\idem}{\operatorname{id}} 
\newcommand{\slide}[2][]{
{\ifx&#1&%
  {\mathfrak{F}_{#2}}
\else
  {\mathfrak{F}_{#2}^{\underline{#1}}}
\fi}
} 
\newcommand{\gz}{\operatorname{GZ}} 
\newcommand{\Perm}{\operatorname{Perm}} 
\newcommand{\ct}{\operatorname{ev_0}} 
\newcommand{\qdes}[2][]{\operatorname{LTer}_{#1}(#2)} 
\newcommand{\tope}[2][]{
{\ifx&#1&%
  {\mathsf{T}_{#2}}
\else
  {\mathsf{T}_{#2}^{\underline{#1}}}
\fi}
} 
\newcommand{\bope}[2][]{
{\ifx&#1&%
  {\mathsf{B}_{#2}}
\else
  {\mathsf{B}_{#2}^{(#1)}}
\fi}
} 
\newcommand{\rope}[1]{\mathsf{R}_{#1}} 
\newcommand{\End}{\operatorname{End}} 
\newcommand{\Trim}[1]{\operatorname{Trim}({#1})} 
\newcommand{\Th}[1][]{\mathsf{ThMon}^{\underline{#1}}} 

\newcommand{\zigzag}[2][]
{
{\ifx&#1&%
  {\mathsf{ZigZag}_{#2}}
\else
  {\mathsf{ZigZag}_{#2}^{#1}}
\fi}
}

\newcommand{\ltfor}[2][] 
{
{\ifx&#1&%
  {\mathsf{LTFor}_{#2}}
\else
  {\mathsf{LTFor}_{#2}^{#1}}
\fi}
}

\newcommand{\rtfor}[2][] 
{
{\ifx&#1&%
  {\mathsf{RTFor}_{>#2}}
\else
  {\mathsf{RTFor}_{>#2}^{#1}}
\fi}
}

\newcommand{\suppfor}[2][] 
{
{\ifx&#1&%
  {\mathsf{For}_{#2}}
\else
  {\mathsf{For}_{#2}^{#1}}
\fi}
}

\newcommand{\binfor}[1][]{
{\ifx&#1&%
    \operatorname{\mathsf{BinFor}}
\else
    \operatorname{\mathsf{BinFor}}^{#1}
\fi}
} 

\newcommand{\hqsym}[2][]{
{\ifx&#1&%
  {\operatorname{HQSym}_{#2}}
\else
  {\operatorname{HQSym}_{#2}^{#1}}
\fi}
} 
\newcommand{\fl}[1]{\mathrm{Fl}_{#1}}
\newcommand{\coinv}[1]{\operatorname{Coinv}_{#1}} 

\newcommand{\qscoinv}[2][]{
{\ifx&#1&%
  {\operatorname{QSCoinv}_{#2}}
\else
  {{}^{#1}\!\operatorname{QSCoinv}_{#2}}
\fi}
}

\newcommand{\cube}{\mathbf{C}} 
\definecolor{ao}{rgb}{0.0, 0.5, 0.0}
\newcommand{\vope}[2][]{
{\ifx&#1&%
  {\mathsf{V}_{#2}}
\else
  {\mathsf{V}_{#2}^{\underline{#1}}}
\fi}
} 

\newcommand{\ins}{\varepsilon} 

\newcommand{\rt}{\Omega}
\newcommand{\prodoperator}[1]{\Pi_{#1}}
\newcommand{\rtc}{\prodoperator{\rt}}
\newcommand{\rtseq}{\mathrm{RTSeq}}

\newcommand{\hhmp}{\mathrm{HHMP}}

\newcommand{\mnfor}{\widetilde{\indexedforests}} 
\newcommand{\mnf}[1]{\wt{F}(#1)} 
 
\newcommand{\nsuppfor}[1]{\widehat{\indexedforests}_{#1}} 
\newcommand{\nfor}{\widehat{\indexedforests}} 
\newcommand{\nf}[1]{\wh{F}(#1)} 

\newcommand{\wt}[1]{\widetilde{#1}} 
\newcommand{\wh}[1]{\widehat{#1}} 
\newcommand{\ul}[1]{\underline{#1}} 
\newcommand{\rletter}[1]{\mathsf{r}_{#1}}
\newcommand{\tletter}[1]{\mathsf{t}_{#1}}
\newcommand{\xletter}[1]{\mathsf{x}_{#1}}
\newcommand{\BS}{\operatorname{BS}}
\newcommand{\gr}{\operatorname{Gr}} 

\title{The geometry of quasisymmetric coinvariants}

\author{Philippe Nadeau}
\address{Universite Claude Bernard Lyon 1, CNRS, Ecole Centrale de Lyon, INSA Lyon, Université Jean Monnet, ICJ UMR5208, 69622 Villeurbanne, France}
\email{\href{mailto:nadeau@math.univ-lyon1.fr}{nadeau@math.univ-lyon1.fr}}

\author{Hunter Spink}
\address{Department of Mathematics,
University of Toronto, Toronto, ON M5S 2E4, Canada}
\email{\href{mailto:hunter.spink@utoronto.ca}{hunter.spink@utoronto.ca}}

\author{Vasu Tewari}
\address{Department of Mathematical and Computational Sciences, University of Toronto Mississauga, Mississauga, ON L5L 1C6, Canada}
\email{\href{mailto:vasu.tewari@utoronto.ca}{vasu.tewari@utoronto.ca}}

\thanks{
PN was partially supported by French ANR grant ANR-19-CE48-0011 (COMBIN\'E). HS and VT acknowledge the support of the Natural Sciences and Engineering Research Council of Canada (NSERC), respectively [RGPIN-2024-04181] and [RGPIN-2024-05433].}


\begin{document}

\begin{abstract}
We develop a quasisymmetric analogue of the theory of Schubert cycles, building off of our previous work on a quasisymmetric analogue of Schubert polynomials and divided differences. Our constructions result in a natural geometric interpretation for the ring of quasisymmetric coinvariants.
\end{abstract}

\maketitle

\setcounter{tocdepth}{1}
\tableofcontents

\section{Introduction}
Let $\poly_n\coloneqq \ZZ[x_1,\ldots,x_n]$, the ring of polynomials in $n$ variables. A ubiquitous family of polynomials in algebraic combinatorics is the ring of quasisymmetric polynomials $\qsym{n}\subset \poly_n$, which are variable truncations of the quasisymmetric functions pioneered by Gessel \cite{Ges84} and Stanley \cite{StThesis}. These polynomials are by definition those that satisfy a weak form of variable symmetry: for any sequence $a_1,\ldots,a_k\geq 1$ and \emph{increasing} sequence $1\le i_1<\cdots < i_k \le n$, the coefficient of $x_1^{a_1}\cdots x_k^{a_k}$ is the same as the coefficient of $x_{i_1}^{a_1}\cdots x_{i_k}^{a_k}$. Note that without the increasing stipulation we recover the familiar notion of symmetric polynomials $\sym{n}\subset \poly_n$, and so we have containments $\sym{n}\subset \qsym{n}\subset \poly_n$.

A significant gap in our understanding of quasisymmetric polynomials is the subject of a research program \cite{ABB04, BeGa23,NST_a, PeSa23,  PeSa22} which seeks to answer the following  question.

\begin{question}
\label{question:qsymanalogue}
    What is the quasisymmetric analogue of Schubert calculus?
\end{question}
The combinatorial side of Schubert calculus consists of the interplay between Schubert polynomials $\{\schub{w}\suchthat w\in S_n\}$,  divided differences $\partial_i\in \End(\poly_n)$ which recursively characterize $\schub{w}$, and the symmetric coinvariants $\coinv{n}\coloneqq \poly_n/\symide{n}$ where $\symide{n}$ is the ideal generated by homogenous positive degree symmetric polynomials. 
These correspond on the geometric side to Schubert cycles $X^w\subset \fl{n}$ in the complete flag variety, Bott--Samelson resolutions, and the cohomology ring $H^\bullet(\fl{n})$. 

The ring of quasisymmetric coinvariants is defined to be $\qscoinv{n}\coloneqq\poly_n/\qsymide{n}$, where $\qsymide{n}$ is the ideal generated by homogenous positive degree quasisymmetric polynomials. An obstruction to answering \Cref{question:qsymanalogue} is that $\qscoinv{n}$ is less well-behaved than $\coinv{n}$. For example, as shown by Aval-Bergeron-Bergeron \cite{ABB04} (see also \cite{NST_a}), the graded ranks are given by 
$$\operatorname{rank} (\operatorname{QSCoinv}_{n}^{(i)})=\frac{n-i}{n+i}\binom{n-i}{i},$$
for $i=0,\ldots, n-1$. This sequence is neither unimodal nor symmetric, so in particular implies that $\qscoinv{n}$ is not the cohomology ring of a smooth projective variety.

In a previous paper \cite{NST_a} we developed the combinatorial and algebraic sides of the quasisymmetric story, by describing \emph{quasisymmetric divided differences} $\tope{i}$ that interact with \emph{forest polynomials} $\forestpoly{F}$ and the quasisymmetric coinvariants $\qscoinv{n}$ in an analogous way to how the usual divided differences $\partial_i$ interact with Schubert polynomials $\schub{w}$ and the symmetric coinvariants $\coinv{n}$. 
One interesting feature of these quasisymmetric divided differences is that the role of the nil-Hecke relations $\partial_i^2=0$, $\partial_i\partial_j=\partial_j\partial_i$ for $|i-j|\ge 2$ and $\partial_i\partial_{i+1}\partial_i=\partial_{i+1}\partial_i\partial_{i+1}$ is played by the Thompson monoid relations $$\tope{i}\tope{j}=\tope{j}\tope{i+1}$$ for $i>j$. This implies that composites of $\tope{i}$ operators are naturally indexed by certain leaf-labelled plane binary forests. 
The relevant parts of this combinatorial theory are recalled in \Cref{sec:recallforests}.
\subsection{The geometry of quasisymmetric coinvariants}
Here we complete this story, and in particular give a natural answer to \Cref{question:qsymanalogue}, by describing a quasisymmetric analogue of the geometric theory of Schubert cycles.
To obtain the clearest geometric picture, we have to extend in \Cref{sec:sequences_and_forests} the combinatorial considerations in \cite{NST_a} on composites of $\tope{i}$ operators to cover composites of both $\tope{i}$ operators and the \emph{Bergeron--Sottile maps} $\rope{j}$, so named in \emph{ibid.} because of their relevance in the seminal paper \cite{BS98} that developed Pieri rules for Schubert calculus. 
These lead us to a new combinatorial object we call the \emph{augmented Thompson monoid} which governs the combinatorics of certain nested plane binary forests.

We will show that the distinct composite operations $\prodoperator{\rt}^n:\poly_n\to \ZZ$ of appropriately supported sequences $\rt$ of $n$ operations of the form $\tope{i}$ or $\rope{j}$ are naturally indexed by  nested plane binary forests $\wh{F}\in \nfor_n$. To such an $\wh{F}$ we show that there is an associated toric variety $X(\wh{F})\subset \fl{n}$ which we call a \textbf{quasisymmetric Schubert cycle}, with the property that the degree map on $X(\wh{F})$ is computed using this composite operation. Our first main theorem (\Cref{maintheorem:QsymSchubertcycles}) shows that the subcollection of $X(\wh{F})$ where $\wh{F}$ does not involve any nesting are dual to the forest polynomials in the same way that Schubert cycles are dual to Schubert polynomials. 

To the sequence $\rt$ itself, we construct a  toric Richardson variety $X(\rt)\subset \fl{n}$ we call a \textbf{quasisymmetric Bott-Samelson variety}. This toric variety is inductively constructed as an iterated $\PP^1$-bundle (i.e. as a ``Bott manifold'') in such a way that the degree map on $X(\rt)$ is computed as the composite operation $\prodoperator{\rt}$ via taking successive degrees of the $\PP^1$-bundle (\Cref{thm:introP1degree}). Furthermore, we show that there is a natural map $X(\rt)\to X(\wh{F})$ (in fact an isomorphism) that geometrically computes the degree map on $X(\wh{F})$ via pullback to $X(\rt)$, as Bott-Samelson varieties do for Schubert varieties.

The $X(\rt)$ 
fit together into a toric complex $$\hhmp_n\coloneqq \bigcup_{\rt\in \rtseq_n}X(\rt)\subset \fl{n}$$ we call the \textbf{$\rt$-flag variety}, whose top-dimensional pieces were first studied by Haruda--Horiguchi--Masuda--Park \cite{HHMP} and later by Lian \cite{lian2023hhmp}. Our second main theorem (\Cref{maintheorem:QFLcontains}) shows that under the restriction map $\psi^*:H^\bullet(\fl{n})\to H^\bullet(\hhmp_n)$, we have 
$$\qscoinv{n}\cong \psi^*(H^\bullet(\fl{n}))\subset H^\bullet(\hhmp_n),$$
 which gives the desired geometric interpretation of the quasisymmetric coinvariants.

A curious feature of $\hhmp_n$ is that unlike in the classical story, each $X(\wh{F})$ appears isomorphically multiple times within $\hhmp_n$ (once for each $\rt\in \Trim{\wh{F}}$). This is ultimately why the containment of $\qscoinv{n}$ in $H^\bullet(\hhmp_n)$ is not an equality.
\subsection{Applications}
We give some applications of our theory beyond answering \Cref{question:qsymanalogue}. 
One such application of our theory comes from the fact that the top dimensional $X(\rt)$ show up in the study of the permutahedral toric variety in \cite{HHMP,lian2023hhmp}.  Because we have a good understanding of the degree map on $X(\rt)$ we can deduce new results about the degree map, which is combinatorially the ``divided symmetrization'' of Postnikov \cite[\S 3]{Pos09}. 
This also lets us better understand the $q$-divided symmetrization from \cite{NT21} and explains the remarkable interactions with quasisymmetric polynomials.

Another application comes from the fact that the degree map on Richardson varieties such as our toric varieties $X(\rt)$ computes generalized Littlewood--Richardson (LR) coefficients $c^v_{u,w}$, the structure coefficients for Schubert polynomial multiplication
$$\schub{u}\schub{w}=\sum c^v_{u,w}\schub{v}.$$ 
By manufacturing operators via composites of $\tope{i}$ and $\rope{j}$ that extract interesting combinatorial invariants from Schubert polynomials, our theory implies that these quantities are themselves generalized LR coefficients, and hence have geometric significance.

As a final application, we give a geometric explanation for a well-known formula of Gessel {\cite[Theorem 3]{Ges84}} for the coefficients in the expansion of a symmetric polynomial $f\in \sym{n}$ into fundamental quasisymmetric polynomials.  It makes use of the Hall inner product with skew Schur polynomials associated to ribbon shapes, and we  relate these inner products to the degrees on some of our varieties $X(\rt)$.

\subsection*{Outline of Paper} 
We detail the main results in the paper in \Cref{sec:Results}, and only give a brief outline here. \Cref{sec:compositesTR} contains the combinatorial description of the augmented Thompson monoid, which describes the composites of $\tope{}$ and $\rope{}$ operators. This extends the combinatorics of the Thompson monoid from \cite{NST_a} which dealt only with $\tope{}$ operators. In \Cref{sec:GeomRT} we describe a geometric interpretation of the $\tope{}$ and $\rope{}$ operators in the flag variety, and introduce the subvariety $\hhmp$. In \Cref{sec:rtRichardson}, using a pair of permutations associated to a sequence of $\tope{}$ and $\rope{}$ operators, we study the toric Richardson varieties $X(\rt)$ which are the primary geometric objects we consider.  In \Cref{sec:Bott} we relate for $\rt\in \Trim{\wh{F}}$ the toric varieties $X(\rt)$ to torus-orbit closures $X(\wh{F})\subset \fl{n}$ we call quasisymmetric Schubert cycles, and show a subset of our quasisymmetric Schubert cycles are dual to the forest polynomials. In \Cref{sec:qfl} we show that the $X(\rt)$ assemble into a toric complex $\hhmp_n$ such that its  cohomology ring naturally contains $\qscoinv{n}$. 

In \Cref{sec:generalizedLR} we apply our results to generalized LR coefficients. In \Cref{sec:DS} we apply our results to general torus-orbit closures and divided symmetrization. 
Finally, in \Cref{sec:Gessel} we show how we can use the projection of our quasisymmetric Schubert cycles to Grassmannians to recover a result of Gessel on extracting the coefficients of a symmetric polynomial in the basis of fundamental quasisymmetric polynomials.

In \Cref{sec:nested_forests_extra} we collect some combinatorial proofs related to nested forests and in \Cref{sec:uvPerms} we collect some combinatorial proofs related to the permutations defining the Richardson varieties $X(\rt)$. 
Finally in \Cref{sec:Moment} we compute the moment polytopes of these varieties  and give a different perspective on the existence of the quasisymmetric Bott--Samelson resolutions.

\subsection*{Acknowledgements}
We would like to thank
Nantel Bergeron, Chris Eur, Lucas Gagnon, Kiumars Kaveh, Alexander Kupers, Daniel Litt, and Arun Ram for helpful conversations.
\section{Results}
\label{sec:Results}

We start by recalling the classical story of Schubert cycles, Schubert polynomials, and divided differences.
Let $B,B^-,T\subset GL_n$ denote the upper triangular, lower triangular, and diagonal matrices, and let $\fl{n}=GL_n/B$ denote the complete flag variety. 
We will denote the Schubert cycles $X^v=\overline{BvB}$ and the opposite Schubert cycles $X_u=\overline{B^-uB}$. For permutations $u\le v$ in the Bruhat order we denote the Richardson variety $X^v_u=X^v\cap X_u$. 
Recall the Borel presentation \cite{Bor53} of the cohomology ring of the complete flag variety as the symmetric coinvariants $H^\bullet(\fl{n})=\coinv{n}\coloneqq\ZZ[x_1,\ldots,x_n]/\symide{n}$, where $x_1,\ldots,x_n$ are the Chern roots of the tautological quotient flag, and $\symide{n}$ is the ideal generated by positive degree symmetric polynomials in $x_1,\ldots,x_n$. 
A basis of $H^\bullet(\fl{n})$ is given by the fundamental classes $[X^w]$ of the Schubert cycles.

We can compute the degree map on the Schubert cycles $X^w$ by the following geometric observation of Bernstein--Gelfand--Gelfand \cite{BGG73} and Demazure \cite{Dem74}. Fix a reduced word decomposition $w=s_{i_1}\cdots s_{i_k}$ of $w$ by adjacent transpositions $s_i=(i,i+1)\in S_n$, and denote $w_j=s_{i_1}\cdots s_{i_j}$. Then there is a commutative diagram
\begin{center}
    \begin{tikzcd}
        \BS(i_1,\ldots,i_k)\ar[r]\ar[d]&\BS(i_1,\ldots,i_{k-1})\ar[r]\ar[d]&\cdots \ar[r]& \BS()=\{\mathrm{pt}\}\ar[d]\\
        X^{w_k}\ar[r,dashed]&X^{w_{k-1}}\ar[r,dashed]&\cdots \ar[r,dashed]& X^{\idem}=\{\mathrm{pt}\}
    \end{tikzcd}
\end{center}
where the bottom row are certain rational maps which are generically $\PP^1$-bundles, and the top row is a resolution of this sequence of rational maps by an iterated $\PP^1$-bundle $\BS(i_1,\ldots,i_k)$ known as the \emph{Bott--Samelson resolution} \cite{BS58, Han73, Dem74}. 
The divided difference operators $\partial_1,\partial_2,\ldots,\partial_{n-1}\in \End(\poly_n)$ defined by \begin{align}\partial_if=\frac{f-f(x_1,\ldots,x_{i-1},x_{i+1},x_i,x_{i+2},\ldots,x_n)}{x_i-x_{i+1}}\end{align}
interact in an important way with symmetric polynomials. For example 
$$f\in \sym{n}\Longleftrightarrow\partial_1f=\cdots \partial_{n-1}f=0,$$
and they descend to $H^\bullet(\fl{n})=\coinv{n}$ since  for $g\in \sym{n}$ we have $\partial_i(gh)=g\partial_i(h)$. The Bott--Samelson resolution has the property that $\deg_{\BS(i_1,\ldots,i_j)}f=\deg_{\BS(i_1,\ldots,i_{j-1})}\partial_{i_j}f$ for $f\in H^\bullet(\fl{n})$. These considerations then imply
$$\deg_{X^w}f=\deg_{\BS(i_1,\ldots,i_k)}f=\deg_{\BS(i_1,\ldots,i_{k-1})}\partial_{i_{k}}f=\cdots = \deg_{\BS()}\partial_{i_1}\cdots \partial_{i_k}f=\ct\partial_w f,$$
where $\ct g=g(0,0,\ldots)$ is the constant term operator and $\partial_w=\partial_{i_1}\cdots \partial_{i_k}$ is the composite operator. For two choices of reduced word $w=s_{i_1}\cdots s_{i_k}=s_{i_1'}\cdots s_{i_k'}$, the fact that $\deg_{X^{w}}$ can be computed with the Bott--Samelson resolution associated to either sequence gives a geometric interpretation of the identity $\ct\partial_{i_1}\cdots \partial_{i_k}=\ct\partial_{i_1'}\cdots \partial_{i_k'}$. Algebraically, the divided difference operators satisfy the nil-Hecke relations $\partial_i^2=0$, $\partial_i\partial_{i+1}\partial_i=\partial_{i+1}\partial_i\partial_{i+1}$, and $\partial_i\partial_j=\partial_j\partial_i$ for $|i-j|\ge 2$, and the identity can also be shown using these local relations via Coxeter word combinatorics. 

 The Schubert polynomials of Lascoux--Sch\"utzenberger \cite{LS82} are a family of homogenous polynomials $\{\schub{w}:w\in S_{n}\}\subset \poly_n$ with $\schub{\idem}=1$ that satisfy
$$\partial_i\schub{w}=\begin{cases}\schub{ws_i}&i\in \des{w}\\0&\text{otherwise.}\end{cases}$$

They also satisfy $\ct \partial_v\schub{w}=\delta_{v,w}$, so they descend to the Kronecker dual basis to the cycles $\{X^w\suchthat w\in S_n\}$ in $H^\bullet(\fl{n})$ under the Poincar\'e pairing. In fact, as the fundamental classes of opposite Schubert cycles $X_w$ are also a  dual basis, we have $\schub{w}=[X_w]$ in $H^\bullet(\fl{n})$.

\subsection{The quasisymmetric operations $\rope{i}$, and $\tope{i}$}

As in \cite{NST_a}, we define the \emph{Bergeron--Sottile maps} $\rope{1},\ldots,\rope{n}\in \End(\poly_n)$ and the \emph{quasisymmetric divided difference} $\tope{1},\ldots,\tope{n-1}\in \End(\poly_n)$ by
\begin{align}
    \label{eqn:ropedef}\rope{i}f&\coloneqq f(x_1,\ldots,x_{i-1},0,x_i,x_{i+1},\ldots,x_{n-1})\\
    \label{eqn:topedef}\tope{i}f&\coloneqq\rope{i}\partial_if=\rope{i+1}\partial_if=\frac{1}{x_i}(\rope{i+1}f-\rope{i}f).
\end{align}
These maps were shown to interact in an important way with quasisymmetric polynomials -- for example, it was shown in \cite[Theorems 2.6 and 2.10]{NST_a} that for $f\in \poly_n$ we have equivalences 
$$f\in \qsym{n}\Longleftrightarrow \rope{1}f=\cdots = \rope{n}f\Longleftrightarrow \tope{1}f=\cdots=\tope{n-1}f=0.$$

Unlike the $\partial_i$ operators, the $\rope{i}$ and $\tope{i}$ operators naturally decrease the number of variables of a polynomial. For $X_n\in \{\rope{1},\tope{1},\rope{2},\tope{2},\ldots,\rope{n-1},\tope{n-1},\rope{n}\}$ we have $X_n(\poly_n)\subset \poly_{n-1},$. These descend to maps $X_n:\coinv{n}\to \coinv{n-1}$ and $X_n:\qscoinv{n}\to \qscoinv{n-1}$ since for $g\in \qsym{n}$ (or $g\in \sym{n}$) we have $X_n(gh)=(\rope{1}g)X_n(h)$ and $\rope{1}$ preserves symmetry and quasisymmetry.

As it is natural to compose the $\partial_i$ operators in an order which makes a reduced word, this variable decreasing property of $X_i$ leads to a natural class of composites of $\tope{j}$ and $\rope{j}$ to consider. We denote by $\rtseq_n$ for the set of words $\rt=\xletter{1}\xletter{2}\cdots\xletter{n}$ with letters $$\xletter{i}\in \{\rletter{1},\tletter{1},\rletter{2},\tletter{2},\ldots,\rletter{i-1},\tletter{i-1},\rletter{i}\}.$$
Note the slight redundancy that $\xletter{1}=\rletter{1}$ necessarily; this will be useful from an algebraic viewpoint. Denote $\prodoperator{\rt}^n=X_1X_2\cdots X_n:\poly_n\to \ZZ$ for the composite linear functional under $\rletter{i}\mapsto \rope{i},\tletter{i}\mapsto\tope{i}$, which we may view also as a linear functional on $\coinv{n}$, or $\qscoinv{n}$. While $\partial_w$ and $\ct\partial_w$ are different operators, we have $\prodoperator{\rt}^n=\ct\prodoperator{\rt}^n$ for $\rt\in \rtseq_n$ because the codomain is $\ZZ$.

Just as a reduced word $w=s_{i_1}\cdots s_{i_k}$ describes a way of reducing a permutation to the identity by applying adjacent transpositions, we will show that a sequence $\rt=\xletter{1}\cdots \xletter{n}\in \rtseq_n$ can be viewed as describing a way of trimming a ``plane nested binary forest'' $\wh{F}=\wh{F}(\rt)\in \nsuppfor{n}$ (\Cref{defn:planebinarynestedforest}) down to the empty forest via certain elementary transformations. We let $$\Trim{\wh{F}}\coloneqq \{\rt\in \rtseq_n \suchthat \wh{F}(\rt)=\wh{F}\}$$ denote the set of all such sequences which trim the forest $\wh{F}$ (\Cref{defn:TrimForest}), analogous to the set of reduced words of $w\in S_n$.

Using commutation relations between $\tope{i}$ and $\rope{j}$ we will show that the operators $\prodoperator{\rt}^n$ for $\rt\in \Trim{\wh{F}}$ are all equal to a common operator $\prodoperator{\wh{F}}^n$, with $\wh{F}$ representing a sequencing order for certain function compositions. This generalizes an analogous result in \cite{NST_a} describing composites $\tope{F}=\tope{i_1}\tope{i_2}\cdots \tope{i_k}$ in terms of certain plane binary forests $F\in \suppfor{n}$.
\begin{eg}
\label{eg:small_trim_example}
In our theory, we have $\rt=\rletter{1}\tletter{1}\tletter{2}\tletter{1}\rletter{2}$ and $\rt'=\rletter{1}\tletter{1}\tletter{1}\rletter{2}\tletter{4}$ both belong to $\Trim{\wh{F}}$ with $\wh{F}$ depicted in the figure below. One can check directly for $f\in \ZZ[x_1,x_2,x_3,x_4,x_5]$ that
    $$\rope{1}\tope{1}\tope{2}\tope{1}\rope{2}f=\rope{1}\tope{1}\tope{1}\rope{2}\tope{4}f,$$
    i.e. $\prodoperator{\rt}^5=\prodoperator{\rt'}^5$,
    which follows by applying the commutation relation $\tope{2}\tope{1}=\tope{1}\tope{3}$ followed by the commutation relation $\tope{3}\rope{2}=\rope{2}\tope{4}$.
\begin{center}
    \includegraphics[height=2cm]{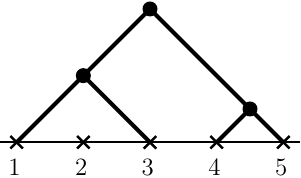}
\end{center}
\end{eg}

\subsection{A quasisymmetric Schubert cycle}

Write $|\rt|_{\tletter{}}$ for the number of $\xletter{i}$ which are equal to $\tletter{j}$ for some $j$. To each element $\rt\in \rtseq_n$ we will associate a pair of permutations $u(\rt),v(\rt)$ in $S_n$ such that $u(\rt)\le v(\rt)$ in Bruhat order and $|\rt|_{\tletter{}}=\ell(v(\rt))-\ell(u(\rt))$ (\Cref{defn:uvfromrt}). We define the $|\rt|_{\tletter{}}$-dimensional ``$\rt$-Richardson variety'' by
$$X(\rt)\coloneqq X^{v(\rt)}_{u(\rt)}\subset \fl{n}.$$
The maximal dimension $X(\rt)$ come from sequences $\rt=\rletter{1}\tletter{i_1}\cdots \tletter{i_{n-1}}\in \rtseq_n$, and give the smooth toric Richardson varieties considered in \cite{HHMP,lian2023hhmp}:
$$\{X^{uc_n}_u\suchthat u\in S_n\text{ and }u(n)=n\}$$
where $c_n=n12\cdots (n-1)$ is the backwards long cycle in $S_n$. The remaining $X(\rt')$ are the torus-orbit closures contained in one of these maximal $X(\rt)$.

The following is the quasisymmetric analogue of the BGG/Demazure geometric interpretation of divided differences.
\begin{thm}[{\Cref{thm:mainP1degree}}]
\label{thm:introP1degree}
    If $\rt=\rt'\xletter{}\in \rtseq_{n}$ and $f\in H^\bullet(\fl{n})$ then the following are true.
    \begin{enumerate}
        \item If $\xletter{}=\rletter{i}$ then $X(\rt)\cong X(\rt')$ and
        $\deg_{X(\rt)}f=\deg_{X(\rt')}\rope{i}f.$
        \item If $\xletter{}=\tletter{i}$ then there is a $\PP^1$-bundle $X(\rt)\to X(\rt')$ and 
        $\deg_{X(\rt)}f=\deg_{X(\rt')}\tope{i}f.$   \end{enumerate}
        In particular, we have
        $\deg_{X(\rt)}f=\prodoperator{\rt}^n f.$
\end{thm}
To each $\wh{F}\in \nfor_n$, we will show there is a torus-orbit closure $X(\widehat{F})\subset \fl{n}$ which has the property that for any $\rt\in \Trim{\wh{F}}$ we have
$$X(\wh{F})=u(\rt)^{-1}\cdot X(\rt).$$ We note that $X(\wh{F})$ need not by a toric Richardson variety.
 Let $\pi_\rt:X(\rt)\to X(\wh{F})$ be the isomorphism induced by multiplication by $u(\rt)^{-1}$. Letting $\rt=\xletter{1}\cdots \xletter{n}\in \rtseq_n$, if we define $\rt_i=\xletter{1}\cdots \xletter{i}\in \rtseq_i$ and $\wh{F}_i=\wh{F}(\rt_i)$, then there is a commutative diagram
\begin{center}
    \begin{tikzcd}        X(\rt_n)\ar[r]\ar[d,"\pi_{\rt_n}"]&X(\rt_{n-1})\ar[r]\ar[d,"\pi_{\rt_{n-1}}"]&\cdots \ar[r]& X(\rt_0)=\{\mathrm{pt}\}\ar[d,"\pi_{\rt_{0}}"]\\
        X(\wh{F}_n)\ar[r]&X(\wh{F}_{n-1})\ar[r]&\cdots \ar[r]& X(\wh{F}_0)=\{\mathrm{pt}\}
    \end{tikzcd}
\end{center}
which for $f\in H^\bullet(\fl{n})$ lets us compute
$$\deg_{X(\wh{F})}f=\deg_{X(\rt_n)}f=\deg_{X(\rt_{n-1})}X_{n}f=\cdots = \deg_{X(\rt_0)}X_1\cdots X_n f=\prodoperator{\rt}^nf.$$
For $\rt,\rt'\in \Trim{\wh{F}}$ this gives a geometric interpretation of the equality $\prodoperator{\rt}^nf=\prodoperator{\rt'}^nf$.

The analogies to the classical theory are therefore the following.
\begin{enumerate}
    \item The varieties $X(\wh{F})\subset \fl{n}$ correspond to the Schubert varieties $X^w\subset \fl{n}$.
    \item $\pi_{\rt}:X(\rt)\cong X(\wh{F})$ for $\rt\in \Trim{\wh{F}}$ corresponds to a Bott--Samelson resolutions of $X^w$.
\end{enumerate}

We note that in the classical case, the Bott--Samelson resolution rarely maps isomorphically onto the corresponding Schubert variety.
\begin{eg}
    Consider $\rt=\rletter{1}\tletter{1}\tletter{2}\tletter{1}\rletter{2}$ and $\rt'=\rletter{1}\tletter{1}\tletter{1}\rletter{2}\tletter{4}$, which both belong to $\Trim{\wh{F}}$ as in~\Cref{eg:small_trim_example}. Writing permutations in one line notation we have the toric varieties $X(\rt)=X^{51243}_{21435}$ and $X(\rt')=X^{52341}_{32415}$, and $X(\rt')=23145\cdot X(\rt)$.
\end{eg}

\subsection{Duality with forest polynomials}
 We cannot find a Kronecker dual basis to the $X(\wh{F})$ because there are linear relations between the fundamental classes, corresponding to nontrivial relations between $\prodoperator{\wh{F}}^n$ functionals arising from the identity
$$\tope{i}\rope{i+1}=\rope{i+1}\tope{i}+\rope{i}\tope{i+1}.$$
It turns out that we obtain nice duality statements if we restrict to the subset of indexed forests $\suppfor{n}\subset \nsuppfor{n}$ where nesting does not occur.

To see this, we note that \Cref{thm:introP1degree} has particular significance in terms of the work from \cite{NST_a} on the interaction between the homogenous family of forest polynomials $\{\forestpoly{F}:F\in \indexedforests\}$ and the operators $\tope{i}$. Similarly to the way that Schubert polynomials interact with the $\partial_i$ operators, the forest polynomials satisfy 
$$\tope{i}\forestpoly{F}=\begin{cases}\forestpoly{F/i}&i\in \qdes{F}\\0&\text{otherwise,}\end{cases}$$
where $\qdes{F}$ is a certain subset of the leaf labels of $F$ and $F/i$ is a certain ``trimmed'' forest (see \Cref{defn:forpolyindirect}). Just as the Schubert polynomials $\{\schub{w}\suchthat w\in S_n\}$ descend to a basis of $\coinv{n}$, the forest polynomials $\{\forestpoly{F}\suchthat F\in \suppfor{n}\}$ descend to a basis of $\qscoinv{n}$.

Furthermore, the distinct composites $\tope{G}$ of the $\tope{i}$ operators are indexed by $G\in \indexedforests$ and $\ct \tope{F}\forestpoly{G}=\delta_{F,G}$. It is this observation that allows us to  identify a subfamily of the $X(\wh{F})$ as dual to the forest polynomials.

\begin{maintheorem}
\label{maintheorem:QsymSchubertcycles}For $F\in \suppfor{n}$ an indexed forest, we have
$\deg_{X(F)}f=\ct\tope{F}f,$ and there is a duality $$\deg_{X(F)}\forestpoly{G}=\delta_{F,G}.$$ Furthermore, for any $\rt\in \rtseq_n$, the fundamental class $[X(\rt)]\in H^\bullet(\fl{n})$ is uniquely a nonnegative linear combination of the fundamental classes $\{[X(F)]\suchthat F\in \suppfor{n}\}$.
\end{maintheorem}

Consequently, the subset $\{X(F)\suchthat F\in \suppfor{n}\}$ of the quasisymmetric Schubert cycles are to forest polynomials $\{\forestpoly{F}\suchthat F\in \suppfor{n}\}$ as the Schubert cycles $\{X^w\suchthat w\in S_n\}$ are to Schubert polynomials $\{\schub{w}\suchthat w\in S_n\}$.

\subsection{The $\rt$-flag variety}
We now describe the quasisymmetric analogue of the flag variety. 
The toric Richardson varieties $X(\rt)$ fit together into a complex of smooth toric varieties
$$\hhmp_n=\bigcup_{\rt\in \rtseq_n}X(\rt).$$
we call the $\rt$-flag variety.
As $\hhmp_n$ is composed of isomorphic copies of $X(\wh{F})$, one for each $\rt\in \Trim{\wh{F}}$, we call $\hhmp_n$ the $\rt$-flag variety.
It can also be defined recursively as flags satisfying certain simple conditions (see \Cref{thm:recursivehhmp}).
This complex of toric varieties was considered before in \cite{HHMP}. In \cite{lian2023hhmp} it was shown that $\hhmp$ arises as a toric degeneration of a general $T$-orbit closure in $\fl{n}$ (the permutahedral variety).
The moment polytopes of the top-dimensional $X(\rt)$ give a subdivision of the permutahedron $$\operatorname{Perm}_{n-1}=\operatorname{conv}\{\sigma\cdot (n,n-1,\ldots,1)\suchthat \sigma\in S_n\}$$ into $(n-1)!$ combinatorial cubes \cite{HHMP}, which we call the $\hhmp$-subdivision, and which by what we have said earlier has the property that the lower dimensional faces are indexed by the moment polytopes of lower dimensional $X(\rt')$ with $\rt'\in \rtseq_n$.
 As we will see later, the facial structure of this subdivision (and hence the poset structure of the toric complex) is identified with the unit cube subdivision of $[1,2]\times [1,3]\times \cdots \times [1,n-1]$, where to a sequence $\xletter{1}\cdots\xletter{n}\in\rtseq_n$ we associate the cuboid $Y_2\times \cdots \times Y_{n}$ where $Y_i=\{j\}$ if $\xletter{i}=\rletter{j}$ and $Y_i=[j,j+1]$ if $\xletter{i}=\tletter{j}$.

\begin{figure}[!ht]
    \centering
    \includegraphics[width=0.8\textwidth]{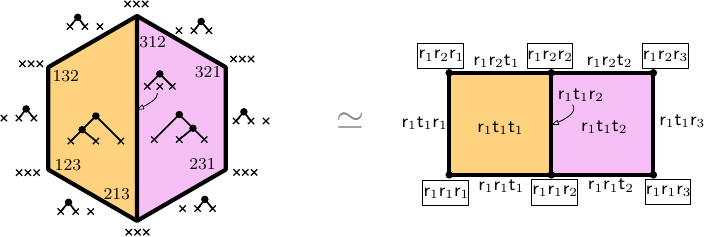}
    \caption{Unit cube subdivision for $n=3$ where we have indicated the face labelings by $\rtseq_3$ on the right, and nested forests on the left.}
    \label{fig:cube_subdivision}
\end{figure}

Not only does $\hhmp_n$ contain the quasisymmetric Schubert cycles, but our next main theorem shows its cohomology ring naturally contains $\qscoinv{n}$.

\begin{maintheorem}
\label{maintheorem:QFLcontains}
 Under the inclusion $\psi:\hhmp_n\to \fl{n}$, we have 
    $$\psi^*(H^\bullet(\fl{n}))\simeq\qscoinv{n}.$$
\end{maintheorem}

\subsection{Applications}

\subsubsection{Generalized Littlewood--Richardson coefficients}
The generalized LR coefficients $c^{v}_{u,w}$ are the structure coefficients for Schubert polynomial multiplication
$$\schub{u}\schub{w}=\sum c^v_{u,w}\schub{v}.$$
Geometrically these can be realized as $c^v_{u,w}=\deg_{X^v_u}\schub{w}$, so we can identify
$$c^{v(\rt)}_{u(\rt),w}=\prodoperator{\rt}^n \schub{w}.$$
We will give a combinatorially nonnegative interpretation to these generalized LR coefficients for all $\rt$. 
By using well-chosen sequences $\rt$ appearing in \cite{NST_a, NST_2}, we arrive at the following results.

\begin{thm}[{\Cref{thm:LR}}]
\label{maintheorem:LR}
    Let $w\in S_n$. The coefficients of
    \begin{enumerate}
        \item A monomial $x_1^{c_1}\cdots x_n^{c_n}$ in the monomial expansion of $\schub{w}$
        \item A slide polynomial coefficient in the slide polynomial expansion $\schub{w}=\sum a_{\textbf{i}}\slide{\textbf{i}}$
        \item A forest polynomial coefficient in the $m$-forest polynomial expansion $\schub{w}=\sum a_F\forestpoly[m]{F}$
    \end{enumerate}
    are all generalized LR coefficients $c^{v(\rt)}_{u(\rt),w}$ for some $\rt\in \rtseq_{N}$ with $N$ possibly larger than $n$, with an explicit combinatorially nonnegative rule for computing them.
\end{thm}
In \cite{LSS06} it was shown that the coefficient of a monomial arises in a different way as a generalized LR coefficient, and it would be interesting to compare these results. 
The second and third results are new, and as we will see also generalize to the $m$-forest polynomials of \cite{NST_a} and the $m$-slide polynomials of \cite{NST_2}.

\subsubsection{Divided symmetrization}
Harada--Horiguchi--Masuda--Park \cite{HHMP} and Lian \cite{lian2023hhmp} showed that for a general $x\in \fl{n}$ we have the equality \begin{align}
\label{eqn:Txissum}[\overline{T\cdot x}]=[\hhmp_n].\end{align}
On the other hand, by the work of Anderson--Tymoczko \cite{And10} we have 
$$\deg_{\overline{T\cdot x}}f(x_1,\ldots,x_n)=\sum_{\sigma\in S_n}\frac{f(x_{\sigma(1)},\ldots,x_{\sigma(n)})}{(x_{\sigma(1)}-x_{\sigma(2)})\cdots (x_{\sigma(n-1)}-x_{\sigma(n)})},$$
realizing the degree map of classes from $H^\bullet(\fl{n})$ on $\overline{T\cdot x}$ as ``divided symmetrization'' (DS henceforth), first studied by Postnikov \cite[\S 3]{Pos09}. 
We can combine these results with \Cref{maintheorem:QsymSchubertcycles} to obtain a new factorized expression for DS. 

\begin{thm}
    For $f\in \poly_n$ homogenous of degree $n-1$, we have
    $$\sum_{\sigma\in S_n}\frac{f(x_{\sigma(1)},\ldots,x_{\sigma(n)})}{(x_{\sigma(1)}-x_{\sigma(2)})\cdots (x_{\sigma(n-1)}-x_{\sigma(n)})}=\tope{1}(\tope{1}+\tope{2})\cdots (\tope{1}+\cdots +\tope{n-1})f.$$
\end{thm}

Expanding out the expression on the right fully, each term corresponds to the degree on a distinct maximal toric Richardson variety in $\hhmp_n$. 
This factorized form of DS explains many of its interesting properties that have been observed in the past, particularly the interaction with quasisymmetric polynomials, and we show that the $q$-analogue $\tope{1}(\tope{1}+q\tope{2})\cdots (\tope{1}+q\tope{2}+\cdots+q^{n-1}\tope{n-1})$ recovers the $q$-divided symmetrization $\langle f\rangle_n^q$ considered in \cite{NT21}.

One corollary of the DS identity is that for any $n$-variable Schubert polynomial $\schub{w}(x_1,\ldots,x_n)$, the DS $\langle \schub{w}\rangle_n$ is a polynomial all of whose coefficients are nonnegative. Through computer experimentation, we make the following conjecture.
\begin{conj}[{\Cref{conj:HL_positivity}}]
    The $q$-divided symmetrization $\langle \schub{w}\rangle_n^q$ of a Schubert polynomial $\schub{w}(x_1,\ldots,x_n)$ is Hall-Littlewood P-positive.
\end{conj}
\subsubsection{A formula of Gessel}
Finally, we use our theory to give a geometric explanation for a formula of Gessel {\cite[Theorem 3]{Ges84}} which, given a symmetric polynomial $f(x_1,\ldots,x_n)$, determines the coefficients of the expansion
$$f(x_1,\ldots,x_n)=\sum_{k=1}^n\sum_{i_k,\ldots,i_n\ge 1}a_{i_k,\ldots,i_n}\slide{i_k,\ldots,i_n}(x_1,\ldots,x_n)$$
where $\slide{i_k,\ldots,i_n}(x_1,\ldots,x_n)$ is the fundamental quasisymmetric polynomial whose reverse lexicographic leading term is $x_k^{i_k}\cdots x_n^{i_n}$. 
This formula is in terms of the Hall inner product with skew Schur polynomials associated to ribbon shapes. 
We geometrically interpret it as a degree map on Grassmannian Richardson varieties associated to these shapes, and show that a certain subset of our $X(\rt)$ project onto these Grassmannian Richardson varieties. 
Using our theory of degree maps on $X(\rt)$ shows that Gessel's formula follows directly follows from a more general formula
$$f(x_1,\ldots,x_n)=\sum_{k=1}^n\sum_{i_k,\ldots,i_n\ge 1}(\ct\mathsf{T}_{k}^{i_k}\cdots \mathsf{T}_{n}^{i_n}f)\slide{i_k,\ldots,i_n}(x_1,\ldots,x_n)$$
proved in \cite{NST_a} that works for $f\in \qsym{n}$ an arbitrary quasisymmetric polynomial.

\section{Composites of \texorpdfstring{$\rope{}$}{R} and  \texorpdfstring{$\tope{}$}{T} via nested forests}
\label{sec:compositesTR}


We first recall the results of~\cite{NST_a} about the combinatorics the Thompson monoid, its faithful polynomial representation given by $i\mapsto \tope{i}$, and the forest polynomials $\forestpoly{F}$.
We then generalize the combinatorics to what we call the augmented Thompson monoid and (marked) nested forests. Finally, we fix $n$ and  restrict the previous constructions to obtain operators $\prodoperator{\wh{F}}^n$ on any of the spaces $\poly_n,\coinv{n},\qscoinv{n}$ that depend on a certain forest $\wh{F}\in\nsuppfor{n}$ in the collection of \emph{nested forests supported on $\{1,\ldots,n\}$}.

To avoid disrupting the exposition and keeping our overarching goal in mind we relegate the proofs of combinatorial results in this section to Appendix~\ref{sec:nested_forests_extra}.

\subsection{Composites of $\tope{}$ via indexed forests, and forest polynomials}
\label{sec:recallforests}
In this section we recall the relationship between the $\tope{}$ operators, plane binary forests, and forest polynomials from \cite{NST_a}. Let $\poly=\ZZ[x_1,x_2,\ldots]=\bigcup_n\poly_n$. The \emph{quasisymmetric divided differences} $\tope{1},\tope{2},\ldots\in \End(\poly)$ defined as in \eqref{eqn:topedef} satisfy the elementary commutation relations
\begin{align}
    \label{eq:t_t_com}\tope{i}\tope{j}=\tope{j}\tope{i+1}\text{ for }i>j,
\end{align}
which imply that  $i\mapsto \tope{i}$ is a representation of the Thompson monoid
$$\Th=\langle 1,2,\ldots\suchthat i\cdot j=j\cdot (i+1)\text{ for }i>j\rangle.$$
It is a classical fact in the theory of Thompson groups \cite{BelkBrown05,CFPnotes96,DehTes19,Sunic07} that $\Th$ is isomorphic to a monoid $\indexedforests$ of forests of plane binary trees. Hence one can define operators $\tope{F}$ for $F\in \indexedforests$. 

\begin{defn}
\label{def:indexedforest}
    A \emph{plane binary tree} is a binary tree where each internal (i.e. non-leaf) node $v$ has a left child $v_L$ and a right child $v_R$. 
    An \emph{indexed forest} $F$ is a sequence of plane binary trees $T_1,T_2,\ldots$ where $T_i=\times$, the trivial plane binary tree, for all but finitely many $i$.
    We denote the set of indexed  forests by $\indexedforests$.
\end{defn}

The leaves of a plane binary tree are naturally ordered, so there is a canonical identification of the leaves of $F\in \indexedforests$ with $\NN$. Denoting $\internal{F}$ for the internal nodes of $F$, the totality of nodes of $F$ may thus be identified with $\internal{F}\sqcup \NN$.

 \begin{defn}
 \label{defn:monoid_For}
     We define a monoid structure on $\indexedforests$ by taking the composition $F\cdot G$ to be the plane forest obtained by identifying the $i$'th leaf of $F$ with the $i$'th root of $G$.
 \end{defn}
 For $\wedge$ the unique binary tree with one internal node, we denote $$\underline{i}\coloneqq \underbrace{\times\cdots \times}_{i-1} \wedge \times \times \cdots .$$
 \begin{thm}[{\cite{NST_a}}]
 \label{thm:isomorphism_forests_thompson}
The map $i\mapsto \underline{i}$  induces a monoid isomorphism $\Th\to \indexedforests$.
\end{thm}
This result allows us to tacitly identify $i$ with $\underline{i}$ from now on. 
  Since $i\mapsto \tope{i}$ is a representation of $\Th$ on $\poly$, the composite operation $\tope{i_1}\cdots \tope{i_k}$ only depends on the forest $i_1\cdots i_k\in \indexedforests$.
\begin{defn}
      For $F\in \indexedforests$, we define $\tope{F}\coloneqq \tope{i_1}\cdots \tope{i_k}$ for any sequence $i_1,\ldots,i_k$ with $F=i_1\cdots i_k$. 
\end{defn}

An internal node of $F$ is \textit{terminal} if both of its children are leaves. The left leaves of terminal nodes form the set $\qdes{F}$. 
Equivalently, $\qdes{F}$ is the set of $i$ such that $F=G\cdot i$ for some (necessarily unique) forest $G\in \indexedforests$. For $i\in \qdes{F}$ we write $F/i\in \indexedforests$ for the unique forest with $F=(F/i)\cdot i$.

\begin{defn}[{\cite[Theorem 6.1]{NST_a}}]
\label{defn:forpolyindirect}
    The forest polynomials $\{\forestpoly{F}\suchthat F\in \indexedforests\}$  are the unique family of homogenous polynomials in $\poly$ with $\forestpoly{\emptyset}=1$ and $$\tope{i}\forestpoly{F}=\begin{cases}\forestpoly{F/i}&i\in \qdes{F}\\0&\text{otherwise.}\end{cases}$$
\end{defn}
These polynomials were introduced combinatorially by the first and third authors in \cite{NT23, NT_forest} but we shall have no need for the explicit description.
The recursive characterization in Definition~\ref{defn:forpolyindirect} then implies \cite[Corollary 6.6]{NST_a} that the forest polynomials are dual to the functionals $\ct\tope{G}$ in the sense that
$$\ct\tope{G}\,\forestpoly{F}=\delta_{G,F}.$$
As was shown in \cite[Corollary 6.7]{NST_a}, because the maps $\End(\poly)\to \ZZ$ given by $\Phi\mapsto \ct\Phi\forestpoly{F}$ separate the linear functionals $\tope{G}$, the map $\ZZ[\Th]\to \End(\poly)$ is a faithful representation (i.e. the $\tope{F}$ operators are $\ZZ$-linearly independent).

We define $\suppfor{n}$ to be the subset of $F\in \indexedforests$ where the leaves of all nontrivial trees lie in $\{1,\ldots,n\}$.

\begin{thm}
[{\cite[Proposition 6.8 and Theorem 9.7]{NST_a}}]
\label{thm:qsymvanishing}The forest polynomials are a $\ZZ$-basis for $\poly$, and every polynomial $f\in \poly$ can be uniquely written as
$$f=\sum_{F\in \indexedforests} a_F\forestpoly{F}.$$
with $a_F=\ct\tope{F}f$. Additionally, the forest polynomials $\{\forestpoly{F}\suchthat F\in \suppfor{n}\}$ lie in $\ZZ[x_1,\ldots,x_n]$ and descend to a $\ZZ$-basis of $\qscoinv{n}$. Finally, for $f\in \poly_n$ we have $f\in \qsymide{n}$ if and only if $\ct\tope{F}f=0$ for all $F\in \suppfor{n}$.
\end{thm}

\subsection{Composites of $\rope{}$ and $\tope{}$ via marked nested forests}
\label{sec:sequences_and_forests}

We now generalize the forest combinatorics of \Cref{sec:recallforests} to include the Bergeron--Sottile maps $\rope{1},\rope{2},\ldots\in \End(\poly)$ defined as in \eqref{eqn:ropedef}.  
These satisfy the further commutation relations
\begin{align}\label{eq:r_t_com_1}
\tope{i}\rope{j}&=\rope{j}\tope{i+1}\text{ for }i\ge j\\
    \label{eq:r_t_com_2}\rope{i}\tope{j}&=\tope{j}\rope{i+1}\text{ for }i>j\\
    \label{eq:r_t_com_3}\rope{i}\rope{j}&=\rope{j}\rope{i+1}\text{ for }i\ge j.
\end{align}
Together with~\eqref{eq:t_t_com} these relations therefore define the following abstract monoid.

\begin{defn}
\label{defn:augmented_thompson}
    The \emph{augmented Thompson monoid} $\wt{\Th}$ is the quotient of the free monoid on the alphabet $\{1,2,\ldots\}\sqcup \{1_{\circ},2_{\circ},\ldots\}$ by the relations
    \begin{align}
        i\cdot j&=j\cdot (i+1)\text{ if }i>j\\
        i\cdot j_{\circ}&=j_{\circ}\cdot {(i+1)}\text{ if }i\ge j\\
        i_{\circ}\cdot j&=j\cdot {(i+1)}_{\circ}\text{ for }i>j\\
        i_{\circ}\cdot j_{\circ}&=j_{\circ}\cdot {(i+1)}_{\circ}\text{ for }i\ge j.
    \end{align}
    We define the polynomial representatiof the augmented Thompson monoid to be the representation $\poly$ given by $i\mapsto \tope{i}$ and $i_{\circ}\mapsto \rope{i}$.
\end{defn}

We define here the monoid $\mnfor$ of ``marked nested binary forests'', which extends the monoid of indexed forests defined in \Cref{sec:recallforests}. We start by defining nested binary forests, which are already an extension of indexed forests.\smallskip

A \textit{finite partition} $\pi$ of $\NN$ is a set partition of $\NN$ such that all blocks have finite size, and all but finitely many blocks are singletons; equivalently, $\{i\}$ is a block for all $i$ large enough. 
We say $\pi$ is \textit{noncrossing} if for any two blocks $B,B'\in \pi$  we never have $a<c<b<d$ with $a,b\in B$ and $c,d\in B'$. 

\begin{defn}\label{defn:planebinarynestedforest}A \emph{\textup{(}plane, binary\textup{)} nested forest} is a family of plane binary trees $\left(T_B\right)_{B\in \pi}$ where $\pi$ is a finite noncrossing partition of $\NN$, and each tree $T_B$ for $B\in \pi$ has $|B|$ leaves.
\end{defn}

We denote by $\nfor$ the set of nested forests. The subset $B$ in the definition is called the \textit{support} of $T_B$. An example of of nested forest is given in \Cref{fig:nfor_example}; the non-singleton supports are $\{2,3\}, \{4,5,9,11\}, \{6,7,8\}$ and $\{13,14,15\}$. 
If all supports of $F\in\nfor$ are intervals $\{a,a+1,\ldots,b\}$, then we retrieve the notion of indexed forests: these intervals are ordered from left to right, giving trees $T_1,T_2,\ldots$ as in \cite[Definition 3.1]{NST_a}. 

\begin{figure}[!ht]
    \centering
    \includegraphics[width=0.7\linewidth]{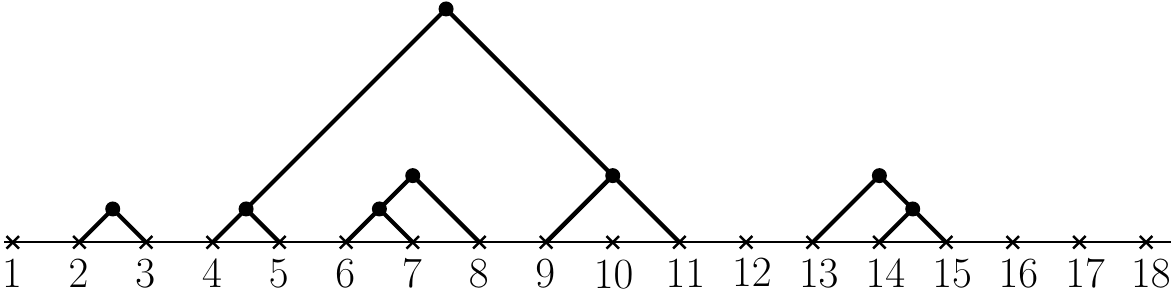}
    \caption{A nested forest.}
    \label{fig:nfor_example}
\end{figure}

This notion will mostly be sufficient for the rest of the article. 
In order to build a monoid isomorphic to the one generated by the $\tope{i}$ and $\rope{i}$ however, we need to enrich this structure with a notion of \textit{marks} on the roots of these trees.

This requires some extra notions on noncrossing partitions. Two blocks $B,B'$ in a noncrossing partition $\pi$ are either:
\begin{itemize} \item \textit{aligned}: the elements of $B$ are smaller than the elements of $B'$ (or the opposite), or \item \textit{nested}: there are two elements $a<b$ of $B$ such all elements of $B'$ are between $a$ and $b$.
\end{itemize}
The \textit{outer blocks} of $\pi$ are those that are not nested under any other block.  Outer blocks are pairwise aligned, and thus are totally ordered from left to right. We call a tree in a forest an outer (resp. nested) tree if its support is an outer (resp. nested) block.

\begin{defn} 
A \emph{marked nested forest} is a nested forest where a finite number of tree roots are marked, and these include the roots of all nested trees. 
We write $\mnfor$ for the set of all marked nested forests.
\end{defn}
\begin{figure}[!ht]
    \centering
    \includegraphics[width=0.7\linewidth]{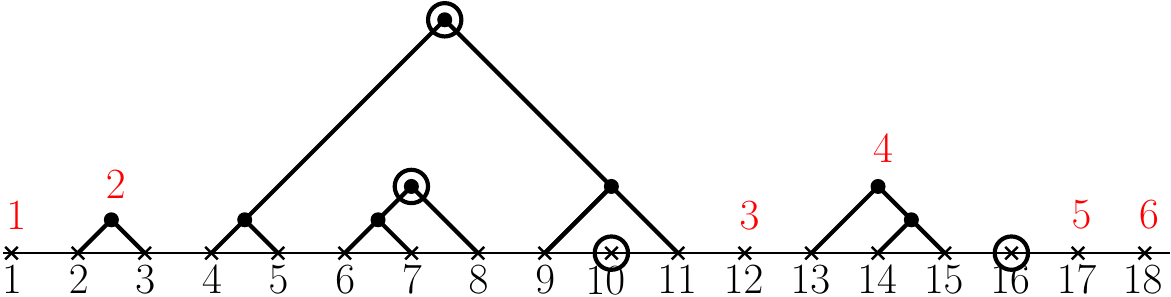}
    \caption{A marked nested forest.}
    \label{fig:mnfor_example}
\end{figure}

\Cref{fig:mnfor_example} depicts a marked nested forest with circles denoting the marking on the roots. 
Note that the roots of the trees with supports $\{6,7,8\}$ and $\{10\}$ are necessarily marked since the trees are nested, while the marks on all other roots can be picked arbitrarily.

For $\wt{F}\in \mnfor$ we write $\wh{F}\in \nfor$ for the nested forest obtained by forgetting the markings. 
As an example,  note that the nested forest in Figure~\ref{fig:nfor_example} is the $\wh{F}$ for the $\wt{F}$ in Figure~\ref{fig:mnfor_example}.

\begin{rem}[Indexed forests as marked nested forests]
\label{rem:indexed_and_marked}
   If a marked nested forest has no marks, then all blocks are outer blocks, and this entails that they are integer intervals $\{a,a+1,\ldots,b\}$. As we saw above this is naturally equivalent to the notion of indexed forests from \cite{NST_a}.
\end{rem}

Note that the \textit{unmarked} roots of (outer) trees are naturally ordered from left to right, and so we may canonically identify the unmarked roots by $\NN$ and talk about the ``$i$'th unmarked root''. We then have the following generalization of \cite[Definition 4.1]{NST_a}, recalled in \Cref{defn:monoid_For}.

\begin{defn}
Let $\wt{F},\wt{G}\in \mnfor$.  
Define a monoid structure on $\mnfor$ by letting $\wt{F}\cdot \wt{G}$ be obtained by identifying the $i$'th leaf of $\wt{F}$ with the $i$'th unmarked root of $\wt{G}$.
\end{defn}

We can describe $\mnfor$ recursively using this product. The 
``empty forest'' $\emptyset \coloneqq \times\times \times\cdots$ is made of unmarked trivial trees $\times$. For $\wedge$ the unmarked binary tree with one root and two leaves and $\otimes$ the trivial binary tree whose only node is marked, consider the elementary forests
\begin{align*}\ul{i}=&\underbrace{\times \cdots \times}_{i-1} \wedge \times \times \cdots\text{ and }
\ul{i_{\circ}}=\underbrace{\times \cdots \times}_{i-1} \otimes \times \times \cdots \end{align*}  Then, given $\wt{F}\in \mnfor,i\in \NN$, we obtain
    \begin{itemize}
        \item $\wt{F}\cdot \ul{i}$ by taking the $i$'th leaf and giving it two children.
        \item $\wt{F}\cdot \ul{i_{\circ}}$ by inserting a trivial marked tree $\otimes$ between the $(i-1)$'st and the $i$'th leaves of $\wt{F}$ (if $i=1$ then we add $\otimes$ before the first leaf of $\wt{F}$).
    \end{itemize}
\begin{thm}
\label{thm:presentation_mnfor}
    The map $i\mapsto \underline{i}, i_{\circ}\mapsto \underline{i_{\circ}}$ induces an isomorphism $\wt{\Th}\cong \mnfor$.
\end{thm}

The proof is in \Cref{sec:nested_forests_extra}. This result allows us to tacitly identify $i,i_{\circ}$ with $\underline{i},\underline{i_{\circ}}$ from now on.


\subsection{$\rtseq$ and composites $\prodoperator{\rt}$ of $\tope{}$ and $\rope{}$ operators}
We now study the polynomial representation of the augmented Thompson monoid on $\poly$ given by $i\mapsto \tope{i}$ and $i_{\circ}\mapsto \rope{i}$. The following definition will be useful for indexing sequences of $\rope{i}$ and $\tope{i}$ operators, and to avoid confusing such sequences with composite operators.
\begin{defn}
\label{defn:rt_sequences}
Let $\rtseq$ denote the set of sequences $\rt=\xletter{1}\cdots\xletter{k}$ with $k\geq 0$ and $$\xletter{i}\in \{\rletter{j},\tletter{j}\suchthat j\in \mathbb{Z}_{>0}\}.$$
We write $|\rt|=k$ and let $|\rt|_{\tletter{}}$ be the number of $\xletter{i}$ equal to $\tletter{j}$ for some $j$.
We write $\operatorname{up}(\rletter{i})=\rope{i}$, $\operatorname{up}(\tletter{i})=\tope{i}$, and the composite operator associated to the sequence $\rt\in \rtseq$ to be $$\prodoperator{\rt}\coloneqq\operatorname{up}(\xletter{1})\cdots \operatorname{up}(\xletter{n})\in \End(\poly).$$
\end{defn}

Viewing $\rtseq$ as the free monoid on the alphabet $\{\rletter{j},\tletter{j}\suchthat i\in\NN\}$, there is a natural monoid surjection $\rtseq\to \wt{\Th}\cong \mnfor$ sending $\tletter{i}\mapsto i$ and $\rletter{i}\mapsto i_{\circ}$, and hence to each $\rt\in \rtseq$ we can associate an element $\mnf{\rt}\in \mnfor$. Recursively, $\mnf{\emptyset}=\emptyset$ and for $\rt=\rt'\xletter{}$ we have $$\mnf{\rt}=\begin{cases}\mnf{\rt'}\cdot j&\text{if }\xletter{}=\tletter{j}\\\mnf{\rt'}\cdot j_{\circ}&\text{if }\xletter{}=\rletter{j}.\end{cases}$$
\begin{defn}\label{defn:TrimForest}
    For $\wt{F}\in \mnfor$, we define
    $$\Trim{\wt{F}}=\{\rt\in \rtseq \suchthat \wt{F}(\rt)=\wt{F}\}\subset \rtseq.$$
\end{defn}
    
\begin{prop}
Under the polynomial representation, the image of $\rt$ in $\wt{\Th}$ is $\prodoperator{\rt}$. In particular, 
for $\rt,\rt'\in \Trim{\wt{F}}$, we have $\prodoperator{\rt}=\prodoperator{\rt'}$.
\end{prop}
\begin{proof}
    The first part follows by definition of the polynomial representation on $\rt$, and the second follows from the fact that $\wt{F}(\rt)$ is by construction the image of $\rt$ under the composite $\rtseq\to \wt{\Th}\cong \mnfor$.
\end{proof}

\begin{defn}
    For $\wt{F}\in \mnfor$, we define $\prodoperator{\wt{F}}\coloneqq \prodoperator{\mnf{\rt}}$ for any $\rt\in \Trim{\wt{F}}$.
\end{defn}

\begin{rem}A more conceptual proof of the commutation relations \eqref{eq:t_t_com} and ~\eqref{eq:r_t_com_1}--\eqref{eq:r_t_com_3} follows from the observations in \cite{NST_a}:  if we define
\begin{enumerate}
    \item $\tope{}:\poly_1^{\otimes 2}\to \poly_1$ by $\tope{}{f(x,y)}=\frac{f(x,0)-f(0,x)}{x}$, and
    \item $\rope{}:\poly_1\to \poly_0$ by $\rope{}(f(x))=f(0)$,
\end{enumerate}
then under the identification $\poly=\poly_1^{\otimes \infty}$ we have $$\tope{i}=\idem^{\otimes i-1}\otimes \tope{}\otimes \idem^{\otimes \infty}\text{ and }\rope{i}=\idem^{\otimes i-1}\otimes \rope{}\otimes \idem^{\otimes \infty}.$$
The commutation relations are then the universal relations satisfied by compositions of the shifts of any fixed operators $\poly_1^{\otimes 2}\to \poly_1$  and $\poly_1\to \poly_0$, irrespective of the particular definitions. The marked nested forest $\wt{F}$ may be seen as a composition tree for two-to-one and one-to-zero operations acting on consecutive variables, with a marking indicating the application of a one-to-zero operation, and then $\prodoperator{\wt{F}}$ is the associated composition of these operations.
\end{rem}
The polynomial representation of the augmented Thompson monoid is not faithful as a linear representation $\ZZ[\wt{\Th}]\to \End(\poly)$, because of the presence of the nontrivial relation
\begin{align}
\label{eqn:tr_rel}\tope{i}\rope{i+1}=\rope{i}\tope{i+1}+\rope{i+1}\tope{i}
\end{align}
which will play an important role later on. 
As the first part of the next theorem shows however, it is however faithful as a monoid representation.
\begin{thm}
\label{thm:faithfulindexing}
    Let $\wt{F},\wt{F'}\in \mnfor$. Then
    \begin{enumerate}
        \item $\wt{F}=\wt{F'}$  if and only if $\prodoperator{\wt{F}}=\prodoperator{\wt{F'}}$, and
        \item $\wh{F}=\wh{F'}$ if and only if $\ct \prodoperator{\wt{F}}=\ct\prodoperator{\wt{F'}}$.
    \end{enumerate}
\end{thm}
\begin{cor}
    The distinct operators $\prodoperator{\rt}$ for $\rt\in \rtseq$ are indexed by $\widetilde{F}\in \mnfor$ and we can index the distinct operators $\ct\prodoperator{\rt}$ by $\wh{F}\in \nfor$.
\end{cor}

The proof is given in \Cref{sec:nested_forests_extra}.

\subsection{Fully supported forests and $\rtseq_n$}

To conclude this section, we classify the nested and marked nested forests arising from a subset $\rtseq_n\subset \rtseq$.
\begin{defn}
    We denote by $\rtseq_n\subset\rtseq$ for the subset of sequences $\rt=\xletter{1}\cdots\xletter{n}$ with the restriction that $\xletter{i}\in \{\rletter{1},\tletter{1},\rletter{2},\tletter{2},\ldots,\rletter{i-1},\tletter{i-1},\rletter{i}\}$.
\end{defn}
Note that for $\rt\in \rtseq_n$ we always have $\xletter{1}=\rletter{1}$.

Let us write
$$\prodoperator{\rt}^n:\poly_n\to \ZZ$$
for the restriction of $\prodoperator{\rt}$ to $\poly_n$. This functional descends to $\coinv{n}$ and $\qscoinv{n}$ by the discussion in \Cref{sec:Results}, and we write by abuse of notation $\prodoperator{\rt}^n:\coinv{n},\qscoinv{n}\to \ZZ$.

\begin{prop}
\label{prop:fullsupp}
Let $\rt\in \rtseq_n$, and write $\wt{F}=\mnf{\rt}$. Then \begin{enumerate}
    \item the $(n+1)$'st leaf of $\wt{F}$ onwards belong to trivial trees, and
    \item  all root nodes of the trees of $\wt{F}$ supported on $\{1,\ldots,n\}$ are marked.
\end{enumerate} Conversely, if $\wt{F}\in\mnfor$ satisfies these two properties then $\wt{F}=\mnf{\rt}$ for $\rt\in \rtseq_n$.
\end{prop}

This is proved in \Cref{sec:nested_forests_extra}.
It follows that for $\rt\in \rtseq_n$ we have $\nf{\rt}=\nf{\rt'}$ if and only if $\mnf{\rt}=\mnf{\rt'}$, so there is no need to remember marked roots.  Let $\nsuppfor{n}$ denote the set of nested forests where the $(n+1)$'st leaf onwards belong to trivial trees. 

\begin{defn}
    For $\wh{F}\in \nsuppfor{n}$, we define
    $$\Trim{\wh{F}}=\{\rt\in \rtseq_n\suchthat \wh{F}(\rt)=\wh{F}\}\subset \rtseq_n,$$
    and we define $\prodoperator{\wh{F}}^n=\prodoperator{\rt}^n$
    for any $\rt\in \Trim{\wh{F}}$.
    As with $\prodoperator{\rt}^n$, by abuse of notation, we also denote by $\prodoperator{\wh{F}}^n$ the induced maps $\coinv{n}\to \ZZ$ and $\qscoinv{n}\to \ZZ$.
\end{defn}
This notion $\prodoperator{\wh{F}}^n$ naturally encompasses the operations $\ct\tope{F}$ used to extract the coefficients of forest polynomials as discussed in \Cref{sec:recallforests}.
\begin{cor}
\label{cor:ev0TF}
    For $F\in \suppfor{n}$ we have $\ct\tope{F}=\prodoperator{F}^n$.
\end{cor}
\begin{proof}
    Suppose $F=i_1\cdots i_k$.
    For $f\in \poly_n$, we have $\tope{F}\in \poly_{n-|F|}$, and we can write $\ct\tope{F}=\prodoperator{\rt}^n$ where $\rt=\rletter{1}^{n-|F|}\tletter{i_1}\cdots \tletter{i_k}$. The nested forest associated to $\rt$ is  $F$, so $\prodoperator{\rt}=\prodoperator{F}^n$.
\end{proof}

We conclude by giving a visual depiction of $\rt\in \Trim{\wh{F}}$. 
For a sequence $\rt=\xletter{1}\cdots\xletter{n}\in\rtseq_n$, define the \textit{trimming diagram} of $\rt$ as follows: replace a letter $\rletter{i}$ (resp. $\tletter{i}$) in position $j$ in $\rt$  by the elementary diagram on the left (resp. right) below, and concatenate such diagrams from top to bottom. 
  \begin{center}
      \includegraphics[scale=0.60]{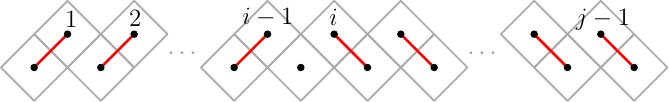}\qquad\quad
      \includegraphics[scale=0.60]{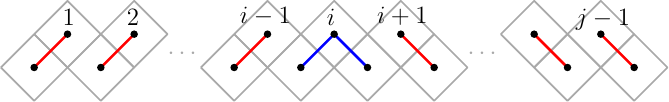}
  \end{center}

We note that the initial $\rletter{1}$ of $\rt$ does not contribute any edges. 
A trimming diagram for a particular $\rt\in \rtseq_{13}$ is given on the left in \Cref{fig:Trimming_diagram_and_forest}.
The highlighted elementary diagram corresponds to the letter in the fifth position being a $\tletter{2}$.
It is easy to see that the nested forest $\wh{F}$ such $\rt\in \Trim{\wh{F}}$ is the blue forest obtained after contracting all red edges. 
This is illustrated on the right.
  
    \begin{figure}[!ht]
    \centering
    \includegraphics[width=0.8\textwidth]{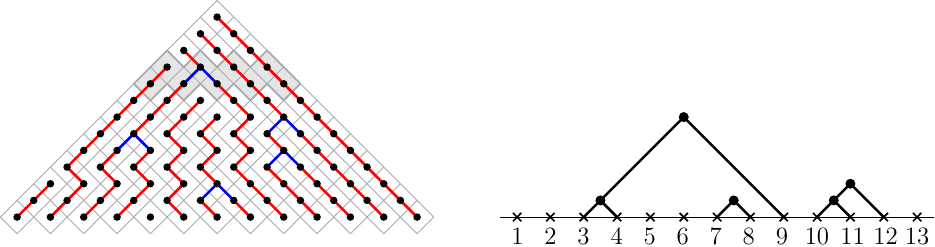}
    \caption{Trimming diagram for $\rletter{1}^4\tletter{2}\rletter{3}\rletter{4}\tletter{6}\tletter{2}\tletter{7}\rletter{1}\tletter{6}\rletter{5}\in \rtseq_{13}$ and the associated nested forest.
    \label{fig:Trimming_diagram_and_forest}}
\end{figure}
In \Cref{sec:Moment} we will show that removing the blue edges from the trimming diagram yields a GZ diagram whose associated polytope is the moment polytope for the $\rt$-Richardson varieties $X(\rt)$ we will be discussing in subsequent sections.

\section{Geometric realizations of \texorpdfstring{$\rope{}$}{R} and \texorpdfstring{$\tope{}$}{T}}
\label{sec:GeomRT}

Denote the tautological flag subbundle on $\fl{n}$ by $\{\bm{0}\}\subset \mathcal{V}_1\subset \mathcal{V}_2\subset \cdots \subset \mathcal{V}_n=\ul{\CC}^n$.
Consider the projection map $$\pi_i:\fl{n}\to \fl{1,\ldots,i-1,i+1,\ldots,n}$$
forgetting the $i$'th subspace, which realizes $\fl{n}$ as the $\PP^1$-bundle given $\PP(\mathcal{V}_{i+1}/\mathcal{V}_{i-1})$ on the partial flag variety. The following folklore facts describe how pullback and pushforward along $\pi_i$ interacts with cohomology.
\begin{fact}
\label{fact:pifacts}\leavevmode
\begin{enumerate}
    \item Pullback  $(\pi_i)^*$ induces an injection $H^\bullet(\fl{1,\ldots,i-1,i+1,\ldots,n})\hookrightarrow H^\bullet(\fl{n})=\coinv{n}$ whose image is the subring generated by $x_1,\ldots,x_{i-1},x_i+x_{i+1},x_ix_{i+1},x_{i+2},\ldots,x_n$, i.e. the subring generated by polynomials symmetric in $x_i,x_{i+1}$,\item Pushforward $(\pi_i)_*:H^\bullet(\fl{n})\to H^\bullet(\fl{1,\ldots,i-1,i+1,\ldots,n})$ is given by $f\mapsto \partial_if$.
\end{enumerate}
\end{fact}
We recall the thrust of the BGG computation in the following geometric fact.
\begin{fact}
\label{fact:BGGP1}
    If $w\in S_n$ and $i\not\in \des{w}$, then $X^{ws_i}\to \pi_i(X^{ws_i})$ is birational onto its image, and $X^w=\pi_i^{-1}\pi_i(X^{ws_i})$ is a $\PP^1$-bundle over $\pi_i(X^w)=\pi_i(X^{ws_i})$.
\end{fact}
From this fact we see that
$$\deg_{X^{w}}f(x_1,\ldots,x_n)=\deg_{\pi_i(X^{w})}\partial_if=\deg_{X^{ws_i}}\partial_i f.$$

In this section we carry out this computation with certain Richardson varieties that will later allow us to geometrically interpret the composite operators $\prodoperator{\rt}^n$ for $\rt\in \rtseq_n$.
For $1\le i \le n$, define the map $\ins_i:S_{n-1}\to S_n$ on permutations by
$$\ins_i(w)(j)=\begin{cases}w(j)+1&j<i\\1&j=i\\
w(j-1)+1&j>i.\end{cases}$$
Concretely this inserts a $1$ in the $i$'th available position in the one-line notation for $w$ and increases all other values by $1$, so e.g. $\ins_3(143652)=2514763$.

\subsection{A geometric realization of $\rope{}$}

The geometric description of the Bergeron--Sottile maps $\rope{i}$ comes from the seminal paper \cite{BS02} of the Pieri rule for Schubert varieties, in particular the map $\Psi_i$ -- we make no claims of originality.
    For $1\le i \le n$, let $\Psi_i:\fl{n-1}\to \fl{n}$ denote the map
    $$\left(\Psi_i(\mathcal{V})\right)_j=\begin{cases}\{0\}\oplus \mathcal{V}_j&\text{if }j<i\\
    \ul{\CC}\oplus \mathcal{V}_{j-1}&\text{if }j\geq  i.
    \end{cases}$$

\begin{thm}
\label{thm:ropepullback}
For $f\in \poly_n$ we have
    $$\Psi_i^*f=\rope{i}f.$$ 
\end{thm}
\begin{proof}
    We have
    $$\Psi_i^*(\mathcal{V}_j)=\begin{cases}\mathcal{V}_j&\text{if }j<i\\
    \ul{\CC}\oplus \mathcal{V}_{j-1}&\text{if }j\ge i.\end{cases}$$
Therefore
$$\Psi_i^*(x_1+\cdots+x_j)=c_1(\Psi_i^*(\mathcal{V}_j)^{\vee})=c_1((\mathcal{V}_{j-\delta_{j\geq i}})^{\vee})=x_1+\cdots+x_{j-\delta_{j\geq i}}=\rope{i}(x_1+\cdots+x_j).$$
Since $\Psi_i^*$ and $\rope{i}$ agree on the generators $x_1,\ldots,x_{n}$ of $H^\bullet(\fl{n})$, they are equal.
\end{proof}

\begin{thm}
\label{thm:geometricrope}
For $u\le v$ in $S_{n-1}$, the map $\Psi_i$ is an isomorphism $X^v_u\to X^{\ins_i(v)}_{\ins_i(u)}$. In particular, $$\deg_{X^{\ins_i(v)}_{\ins_i(u)}}f=\deg_{X^v_u}\rope{i}f.$$
\end{thm}
\begin{proof}
    $X^v_u$ and $X^{\ins_i(v)}_{\ins_i(u)}$ are irreducible of the same dimension $\ell(v)-\ell(u)=\ell(\ins_i(v))-\ell(\ins_i(u))$. Since $\Psi_i$ is a closed embedding, it suffices to show that $\Psi_i(X^v_u)\subset X^{\ins_i(v)}_{\ins_i(u)}.$ 
    
    For $\mathcal{V}\in X^v_u$, let $w_1,\ldots,w_{n-1}$ be a basis of $\CC^{n-1}$ such that for all $j$ we have $\langle w_1,\ldots,w_j\rangle=\mathcal{V}_j$ and $\langle w_{(v')^{-1}(1)},\ldots,w_{(v')^{-1}(j)}\rangle=\langle e_1,\ldots,e_j\rangle$ for all $j$ for some $v'\le v$. Then the vectors $\widetilde{w}_1,\ldots,\widetilde{w}_n\in \CC^n$ with $$\widetilde{w}_j=\begin{cases}\{0\}\oplus w_j&j<i\\e_1&j=i\\ \{0\}\oplus w_{j-1}&j>i\end{cases}$$ are a basis of $\CC^n$ such that $\langle \widetilde{w}_1,\ldots,\widetilde{w}_j\rangle=\Psi_i(\mathcal{V})_j$ and $\langle \widetilde{w}_{\ins_i(v')^{-1}(1)},\ldots,\widetilde{w}_{\ins_i(v')^{-1}(j)}\rangle =\langle e_1,\ldots,e_j\rangle$ so $\Psi_i(\mathcal{V})\in X^{\ins_i(v')}\subset X^{\ins_i(v)}$. 
    Similarly $\Psi_i(\mathcal{V})\in X_{\ins_i(u)}$ and the result follows since $X^{\ins_i(v)}_{\ins_i(u)}=X^{\ins_i(v)}\cap X_{\ins_i(u)}$.
\end{proof}
Recall that a \emph{toric Richardson variety} in $\fl{n}$ is a Richardson variety which is also a torus-orbit closure under the left action of torus of diagonal matrices $(\mathbb{C}^*)^n$ on $GL_n/B=\fl{n}$.
\begin{cor}
\label{cor:iterativetoricR}
    If $X^v_u$ is a toric Richardson variety, then so is $X^{\ins_{i}(v)}_{\ins_i(u)}$.
\end{cor}
\begin{proof}
For $(t_1,\ldots,t_n)\in (\mathbb{C}^*)^n$
    this follows because
    $(t_1,\ldots,t_n)\cdot \Psi_i(\mathcal{V})=\Psi_i((t_2,\ldots,t_n)\mathcal{V})$. 
\end{proof}

\subsection{A geometric realization of $\tope{}$}
To realize $\tope{i}$ geometrically we will compose the geometric realization of $\rope{i}$ from the previous subsection with $\pi_i^{-1}\pi_i$, as in the BGG realization of $\partial_i$. This happens to interact well with the Richardson varieties produced in the previous subsection.
\begin{prop}
\label{prop:projectionNiceRichardsons}    $X^{\ins_{i+1}(v)}_{\ins_i(u)}=\pi_i^{-1}\pi_i\left(X^{\ins_{i+1}(v)}_{\ins_i(u)}\right)$ and $\pi_i\left(X^{\ins_{i+1}(v)}_{\ins_i(u)}\right)=\pi_i\left(X^{\ins_{i}(v)}_{\ins_i(u)}\right)$
\end{prop}
\begin{proof}
 By \Cref{fact:BGGP1} we have  $X^{\ins_{i+1}(v)}=X^{\ins_i(v)s_i}=\pi_i^{-1}\pi_i X^{\ins_i(v)}$ because $i\not\in \des{\ins_i(v)}$.  
 Since $X_{\ins_i(u)}=w_{0,n}X^{w_{0,n}\ins_i(u)}$, where $w_{0,n}$ denotes the longest permutation in $S_n$, we may apply \Cref{fact:BGGP1} again to conclude that $X_{\ins_i(u)}=\pi_i^{-1}\pi_iX_{\ins_i(u)}$ because $i\not\in \des{\ins_i(u)}$.
 
 Combining these, we obtain the equality
$$\pi_i^{-1}\pi_iX^{\ins_i(v)}_{\ins_i(u)}=X_{\ins_i(u)}\cap \pi_i^{-1}\pi_iX^{\ins_i(v)}=X_{\ins_i(u)}\cap X^{\ins_{i+1}(v)}=X^{\ins_{i+1}(v)}_{\ins_i(u)},$$
where we note that the first equality follows just from the $\pi_i$-saturatedness $\pi_i^{-1}\pi_i X_{\epsilon_i(u)}=X_{\epsilon_i(u)}$. This implies both desired statements. 
\end{proof}
\begin{thm}
\label{thm:geometrictope}
For $f\in \poly_n$ and $1\le i \le n-1$ we have
    \begin{align}\deg_{X^{\ins_{i+1}(v)}_{\ins_i(u)}}f&=\deg_{X^v_u}\tope{i}f.\end{align}
\end{thm}
\begin{proof}
We have \begin{equation*}\deg_{X^{\ins_{i+1}(v)}_{\ins_i(u)}}f=\deg_{\pi_i(X^{\ins_{i+1}(v)}_{\ins_i(u)})}\partial_i f=\deg_{\pi_i(X^{\ins_{i}(v)}_{\ins_i(u)})}\partial_i f=\deg_{X^{\ins_{i}(v)}_{\ins_i(u)}}\partial_if=\deg_{X^v_u}\rope{i}\partial_if=\deg_{X^v_u}\tope{i}f.\end{equation*}
Here the first and third equalities follow from \Cref{fact:pifacts}, the second equality follows from \Cref{prop:projectionNiceRichardsons}, the fourth equality follows from \Cref{thm:geometricrope}, and the fifth equality follows from the identity $\rope{i}\partial_i=\tope{i}$.
\end{proof}

Note that $\pi_i\left(X^{\ins_{i+1}(v)}_{\ins_i(u)}\right)=\pi_i\left(X^{\ins_i(v)}_{\ins_i(u)}\right)=(\pi_i\circ \Psi_i)(X^v_u)$, and $\pi_i\circ \Psi_i|_{X^v_u}$ is an isomorphism onto its image since $\Psi_i$ is a closed embedding and $\pi_i:X^v_u\to \pi_i\left(X^{\ins_i(v)}_{\ins_i(u)}\right)$ is an isomorphism. 
This allows us to make the following definition.

\begin{defn}
\label{defn:Phi}
    For $u\le v$ permutations in $S_{n-1}$ and $1\le i \le n-1$, let $\Phi_i:X^{\ins_{i+1}(v)}_{\ins_i(u)}\to X^v_u$ denote the composite map $(\pi_i\circ \Psi_i)^{-1}\circ \pi_i$. 
\end{defn}

\begin{thm}
\label{thm:TP1bundle}
    For $u\le v$ permutations in $S_{n-1}$ and $1\le i \le n-1$, the map $\Phi_i:X^{\ins_{i+1}(v)}_{\ins_i(u)}\to X^v_u$ realizes $X^{\ins_{i+1}(v)}_{\ins_i(u)}$ as the $\PP^1$-bundle $\PP(\ul{\CC}\oplus (\mathcal{V}_i^{(n-1)}/\mathcal{V}_{i-1}^{(n-1)}))$ on $X^v_u$. 
\end{thm}
\begin{proof}
    The map $\pi_i$ makes $X^{\ins_{i+1}(u)}_{\ins_i(v)}$ a $\PP^1$-bundle $\PP\left(\mathcal{V}^{(n)}_{i+1}/\mathcal{V}^{(n)}_{i-1}\right)$ over $\pi_i\left(X^{\ins_i(v)}_{\ins_i(u)}\right)$. 
    Under the isomorphism $\pi_i\circ \Psi_i:X^v_u\cong  \pi_i\left(X^{\ins_i(v)}_{\ins_i(u)}\right)$, we see that $\mathcal{V}^{(n)}_{i+1}/\mathcal{V}^{(n)}_{i-1}$ corresponds to the vector bundle
    \begin{equation*}
        \Psi_i^*\pi_i^*\left(\mathcal{V}^{(n)}_{i+1}/\mathcal{V}^{(n)}_{i-1}\right)=\Psi_i^*\left(\mathcal{V}^{(n)}_{i+1}/\mathcal{V}^{(n)}_{i-1}\right)=\left(\ul{\CC}\oplus \mathcal{V}_i^{(n-1)}\right)/\left(\{0\}\oplus\mathcal{V}_{i-1}^{(n-1)}\right)=\ul{\CC}\oplus \left(\mathcal{V}_{i}^{(n-1)}/\mathcal{V}_{i-1}^{(n-1)}\right).\qedhere
    \end{equation*}
\end{proof}
\begin{cor}
\label{cor:iterativetoricT}
    If $X^v_u$ is a toric Richardson variety, then so is $X^{\ins_{i+1}(v)}_{\ins_i(u)}$.
\end{cor}
\begin{proof}
    The action of  $(\mathbb{C}^*)^n$ on a triple $(x,a,s)$ where $x\in X^{\ins_i(v)}_{\ins_i(u)}$, $a\in \CC$ and $s\in \left(\mathcal{V}_{i}^{(n-1)}/\mathcal{V}_i^{(n-1)}\right)_x$ is given by \[
    (t_1,\ldots,t_n)\cdot (x,a,s)=((t_2,\ldots,t_n)\cdot x,t_1\cdot a, (t_2,\ldots,t_n)\cdot s).\] 
    For generic $x$, we can use $t_2,\dots,t_n$ to move $x$ to any point in the base and then for $a,s\ne 0$ we can use $t_1$ to change the fiber to any $a',s'\ne 0$.
\end{proof}



\section{\texorpdfstring{$\rt$}{Omega}-Richardson varieties}
\label{sec:rtRichardson}

We will want to iterate the results of the previous section, which will involve iteratively applying the maps $(u,v)\mapsto (\ins_j(u),\ins_j(v))$ and $(u,v)\mapsto (\ins_j(u),\ins_{j+1}(v))$ for varying $j$. Let $S_{\infty}=\bigcup S_n$ denote the group of permutations of $\{1,2,\ldots\}$ that fix all but finitely many elements, with $S_n$ identified with the subgroup fixing all $i\ge n+1$. 
Then we may view $\ins_i:S_{\infty}\to S_{\infty}$.
\begin{defn}
\label{defn:uvfromrt}
    For $\rt\in\rtseq$ we define $u(\rt),v(\rt)\in S_{\infty}$ recursively as follows. 
    If $|\rt|=0$ then we set $u=v=\idem$. 
    If $|\rt|\ge 1$ then writing $\rt=\rt'\xletter{}$, we define $$(u(\rt),v(\rt))=\begin{cases}(\ins_ju(\rt'),\ins_jv(\rt'))&\text{ if }\xletter{}=\rletter{j}\\(\ins_ju(\rt'),\ins_{j+1}v(\rt'))&\text{ if }\xletter{}=\tletter{j}.
    \end{cases}$$
\end{defn}
Appendix~\ref{sec:uvPerms} collects relevant combinatorial results about $u(\rt)$ and $v(\rt)$. 
We shall recall them wherever necessary.

It is immediate that for $\rt\in \rtseq_n$ the permutations $u(\rt)$ and $v(\rt)$ are in $S_n$.
We are now ready to introduce a class of Richardson varieties $X(\rt)$ which will play the quasisymmetric analogue of the Bott--Samelson varieties. 
\begin{defn}
    For $\rt\in \rtseq_n$ we define the \emph{$\rt$-Richardson variety} to be
    $$X(\rt)\coloneqq X^{v(\rt)}_{u(\rt)}\subset \fl{n}.$$
\end{defn}
The dimension of this variety is given by
\begin{align}\dim X(\rt)=\ell (v(\rt))-\ell (u(\rt))=|\rt|_{\tletter{}}.\end{align}
By \Cref{cor:minmaxrt}, the minimal dimension $X(\rt)$ are in bijection with the $T$-fixed points $X^u_u$ of $\fl{n}$, and the maximal dimension $X(\rt)$, i.e. those with $|\rt|_{\tletter{}}=n-1$, are exactly those of the form $X^{uc_n}_u$ for $u\in S_n$ satisfying $u(n)=n$ and $c_n=n1\cdots (n-1)$ the backwards long cycle, which are precisely the toric Richardson varieties considered in \cite{HHMP,lian2023hhmp}.

Recall that for any $f\in H^\bullet(\fl{n})$, the Bott--Samelson resolution $\BS(i_1,\ldots,i_k)\to X^w$ has the property that
$\deg_{\BS(i_1,\ldots,i_k)}f=\deg_{\BS(i_1,\ldots,i_{k-1})}\partial_{i_k}f$, and
$$\deg_{X^w}f=\deg_{\BS(i_1,\ldots,i_k)}f=\deg_{\BS(i_1,\ldots,i_{k-1})}\partial_{i_k}f=\cdots =\deg_{\BS()}\partial_{i_1}\cdots \partial_{i_k} f=\ct\partial_w f.$$

\begin{thm}[{\Cref{thm:introP1degree}}]
\label{thm:mainP1degree}
Let $f(x_1,\ldots,x_n)\in H^\bullet(\fl{n})$ and $\rt\in \rtseq_n$. Then writing $\rt=\rt'\xletter{}$, the following are true.
\begin{enumerate}
    \item If $\xletter{}=\rletter{i}$ then $X(\rt)\cong X(\rt')$ and $\deg_{X(\rt)}f=\deg_{X(\rt')}\rope{i}f$.
    \item If $\xletter{}=\tletter{i}$ then $X(\rt)\to X(\rt')$ is a $\PP^1$-bundle given as a projectivization $\PP(\CC\oplus \mathcal{L})$, and $\deg_{X(\rt)}f=\deg_{X(\rt')}\tope{i}f$.
\end{enumerate}
Furthermore 
$$\deg_{X(\rt)}f(x_1,\ldots,x_n)=\prodoperator{\rt}^n f.$$   
\end{thm}
\begin{proof}
The result follows from \Cref{thm:geometricrope} and the second result follows from \Cref{thm:TP1bundle} and \Cref{thm:geometrictope}. The final result follows from iteratively applying these first two results.
\end{proof}

Recall that an algebraic variety $X$ is called a \emph{Bott manifold}  if there is a sequence of maps (called a \emph{Bott tower}) $$X_m\to X_{m-1}\to \cdots \to X_0=\{\text{pt}\}$$ where $X_i=\mathbb{P}^1(\ul{\CC}\oplus \mathcal{L})$ for some line bundle $\mathcal{L}$ on $X_{i-1}$ \cite[\S 2]{GK94}.
Bott manifolds are always smooth toric varieties whose moment polytopes are combinatorial cubes \cite{MP08} (i.e. a polytope whose face poset is isomorphic to the face poset of a cube). We will now see that $X(\rt)$ are in fact Bott manifolds, and are smooth torus-orbit closures under the $T$-action on $\fl{n}$. The following result may also be deduced from \cite{Lee21} on smooth toric Bruhat interval polytopes, but it will be important for us to identify explicitly the way the Bott manifold structure is realized on $X(\rt)$.

\begin{thm}
    The $\rt$-Richardson variety $X(\rt)$ is a Bott manifold and a smooth toric Richardson variety.
\end{thm}
\begin{proof}
Let $\rt=\xletter{1}\cdots \xletter{n}\in\rtseq_n$ and $\rt_i=\xletter{1}\cdots\xletter{i}$. Then by \Cref{thm:mainP1degree} there are maps
$$X(\rt)=X(\rt_{n})\to X(\rt_{n-1})\to \cdots \to X(\rt_0)=\{\mathrm{pt}\}$$
such that each map is either an isomorphism or realizes  $X(\rt_i)=\PP(\underline{\CC}\oplus \mathcal{L})$ for some line bundle $\mathcal{L}$ on $X(\rt_{i-1})$, which shows that it is a Bott manifold. Furthermore, iteratively applying  \Cref{cor:iterativetoricR} and \Cref{cor:iterativetoricT} to this sequence shows that $X(\rt)$ is in fact a smooth torus-orbit under the $T$-action on $\fl{n}$.
\end{proof}
We now show that the sub-torus-orbit closures in $X(\rt)$ are themselves of the form $X(\rt')$.
\begin{thm}
\label{thm:containedtorusorbit}
     The torus-orbit closures contained inside $X(\rt)$ are exactly $X(\rt')$ with $\rt'$ obtained by replacing some subset of the $\tletter{i}$ appearing in $\rt$ with either $\rletter{i}$ or $\rletter{i+1}$.
\end{thm}
\begin{proof}
     Because $X(\rt)$ is toric, by \cite{TW15} the sub-torus-orbit closures are the Richardson varieties contained in $X(\rt)$. These are the $X(\rt')$ as described in~\Cref{prop:Bruhat}. Each of these is smooth since $X(\rt)$ is smooth, and the moment polytope of $X(\rt')$ is a face of the moment polytope of $X(\rt)$ which is therefore a combinatorial cube, so the result follows.
\end{proof}

We now give a more concrete interpretation of the $\rt$-Richardson varieties by describing a parametrization of the open torus orbit in $X(\rt)$, which for the maximal dimension $X(\rt)$ recovers the description from \cite[Section 4.2]{lian2023hhmp}.

\begin{defn}
    For $\rt\in \rtseq_n$, let $M(\rt)$ be an $n\times n$ matrix filled with $0,1,\ast$ defined recursively as follows. Set $M(\emptyset)$ to be the unique $0\times 0$ matrix, and for $M(\rt')=\begin{bmatrix}v_1&\cdots & v_{n-1}\end{bmatrix}$  let
    \begin{align*}
        M(\rt'\rletter{i})&=\begin{bmatrix}0&\cdots & 0 & 1 & 0 & \cdots & 0\\
    v_1 & \cdots & v_{i-1}& \bm{0}&v_i&\cdots & v_{n-1}\end{bmatrix},\\
    M(\rt'\tletter{i})&=\begin{bmatrix}0&\cdots &  \ast & 1 & 0 & \cdots & 0\\
    v_1 & \cdots & v_{i}& \bm{0}&v_{i+1}&\cdots & v_{n-1}\end{bmatrix}.
    \end{align*}
\end{defn}

\begin{eg}
\label{eg:matrices}
The matrices for $\rletter{1},\rletter{1}\tletter{1},\rletter{1}\tletter{1}\tletter{1},\rletter{1}\tletter{1}\tletter{1}\rletter{2}$ and $\rletter{1}\tletter{1}\tletter{1}\rletter{2}\tletter{4}$ are
\begin{align*}
    \begin{bmatrix}1\end{bmatrix},\text{ }\begin{bmatrix}\ast&1\\1&0\end{bmatrix},\text{ }\begin{bmatrix}\ast &1 & 0\\
    \ast&0 & 1\\
    1 & 0&0\end{bmatrix}\text{, }\begin{bmatrix}0&1&0&0\\ \ast&0 & 1 & 0\\ \ast & 0 & 0 & 1\\ 1 & 0 & 0 & 0\end{bmatrix},\text{ } \begin{bmatrix}0&0&0&\ast&1\\
    0&1&0&0&0\\
    \ast&0&1&0&0\\
    \ast&0&0&1&0\\1&0&0&0&0\end{bmatrix}.
\end{align*}
\end{eg}
\begin{thm}
    The open torus in $X(\rt)$ is parametrized by replacing the $|\rt|_t$-many $\ast$ that appear in $M(\rt)$ with elements of $\CC^*$.
\end{thm}
\begin{proof}
    Assume by induction that the open torus in $X(\rt')$ is given by this description, and
    let $A=\begin{bmatrix}w_1&\cdots & w_{n-1}\end{bmatrix}$ be a matrix representing a point in this open torus-orbit. Then $$\begin{bmatrix}0&\cdots & 0 & 1 & 0 &0& \cdots & 0\\
    w_1 & \cdots & w_{i-1}& \bm{0}&w_i&w_{i+1}&\cdots & w_{n-1}\end{bmatrix}
    $$ represents $\Psi_i(A)$ so applying this for all points in the open torus of $X(\rt')$ shows that the new open torus in $X(\rt'\rletter{i})$ is $M(\rt'\rletter{i})$ as desired. Applying $\pi_i^{-1}\pi_i$ to $\Psi_i(A)$ corresponds to replacing the $i$'th and $(i+1)$'st column of this matrix with two linear combinations. 
    In $GL_n/B$ two matrices are the same if we can obtain one from the other by scaling columns and adding multiples of one column to a future column. 
    Doing this we can put every such matrix either into the form $$\begin{bmatrix}0&\cdots &0& \ast & 1 & 0 & \cdots & 0\\
    w_1 & \cdots &w_{i-1}& w_{i}& \bm{0}&w_{i+1}&\cdots & w_{n-1}\end{bmatrix}\text{ or }\begin{bmatrix}0&\cdots & 0 & 1 & 0 &0& \cdots & 0\\
    w_1 & \cdots & w_{i-1}& \bm{0}&w_i&w_{i+1}&\cdots & w_{n-1}\end{bmatrix}$$
    with $\ast\in \CC$. The first matrix with $\ast=0$ is gives a point in the boundary torus-orbit closure $X(\rt'\rletter{i+1})$, the second matrix gives a point in the boundary torus-orbit closure $X(\rt'\rletter{i})$, so the remaining points where $\ast\in \CC^*$ in the first matrix lie in the open torus-orbit of $X(\rt'\tletter{i})$.
\end{proof}

From $M(\rt)$ one can read off the sequence in $\rt\in \rtseq_n$ by reversing the recursive process used to build $M(\rt)$. The following non-recursive characterization of the matrices $M(\rt)$ is straightforward to show.
\begin{thm}
\label{thm:rtmatrix}
    A square matrix filled with $0,1,\ast$ is of the form $M(\rt)$ if and only if
    \begin{itemize}
        \item The $1$'s form a permutation matrix
        \item There is at most one $\ast$ per row. 
        \item Every $\ast$ appears above the $1$ in its column and to the left of the $1$ in its row.
        \item For every $0$ between $\ast$ and $1$ in a row, all the entries below this $0$ are also $0$.
    \end{itemize}
\end{thm}

\begin{figure}[!ht]
    \centering
    \includegraphics[width=0.5\linewidth]{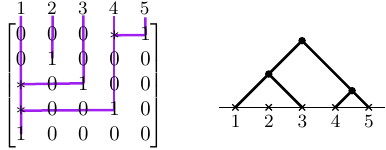}
    \caption{Recovering $\wh{F}(\Omega)$ from $M(\Omega)$}
    \label{fig:matrix-tree}
\end{figure}

We conclude by describing how to directly recover the nested forest $\wh{F}(\rt)$ from $M(\rt)$. 
The forest $\wh{F}(\rt)$ has internal nodes the nonzero entries of $M(\rt)$ and leaves $\{1,\ldots,n\}$. For each row containing an $\ast$ and a $1$, we do the following. 
Connect the $1$ and $\ast$, and then connect each of these to the nearest $\ast$ above them in the same column provided it exists; otherwise connect to the leaf $i$ where $i$ is the column number. 
This is illustrated in Figure~\ref{fig:matrix-tree} using the last matrix in \Cref{eg:matrices}.

\section{The quasisymmetric Schubert cycles \texorpdfstring{$X(\wh{F})$}{X(F tilde)} and Bott manifold structures}
\label{sec:Bott}
We now define the quasisymmetric Schubert cycles associated to $\wh{F}\in \nsuppfor{n}$. In order to define them, we will need to prove the following theorem.

\begin{thm}
\label{thm:rtwexists}
If $\wh{F}\in \nfor_n$ and $\rt,\rt'\in \Trim{\wh{F}}$, then
    $$u(\rt)^{-1}X(\rt)=u(\rt')^{-1}X(\rt).$$
\end{thm}
Given this theorem, the following is well-defined.
\begin{defn}
    We define the \emph{quasisymmetric Schubert cycle} associated to $\wh{F}\in \nfor_n$ to be $X(\wh{F})=u(\rt)^{-1}X(\rt)$ for any $\rt\in \Trim{\wh{F}}$.
\end{defn}
\begin{proof}[Proof of \Cref{thm:rtwexists}]
First, note that the set of row vectors in $M(\rt)$ is determined by $\wh{F}$. Indeed, there is a row with exactly one $1$ in position $i$ exactly when there is a tree in $\wh{F}$ whose leftmost leaf is $i$, and there is a row with a $\ast$ in position $i$ and a $1$ is position $j$ if there is an internal node $v$ with the property that the leftmost leaf descendent of its left child $v_L$ is $i$ and the leftmost leaf descendent of its right child $v_R$ is $j$.

By induction, $v(\rt)$ is the permutation matrix obtained by selecting the rightmost entry from each row of $M(\rt)$. Therefore $v(\rt)^{-1}X(\rt)$ is obtained by permuting the rows in the matrix model in the unique way so that the rightmost entries lie on the main diagonal, which shows $v(\rt)^{-1}X(\rt)=v(\rt')^{-1}X(\rt')$.

By induction, $u(\rt)^{-1}v(\rt)$ is the product of the backwards cycles on the support sets of the trees in $\wh{F}$. Therefore $u(\rt)^{-1}v(\rt)=u(\rt')^{-1}v(\rt')$, and so $u(\rt)^{-1}X(\rt)=u(\rt')^{-1}X(\rt')$ as desired.
\end{proof}

The following is analogous to the fact that for $f\in H^\bullet(\fl{n})$ we have $\deg_{X^w}f=\ct\partial_w f$.
\begin{cor}
For $\wh{F}\in \nsuppfor{n}$, we have $$\deg_{X(\wh{F})}f=\prodoperator{\wh{F}}^nf.$$
\end{cor}
\begin{proof}
If $\rt\in \Trim{F}$, then for any $f\in H^\bullet(\fl{n})$ we have
$\deg_{X(\rt)}f=\prodoperator{\rt}^nf=\prodoperator{\wh{F}}^nf.$
\end{proof}

Fix some $\rt=\xletter{1}\cdots \xletter{n}\in \Trim{\wh{F}}$, and let $\rt_i=\xletter{1}\cdots \xletter{i}$. 
Then under the isomorphism $X(\rt)\to X(\wh{F})$, if $F\in H^\bullet(\fl{n})$ we can compute using the Bott manifold structure on $X(\rt)$
$$\deg_{X(\wh{F})}f=\deg_{X(\rt_n)}f=\deg_{X(\rt_{n-1})}\operatorname{up}(\xletter{n})f=\cdots = \deg_{X(\rt_0)}\operatorname{up}(\xletter{1})\cdots \operatorname{up}(\xletter{n})f=\prodoperator{\wh{F}}^nf.$$
In this way, the isomorphisms $X(\rt)\to X(\wh{F})$ play an analogous role as the Bott--Samelson resolutions play for Schubert varieties outlined in \Cref{sec:Results}.

The fundamental classes of the quasisymmetric Schubert cycles $X(\wh{F})$ satisfy linear relations among them, but by passing to the subset indexed by indexed forests $\suppfor{n}\subset \nsuppfor{n}$ we obtain a duality with the family of forest polynomials.
\begin{thm}
For $F\in \suppfor{n}\subset \widehat{\suppfor{n}}$ and $f\in \poly_n$ we have $\deg_{X(F)} f=\ct\tope{F}f$. In particular,
    $$\deg_{X(F)}\forestpoly{G}=\delta_{F,G},$$
    i.e. the forest polynomials $\{\forestpoly{G}\suchthat G\in \suppfor{n}\}\subset H^\bullet(\fl{n})$ are Kronecker dual to the fundamental classes $\{[X(F)]\suchthat F\in \suppfor{n}\}\subset H^\bullet(\fl{n})$.
\end{thm}
\begin{proof}
Because $F\in \suppfor{n}$, we can express $F=\nf{\rt}$ for $\rt=\rletter{1}^{n-\sum c_i}\tletter{1}^{c_1}\cdots \tletter{k}^{c_k}$ where $F=1^{c_1}2^{c_2}\cdots k^{c_k}$. This means that $\ct\rtc=\ct \rletter{1}^{n-\sum c_i} \tletter{1}^{c_1}\cdots \tletter{k}^{c_k}=\ct\tope{F}$. Therefore
\begin{equation*}
    \deg_{X(F)}\forestpoly{G}=\ct\tope{F} \forestpoly{G}=\delta_{F,G}.\qedhere
\end{equation*}
\end{proof}
We now show that the remaining fundamental classes are nonnegative linear combinations of the fundamental classes associated to the non-nested forests.
\begin{thm}\label{th:a_F}
For $\widehat{H}\in\nsuppfor{n}$, we have $$[X(\widehat{H})]=\sum_{F\in \suppfor{n}} a_F[X(F)]\in H^\bullet(\fl{n})$$
for nonnegative integers $a_F$.
\end{thm}
\begin{proof}
Let $\rt\in \rtseq_n$ have $\nf{\rt}=\wh{H}$ so that $[X(\rt)]=[X(\wh{H})]$.
    Note that $$\tope{i}\rope{i+1}=\rope{i}\tope{i+1}+\rope{i+1}\tope{i}.$$
    If $\tletter{i}\rletter{i+1}$ in $\rt$ and $\rt_1$, $\rt_2$  are the sequences where these two letters are replaced with $\rletter{i}\tletter{i+1}$ and $\rletter{i+1}\tletter{i}$ respectively, then $\rt_1$, $\rt_2\in \rtseq_n$ and $\prodoperator{\rt}^n=\prodoperator{\rt_1}^n+\prodoperator{\rt_2}^n$. This implies  $[X(\rt)]=[X(\rt_1)]+[X(\rt_2)]$ by Poincar\'e duality.
    Applying this repeatedly allows us to move all $\rletter{i}$ to the left of all $\tletter{j}$ to express $$[X(\rt)]=\sum a_{\rt'}[X(\rt')]$$
    where each $\rt'=\rletter{i_1}\cdots \rletter{i_k}\tletter{j_1}\cdots \tletter{j_\ell}$, and so $\nf{\rt'}\in \indexedforests_n$.
\end{proof}

\section{The \texorpdfstring{$\rt$}{Omega}-flag manifold}
\label{sec:qfl}
We are now ready to define our quasisymmetric analogue of the flag variety. Rather than take the union of the $X(\wh{F})$, which turns out to be a much more combinatorially opaque object as a toric complex, we instead take the union of the $X(\rt)$ with $\rt\in \rtseq_n$, each of which as we have already shown is isomorphic to some $X(\wh{F})$.
\begin{defn}
We define the $\rt$-flag variety to be 
$$\hhmp_n=\bigcup_{\rt\in \rtseq_n}X(\rt)\subset \fl{n}.$$
\end{defn}

This also has a natural intrinsic recursive characterization.
\begin{thm}
\label{thm:recursivehhmp}
We have $\hhmp_1=\fl{1}$, and for $n>1$ we have $\hhmp_n$ is the set of all $\mathcal{V}\in \hhmp_n$ such that there is $\mathcal{W}\in \hhmp_{n-1}$ and $i\in\{1,\ldots,n-1\}$ such that $\mathcal{V}_j=\{0\}\oplus \mathcal{W}_j$ for $j<i$ and $\mathcal{V}_j=\CC \oplus \mathcal{W}_{j-1}$ for $j>i$.
\end{thm}
\begin{proof}
    This recursive description claims that $\hhmp_n=\Phi_1(\hhmp_{n-1})\cup \cdots \cup \Phi_{n-1}(\hhmp_{n-1})$ where $\Phi_i$ is the operator from \Cref{defn:Phi}. By induction, it remains to prove
    $$\bigcup_{1\le i_j\le j}\Phi_{i_{n-1}}\cdots \Phi_{i_1}(\fl{1})=\hhmp_n.$$
     Because $\Psi_1(\fl{0})=\fl{1}$, by the recursive construction of the $X(\rt)$, we see that the left hand side is the union of all $X(\rt)$ with $\rt\in \rtseq_n$ of the form $\rletter{1}\tletter{i_1}\cdots \tletter{i_{n-1}}$. But by \Cref{thm:containedtorusorbit}, every other $X(\rt')$ with $\rt'\in \rtseq_n$ is a subvariety of the $X(\rt)$ with $\rt\in \rtseq_n$ obtained by replacing each $\xletter{i}=\rletter{j}$ beyond the first one with either $\tletter{j}$ or $\tletter{j-1}$ so that the index lies in $\{1,\ldots,i\}$, so we conclude.
\end{proof}

\begin{eg}
\label{eg:small_qfl}
$\hhmp_1=\fl{1}$ is a point, and $\hhmp_2=\fl{2}$.\\
For $\mathcal{V}=(0\subset \mathcal{V}_1\subset\mathcal{V}_2\subset\CC^3)\in \fl{3}$, we have $\mathcal{V}\in \hhmp_{3}$ if $\mathcal{V}_2$ contains $\CC e_1$ or if $\mathcal{V}_1$ is contained in $\CC e_2 \oplus \CC e_3$. Note that both can be true, and this occurs precisely for flags with $\mathcal{V}_1\in \CC e_2 \oplus \CC e_3$ and $\mathcal{V}_2=\CC e_1\oplus \mathcal{V}_1$. 
\end{eg}
\subsection{A cubical complex}
We now describe the structure of the toric complex $\hhmp_n$. In \cite{HHMP} it was shown that the images of the top-dimensional $X(\rt)$ appearing in $\hhmp_n$ under the moment map $\fl{n}\to \mathbb{R}^n/\langle (1,\ldots,1)\rangle$ give a subdivision of the permutahedron
$$\Perm_n\coloneqq \operatorname{conv}\{w\cdot (n,\ldots,1)\suchthat w\in S_n\}$$
into combinatorial cubes.

As a consequence of the $\rtseq_n$-encoding for the combinatorial cubes appearing in this subdivision, this subdivision is combinatorially isomorphic to a unit cube subdivision of a particular cuboid (a fact that does not appear to have been previously observed).
\begin{defn}
 For $\rt=\xletter{1}\cdots\xletter{n}\in \rtseq_n$,
 let $\square_\rt$ be defined as the box $Y_2\times \cdots \times Y_{n}$ where
    $$Y_i=\begin{cases}[j,j+1]&\xletter{i}=\tletter{j}\\\{j\}&\xletter{i}=\rletter{j}.\end{cases}$$
\end{defn}

\begin{thm}
There is a face preserving bijection from the image of $\hhmp_n$ under the moment map to the unit cube subdivision of the $(n-1)$-dimensional cuboid $[1,2]\times [1,3]\times \cdots \times [1,n]$, mapping the polytope $P^{v(\rt)}_{u(\rt)}$ to the box $\square_{\rt}$.
\end{thm}
\begin{proof}
For a toric variety $X$, the faces of the moment polytope are the images of the sub-torus-orbit closures of $X$. For any $X(\rt)$ with $\rt\in \rtseq_n$, these sub-torus-orbit closures were identified in \Cref{thm:containedtorusorbit} to be the $X(\rt')$ with $\rt'$ obtained by replacing some subset of the $\tletter{j}$ with either $\rletter{j}$ or $\rletter{j+1}$. By construction this happens precisely when $\square_{\rt'}\subset \square_{\rt}$, and we conclude.
\end{proof}
See Figure~\ref{fig:squares_and_cubes_4} demonstrating the $\hhmp$-subdivision and the unit cube subdivision for $n=4$.
\begin{figure}
    \centering
    \includegraphics[scale=1.0]{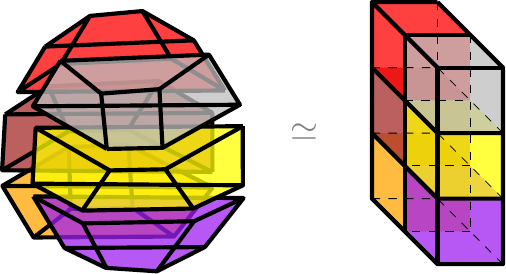}
    \caption{The $\hhmp$ subdivision and the unit cube subdivision for $n=4$}
    \label{fig:squares_and_cubes_4}
\end{figure}

\subsection{Quasisymmetric coinvariants in $H^\bullet(\hhmp_n)$}

We now prove the following.
\begin{thm}
\label{thm:QsymSchubertcycles}[{\Cref{maintheorem:QsymSchubertcycles}}]
    The image of $H^\bullet(\fl{n})$ in $H^\bullet(\hhmp_n)$ under the natural restriction map is isomorphic to $\qscoinv{n}$.
\end{thm}
First, we show that the normalization $\wt{\hhmp}_n$ of $\hhmp_n$ knows $\qscoinv{n}$. 
Because $X(\rt)$ are smooth,  $\wt{\hhmp}_n$ is the disjoint union $$\bigsqcup_{\substack{\rt\in \rtseq_n\\ |\rt|_{\tletter{}}=n-1}}X(\rt).$$

\begin{prop}
\label{prop:normalqfl}
The image of $H^\bullet(\fl{n})$ in 
    $H^\bullet(\widetilde{\hhmp}_n)$ is isomorphic to $\qscoinv{n}$.
\end{prop}
\begin{proof}
    The statement is equivalent to showing that $\qsymide{n}$ is the kernel of the map 
    \[\ZZ[x_1,\ldots,x_n]\to\prod_{\substack{\rt\in \rtseq_n\\ |\rt|_{\tletter{}}=n-1}}H^\bullet(X(\rt)).
    \]
    By the Minkowski weight description of the cohomology ring of smooth toric varieties \cite{FuSt97}, a cohomology class of degree $k$ vanishes if and only if its restriction to every sub-torus-orbit closure is $0$. Indeed, \cite[Proposition 1.1]{FuSt97} tells us that this cohomology ring is generated by the sub-torus-orbit closures.
    Since the sub-torus-orbit closures occurring in some $X(\rt)$ are of the form $X(\rt')$ with $\rt'\in \rtseq_n$ and $\deg_{X(\rt')}f=\prodoperator{\rt'}f$ for $f\in \ZZ[x_1,\ldots,x_n]$, we have reduced to showing that
    $$f\in \qsymide{n}\Longleftrightarrow \prodoperator{\rt'}^nf=0\text{ for all }\rt'\in\rtseq_n.$$ 
    The forward direction follows because, as mentioned in \Cref{sec:Results}, the operator  $\prodoperator{\rt'}$ descends to $\qscoinv{n}$. 
    For the reverse direction we note that $\prodoperator{\wh{F}}f=0$ for all $\wh{F}\in \nsuppfor{n}$, so we apply this fact to the subset of indexed forests $\suppfor{n}\subset \nsuppfor{n}$. By \Cref{cor:ev0TF} we have $\ct\tope{F}=\prodoperator{F}^n$ for $F\in \suppfor{n}$ and we conclude.
\end{proof}

To show that $H^\bullet(\hhmp_n)$ contains the quasisymmetric coinvariants, we will show that 
\[
H^\bullet(\hhmp_n)\to H^\bullet(\wt{\hhmp}_n)
\]
is an injection. This is a topological statement, and we show this using a spectral sequence.
\begin{defn}
        Let $F_n$ be the topological space obtained by gluing $n$ copies of the sphere $\PP^1$ in a line. Formally, if we define points $s_i$ for $1\le i \le n-1$ belonging to the $i$'th sphere and $t_i$ for $1\le i \le n-1$ belonging to the $(i+1)$'st sphere, with $s_{i+1}\ne t_i$ for $2\le i \le n-1$, then we define $F$ to be the quotient
        $$F=\underbrace{\mathbb{P}^1\sqcup \cdots \sqcup \mathbb{P}^1}_n/\sim$$ by the equivalence relation $\sim$ identifying $s_{i}$ with $t_i$.
\end{defn}
    \begin{prop}
    \label{prop:spectral}
        Suppose that $X$ is a (not necessarily reducible) algebraic variety with trivial first fundamental group $\bm{\pi}_1(X)$ and all even dimensional cohomology groups. 
        Suppose $Y_1,\ldots,Y_k\to X$ are $\PP^1$-bundles with distinguished sections $s_i:X\to Y_i$ for $i=1,\ldots,k-1$ and $t_i:X\to Y_{k+1}$ for $i=1,\ldots,k-1$ such that $s_{i+1}$ and $t_i$ are disjoint for $2\le i \le n-1$.
        Let $Y$ be the space obtained by identifying the sections $s_i$ and $t_i$ for each  $1\leq i\leq n-1$. 
        Then the pullback map$$H^\bullet(Y)\to H^\bullet(Y_1\sqcup \cdots \sqcup Y_k)$$
        is an injection.
    \end{prop}
    \begin{proof}
        $Y\to X$ is a fiber bundle with fiber $F_k$, which has $H^0(F)=\ZZ$ and $H^2(F)=\ZZ^k$. 
        Because $\bm{\pi}_1(X)=0$ there are no nontrivial local systems on $X$, so  $H^i(X;H^j(F))=0$ if $i,j$ are not both even, so the first page of the Serre spectral sequence degenerates and we have as abelian groups
        \begin{align*}
        \bigoplus_{i,j} H^i(X;H^j(F)))\cong H^\bullet(Y).
        \end{align*}
        The same reasoning for $Y_1\sqcup \cdots \sqcup Y_k\to X$ shows that
        \[
        \bigoplus_{i,j} H^i(X;H^j(S^2\sqcup \cdots \sqcup S^2)))\cong H^\bullet(Y).
        \]
        Finally, the map $H^\bullet(Y)\to H^\bullet(Y_1\sqcup \cdots \sqcup Y_k)$ is induced by the natural maps
        $H^i(X;H^j(F))\to H^i(X;H^j(\PP^1\sqcup \cdots \sqcup \PP^1))$. Because $H^j(F)$ and $H^j(\PP^1\sqcup \cdots \sqcup \PP^1)$ are free abelian groups, this map is given by the composite
        $$H^i(X;H^j(F))\cong H^i(X)\otimes H^j(F)\to H^i(X)\otimes H^j(S^2\sqcup \cdots \sqcup S^2)\cong H^i(X;H^j(S^2\sqcup \cdots \sqcup S^2)),$$
        and the middle map is an injection because $H^j(F)\hookrightarrow H^j(S^2\sqcup \cdots \sqcup S^2)$ is an injection of free abelian groups.
    \end{proof}
    We are now ready to prove \Cref{thm:QsymSchubertcycles}. 
    \begin{proof}[{Proof of \Cref{thm:QsymSchubertcycles}}]
    By \Cref{prop:normalqfl}, it suffices to show that the pullback map $H^\bullet(\hhmp_n)\to H^\bullet(\wt{\hhmp}_n)$ is an injection. 
    First we show that $\bm{\pi}_1(\hhmp_n)$ is trivial by induction. 
    For $1\le i \le n-1$ let $$Y_i=\bigcup_{\rt=\xletter{1}\cdots \xletter{n}\text{ with }\xletter{n}\in \{\rletter{i},\tletter{i},\rletter{i+1}\}} X(\rt),$$
        and note that this is a $\PP^1$-bundle over $\hhmp_{n-1}$ by the map $(\Psi_i\circ \pi_i)^{-1}\pi_i$, which has distinguished sections coming from $\Psi_i$ and $\Psi_{i+1}$. This realizes $\hhmp_n$ as an $F_{n-1}$-bundle over $\hhmp_{n-1}$. As $\bm{\pi}_1(F_{n-1})$ is trivial we have by the long exact sequence on homotopy groups that $\bm{\pi}_1(\hhmp_n)=\bm{\pi}_1(\hhmp_{n-1})$, which is trivial by the inductive hypothesis. 
        
        Now we prove that $H^\bullet(\hhmp_n)\to H^\bullet(\wt{\hhmp}_n)$ is an injection by induction.  By the inductive hypothesis we have an injection $H^\bullet(\hhmp_{n-1})\hookrightarrow H^\bullet(\wt{\hhmp}_{n-1})$, so because $\hhmp_{n-1}$ is a disjoint union of smooth projective varieties we have that the  even-dimensional cohomology of $\hhmp_{n-1}$ vanishes.

The variety $\hhmp_n$ arises as the $Y$ from \Cref{prop:spectral} applied to the $\PP^1$-bundles $Y_1,\ldots,Y_{n-1}\to Y$, so we conclude that $ H^\bullet(\hhmp_n)\to H^\bullet(Y_1\sqcup \cdots \sqcup Y_{n-1})$ is an injection. 
Let $$\wt{Y_i}=\bigsqcup_{\rt=\xletter{1}\cdots \xletter{n}\text{ with }\xletter{n}\in \{\rletter{i},\tletter{i},\rletter{i+1}\}} X(\rt).$$
   Note that $\wt{Y_i}$ is the $\PP^1$-bundle $Y_i\to \hhmp_{n-1}$ pulled back along the normalization $\wt{\hhmp}_{n-1}\to \hhmp_{n-1}$. By induction $H^\bullet(\hhmp)\hookrightarrow H^\bullet(\wt{\hhmp}_{n-1})$ is an injection on cohomology, so by the projective bundle formula we have the pullback map $H^\bullet(Y_i)\to H^\bullet(\wt{Y_i})$ is an injection. Therefore we have the composite pullback map
        $$H^\bullet(\hhmp_n)\hookrightarrow H^\bullet(Y_1\sqcup \cdots \sqcup Y_n)\hookrightarrow H^\bullet(\wt{Y_1}\sqcup \cdots \sqcup \wt{Y_n})\hookrightarrow H^\bullet(\wt{\hhmp}_n),$$
        is an injection as desired.     
    \end{proof}

\section{Applications to generalized Littlewood--Richardson coefficients}
\label{sec:generalizedLR}
In this section we show the combinatorial aspects of \Cref{thm:introP1degree} and its application to generalized LR coefficients $c^v_{u,w}$ by recasting various coefficients of interest in algebraic combinatorics as generalized LR coefficients.
Recall that the $c^v_{u,w}$  are defined by the equation $$\schub{u}\schub{w}=\sum_v c^v_{u,w}\schub{v},$$ or equivalently as 
$$c^v_{u,w}=\ct\partial_v(\schub{u}\schub{w}).$$
We have $\ct\partial_v(\schub{u}f)=0$ if $u\not \le v$ in the Bruhat order on $S_{\infty}$, so in particular in this case $c^v_{u,w}=0$ for all $w$. Coming up with a combinatorial rule for the $c_{u,w}^v$ is a major open problem in combinatorial algebraic geometry; see \cite{GaoZhu24, Hu23, huang2022bumpless, Ko01, KZJ1, KZJ3, PW22} for various results in special cases.


If we want to show combinatorially for fixed $u\le v$ that $c^v_{u,w}$ is nonnegative for all $w$, then this is equivalent to showing combinatorially that the operator $$f\mapsto \ct\partial_v(\schub{u}f)$$
is nonnegative on Schubert polynomials. 
We will show for $(u,v)=(u(\rt),v(\rt))$ for $\rt\in \rtseq$ that we can combinatorially realize this nonnegativity by a recursive procedure, and then give certain pairs of $(u,v)$ where $c^v_{u,w}$ computes interesting combinatorial invariants of $\schub{w}$.
As we shall need them we record here the \emph{twisted Leibniz relations} satisfied by the divided difference operators:
$$\partial_i(fg)=f\partial_i(g)+\partial_i(f)(s_i\cdot g)$$
for all $f,g\in \poly$, so in particular if $f$ is symmetric in variables $\{x_i,x_{i+1}\}$ then $\partial_i(fg)=f\partial_i(g)$.



The results in this section rest on the following proposition. 
We shall give two proofs, one geometric and the other combinatorial. 
The latter is to emphasize the fact that the generalized LR coefficient computation of this section can be made entirely combinatorial.
\begin{prop}
\label{prop:combinatorialropetope}
For $u,v\in S_{\infty}$ we have
    \begin{align} 
\label{eqn:combinatorialrope}
\ct\partial_{\ins_j(v)}(\schub{\ins_j(u)}f)&=\ct\partial_{v}(\schub{u}\rope{j}f),\\
\label{eqn:combinatorialtope}
\ct \partial_{\ins_{i+1}(v)}(\schub{\ins_i(u)}f)&=\ct \partial_{v}(\schub{u}\tope{i}f).
    \end{align}
    Furthermore, for $\rt\in \rtseq$ we have
    \begin{align}\label{eq:gen_lr}
        \ct \partial_{v(\rt)}(\schub{u(\rt)}f)=\ct\prodoperator{\rt}f.
    \end{align}
\end{prop}
\begin{proof}[Geometric Proof of \Cref{prop:combinatorialropetope}]
    Because $\ct\partial_v(\schub{u}f)=\deg_{X^v_u}f$, these results follow directly from \Cref{thm:mainP1degree}.
\end{proof}
\begin{proof}[Combinatorial Proof of \Cref{prop:combinatorialropetope}]
    First, we establish \eqref{eqn:combinatorialrope}. For $j=1$ we first note that if $v=s_{i_1}\cdots s_{i_k}$ is a reduced word then $\ins_1(v)=s_{i_1+1}\cdots s_{i_k+1}$ is a reduced word for $\ins_1(v)$. Therefore because $\partial_i\rope{1}=\rope{1}\partial_{i+1}$, we have $\partial_v\rope{1}=\rope{1}\partial_{\ins_1(v)}$.
    Furthermore, because $\schub{u}$ is characterized by the property that $\ct \partial_v\schub{u}=\delta_{u,v}$ and $\ct \partial_v\rope{1}\schub{\ins_1(u)}=\ct \rope{1}\partial_{\ins_1(v)}\schub{\ins_1(u)}=\delta_{\ins_1(u),\ins_1(v)}=\delta_{u,v}$ we conclude that $\rope{1}\schub{\ins_1(u)}=\schub{u}$. Therefore,
    $$\ct\partial_v(\schub{u}(\rope{1}f))=\ct\partial_v(\rope{1}(\schub{\ins_1(u)}f)=\ct\rope{1}\partial_{\ins_1(v)}(\schub{\ins_1(u)}f)=\ct\partial_{\ins_1(v)}(\schub{\ins_1(u)}f).$$

    Now, assume the result for $j$, we will show the result for $j+1\geq 2$. 
    Since $\varepsilon_{j+1}(v)=\varepsilon_j(v)s_j$ we know that $\partial_{\varepsilon_{j+1}(v)}=\partial_{\varepsilon_{j}(v)}\partial_j$.
    Thus we compute first that
    \begin{align*}
        \ct\partial_{\ins_{j+1}(v)}(\schub{\ins_{j+1}}(u)f)=\ct\partial_{\ins_j(v)}\partial_j(\schub{\ins_{j+1}(u)}f).
    \end{align*}
    Now applying the twisted Leibniz rule to the right-hand side we get
    \begin{align*}
        \ct\partial_{\ins_{j+1}(v)}(\schub{\ins_{j+1}}(u)f)=\ct\partial_{\ins_j(v)}((\partial_j\schub{\ins_{j+1}(u)})\,s_jf)+\ct\partial_{\ins_j(v)}(\schub{\ins_{j+1}(u)}\partial_jf).
    \end{align*}
    Now consider the second term on the right-hand side.
     Since $\ins_j(v)\not\ge \ins_{j+1}(u)$ (as $j=(\ins_j(v))^{-1}(1)<(\ins_{j+1}(u))^{-1}(1)=j+1$) we have the vanishing 
     \[
     \ct \partial_{\ins_j(v)}(\schub{\ins_{j+1}(u)}\partial_jf)=0.
     \] 
     For the first term, using  $\ins_{j+1}(u)=\ins_j(u)s_j$, we know that $\partial_j\schub{\ins_{j+1}(u)}=\schub{\ins_j(u)}$.
     Thus it remains to compute $\ct\partial_{\ins_j(v)}(\schub{\ins_{j}(u)}\,s_jf)$.
     By our inductive hypothesis we have
     \begin{align*}
         \ct\partial_{\ins_j(v)}(\schub{\ins_{j}(u)}\,s_jf)=\ct\partial_v(\schub{u}\,\rope{j}s_jf)=\ct\partial_v(\schub{u}\,\rope{j+1}f),
     \end{align*}
      as desired.

     Now we establish \eqref{eqn:combinatorialtope}. Like before we have $\partial_{\ins_{i+1}(v)}=\partial_{\ins_i(v)}\partial_i$  and so we get 
     \begin{align*}
         \ct\partial_{\ins_{i+1}(v)}(\schub{\ins_i(u)}f)=\ct\partial_{\ins_i(v)}\partial_i(\schub{\ins_i(u)}f)
    \end{align*}
    Since $\schub{\ins_i(u)}$ is symmetric in variables $\{x_i,x_{i+1}\}$, the twisted Leibniz rule simplifies to give
    \begin{align*}
       \ct\partial_{\ins_{i+1}(v)}(\schub{\ins_i(u)}f) =\ct\partial_{\ins_i(v)}(\schub{\ins_i(u)}\,\partial_if)=\ct\partial_v(\schub{u}\rope{i}\partial_if)=\ct\partial_v(\schub{u}\tope{i}f),
    \end{align*}
     where the last two equalities use~\eqref{eqn:combinatorialrope} and $\tope{i}=\rope{i}\partial_i$ respectively.

     Finally, iterating~\eqref{eqn:combinatorialrope} and~\eqref{eqn:combinatorialtope} yields~\eqref{eq:gen_lr}, thereby concluding the proof.
\end{proof}

\begin{thm}
   For $\rt\in \rtseq$ we have
    $$c^{v(\rt)}_{u(\rt),w}=\ct \rtc \schub{w}$$
\end{thm}
\begin{proof}
    This follows from \Cref{prop:combinatorialropetope} as $c^v_{u,w}=\ct \partial_v(\schub{u}\schub{w})$.
\end{proof}

We now show that these LR coefficients are combinatorially nonnegative with an explicit combinatorial rule for computing them. To do this, we remind the reader of Sottile's Pieri rule for the Schubert polynomial expansion of $x_1\cdots x_k\schub{w}$ \cite[Theorem I]{Sot96} phrased  in terms of  the \emph{$k$-Bruhat order} introduced by Bergeron--Sottile \cite[Section 3]{BS98}.

Fix $k$ a positive integer.
We say that $u$ is covered by $v$ in $k$-Bruhat order if $v=us_{ij}$ if $i\leq k<j$ and $\ell(v)=\ell(u)+1$, in which case we write $u\prec_{v(j)}v$.
Here $s_{ij}$ denotes the transposition swapping $i$ and $j$.
A saturated chain $u\prec_{i_1} u_1\prec\cdots \prec_{i_p} u_p=v$ in $k$-Bruhat order is said to be decreasing if $i_1>\cdots>i_p$. If a decreasing chain from $u$ to $v$ exists then it is unique -- we write $u\overset{c_{k,p}}{\to}v$ and write $c_k$ as a shorthand for $c_{k,k}$. 
Sottile \cite{Sot96} established that
\begin{equation}\label{eq:sottile_pieri}
    c_{s_1\cdots s_{k},w}^v=\left \lbrace \begin{array}{ll} 1 & w\overset{c_{k}}{\longrightarrow} v\\ 0 & \text{otherwise.}\end{array}\right.
\end{equation}

\begin{eg}
    Let $w=15243$ and $k=3$, which means $s_1s_2s_3=2341$ in one line notation. One can check that 
    \[
        \schub{2341}\schub{15243}=\schub{263415}+\schub{264135},
    \]
    and the terms on the right-hand come from the following decreasing chains in $3$-Bruhat order:
    $152436 \prec_6 162435 \prec_4 164235 \prec_2  264135$ and $152436 \prec_6 162435 \prec_3 163425 \prec_2  263415$.
\end{eg}

Equation~\ref{eq:sottile_pieri} was used \cite{BS98} in the computation of the Schubert expansion of $\rope{i}\schub{w}$ -- we rederive this result emphasizing that by following the combinatorial proof above one could avoid all geometric considerations. 
From this result we also compute the Schubert expansion of $\tope{i}\schub{w}$.

\begin{prop}
\label{thm:pieriRiTi}
    Given a positive integer $i$ and $w\in S_{\infty}$ we have
    \begin{align}\label{eq:Ri_sottile}
    \rope{i}\schub{w}=&\sum_{w\overset{c_{i-1}}{\to}\ins_i(v)}\schub{v}\\   \label{eq:Ti_from_Ri}
    \tope{i}\schub{w}=&\delta_{i\in \des{w}}\sum_{ws_i\overset{c_{i-1}}{\to}\ins_i(v)} \schub{v}
    \end{align}
\end{prop}
\begin{proof}
Noting that $\schub{\idem}=1$ and $\ins_i(\idem)=s_1\cdots s_{i-1}$, we compute the coefficient of $\schub{v}$ in the Schubert expansion of $\rope{i}\schub{w}$ as 
    $$\ct \partial_v\rope{i}\schub{w}=\ct \partial_{\ins_i(v)}(\schub{\ins_i(\idem)}\schub{w})=c^{\ins_i(v)}_{s_1\cdots s_{i-1},w}.$$  where the first equality follows by \Cref{prop:combinatorialropetope}.
    Since $\tope{i}=\rope{i}\partial_i$ we infer~\eqref{eq:Ti_from_Ri} from~\eqref{eq:Ri_sottile}.
\end{proof}

\begin{eg}
    Take $w=146352$ and $i=4$.
    We get the following expansions
    \begin{align*}
        \rope{4}\schub{w}=\schub{346215}, \quad 
        \tope{4}\schub{w}=\schub{246315}
    \end{align*}
    from the decreasing chains $1465327 \prec_7 1475326 \prec_5 1574326 \prec_4 4571326=\varepsilon_4(346215)$ and $1463527 \prec_7 1473526 \prec_5 1573426 \prec_3 3571426=\varepsilon_4(246315)$  respectively.
\end{eg}

Before stating the central result in this section, we recall here that the basis of slide polynomials was introduced and studied in Assaf--Searles \cite{AS17}. 
This family of polynomials is precisely the family that appears in the seminal work of Billey--Jockusch--Stanley \cite{BJS93} describing Schubert polynomials combinatorially via reduced pipe dreams.
One consequence of our next result is a new proof of the nonnegativity of the slide expansion of Schubert polynomials, which can be considered either geometric or combinatorial depending on how one reads Proposition~\ref{prop:combinatorialropetope}.
\begin{thm}
\label{thm:LR}[{\Cref{maintheorem:LR}}]
    Let $w\in S_n$. The coefficients of
    \begin{enumerate}
        \item A monomial $x_1^{c_1}\cdots x_n^{c_n}$ in the monomial expansion of $\schub{w}$
        \item A slide polynomial coefficient in the slide polynomial expansion $\schub{w}=\sum a_{\textbf{i}}\slide{\textbf{i}}$
        \item A forest polynomial coefficient in the $m$-forest polynomial expansion $\schub{w}=\sum a_F\forestpoly[m]{F}$
    \end{enumerate}
    are all generalized LR coefficients $c^{v(\rt)}_{u(\rt),w}$ for some $\rt\in \rtseq_{N}$ with $N$ possibly larger than $n$, with an explicit combinatorially nonnegative rule for computing them.
\end{thm}
\begin{proof}
We note that if $\rt\in \rtseq$, if $N$ is sufficiently large then $\rletter{1}^{N-|\rt|}\rt\in \rtseq_N$ and $$\ct\prodoperator{\rletter{1}^{N-|\rt|}\rt}=\ct\rope{1}^{N-|\rt|}\prodoperator{\rt}=\ct\prodoperator{\rt}$$ so it suffices to find such an $\rt\in \rtseq$ rather than $\rtseq_N$.
Having established the explicit combinatorial rules for $\rope{i}\schub{w}$ and $\tope{i}\schub{w}$, we conclude that $\rtc\schub{w}$, and hence $\ct \rtc\schub{w}$, has an explicit combinatorial rule for computing it. 
Therefore it remains to show that each of these quantities can be computed as $\ct \rtc \schub{w}$ for some $\rt$. But this follows from results \cite{NST_2} wherein we find \emph{extractors} $\Omega$ for the families under consideration here.
\end{proof}

\section{Divided symmetrization}
\label{sec:DS}
Divided symmetrization (henceforth DS)  is the map $\langle\rangle_n:\poly_n\to \sym{n}$ given by
\begin{align}
    \label{eq:def_ds}
    \langle f\rangle_n=\sum_{\pi\in S_n}\pi\left(\frac{f}{\prod_{1\leq i\leq n-1}(x_i-x_{i+1})}\right).
\end{align}
It was introduced by Postnikov \cite[Section 3]{Pos09} in the context of volume polynomials of permutahedra.
This operator was then shown \cite{DS,NT20} to come up naturally when expressing the class of the permutahedral toric variety, i.e. the closure $\overline{T\cdot x}$ of a generic torus orbit in the flag variety $\fl{n}$, in terms of Schubert classes. Indeed, by work of Anderson--Tymoczko \cite{And10} we know that the constant term of divided symmetrization computes $\deg_{\overline{T\cdot x}}f$, so by the Kronecker duality between Schubert polynomials $\schub{w}$ and Schubert cycles $X^w$ we have
\begin{align}\label{eq:perm_class_expansion}
    [\overline{T\cdot x}]=\sum_{\substack{w\in S_n\\ \ell(w)=n-1}} \langle \schub{w}\rangle_n [X^w].
\end{align}
The main result of \cite{DS} is that $\langle f\rangle_n$ for any homogenous polynomial $f$ of degree $n-1$ can be computed from the knowledge of the representative in $\qscoinv{n}$ of $f$ expressed in the distinguished monomial basis of Aval--Bergeron--Bergeron \cite{ABB04}.
For this purpose forest polynomials were introduced in \cite{NT_forest}.
We now come full circle and demonstrate how our $\tope{}$ operators arise from grouping terms on the right-hand side of~\eqref{eq:def_ds} appropriately.

In contrast to previous work that has primarily focused on understanding DS in degree $n-1$, we give a formula that works in all degrees and is amenable to combinatorics.

\subsection{DS via a generalized trim}
\label{subsec:the_genesis}

For $1\leq i\leq n$ define the map $\rope{i,n}:\poly_n\to \poly_n$ by
\[
f(x_1,\dots,x_n)\mapsto f(x_1,\dots,x_{i-1},x_n,x_i,\dots,x_{n-1}).
\]
It is simply the action of the permutation with a single nontrivial cycle $(n,n-1,\ldots,i)$. 
Note that after the specialization $x_n=0$, we obtain the operator $\rope{i}$ on $\poly_{n}$. 
The following analogue of $\tope{i}$ is now natural: For $1\leq i\leq n-1$ define $\tope{i,n}:\poly_n\to \poly_n$ by
\[
    \tope{i,n}(f)=\frac{\rope{i+1,n}f-\rope{i,n}f}{x_i-x_n}=\rope{i+1,n}\partial_i\,f=\rope{i,n}\partial_i\,f.
\]

We recall the divided symmetrization operator $f\mapsto \langle f\rangle_n$ on $\poly_n$ defined by $$\langle f(x_1,\ldots,x_n) \rangle_n=\sum_{\sigma\in S_n}\frac{f(x_{\sigma(1)},\ldots,x_{\sigma(n)})}{(x_{\sigma(1)}-x_{\sigma(2)})\cdots (x_{\sigma(n-1)}-x_{\sigma(n)})}.$$

For $1\leq j\leq n$ let $\mathrm{cyc}_{j,n}\in S_n$ be the cycle $(n,n-1,\dots,n-j+1)$, and subsequently define the group algebra element $\tau_n\coloneqq \sum_{1\leq j\leq n} \mathrm{cyc}_{j,n}$.
    We then have the following factorization in the group algebra (over $\mathbb{Z}$) of $S_n$:
    \begin{align}\label{eq:important_factorization}
        \sum_{\sigma\in S_n}\sigma=\tau_2\tau_3\cdots\tau_n.
    \end{align}

\begin{thm}\label{th:ds_as_trims}
    For $f\in \poly_n$ we have
    \begin{align}
    \label{eq:ds_as_trims}
        \langle f\rangle_n=\tope{1,2}(\tope{1,3}+\tope{2,3})\cdots (\tope{1,n}+\cdots + \tope{n-1,n})f.
    \end{align}
    In particular, if $f$ is homogenous of degree $n-1$ then
    \[
    \langle f\rangle_n=\tope{1}(\tope{1}+\tope{2})\cdots (\tope{1}+\cdots + \tope{n-1})f=\sum_{\substack{F\in \suppfor{n}\\|F|=n-1}}|\Trim{F}|\tope{F}f.
    \]
\end{thm}
\begin{proof}
Let $D_n\coloneqq (x_1-x_2)(x_2-x_3)\cdots (x_{n-1}-x_{n})$.
Begin by noting thanks to~\eqref{eq:important_factorization} that    \begin{align}\label{eq:recursion_thanks_to_factorization_without_q}
        \langle f\rangle_n= \tau_2\tau_3\cdots\tau_n\left(\frac{f}{D_n}\right).
    \end{align}
    Note that $\tau_2\cdots \tau_{n-1}$ only acts on variables $x_1$ through $x_{n-1}$.
    Let us consider $\tau_n(f/D_n)$.
    
    Writing $\tau_n$ as $\sum_{1\leq j\leq n}\mathrm{cyc}_{j,n}$, we have
    \begin{align}
        \tau_n\left(\frac{f}{D_n}\right)=
        \frac{1}{D_{n-1}}\left(-\frac{\rope{1,n}f}{(x_1-x_n)}+\sum_{2\leq i\leq n-1}\frac{(x_i-x_{i-1})\rope{i,n}f}{(x_{i-1}-x_n)(x_i-x_n)} +\frac{\rope{n,n}f}{(x_{n-1}-x_n)}\right).
    \end{align}
    Rewriting $x_i-x_{i-1}$ as $(x_{i}-x_n)-(x_{i-1}-x_n)$ for $2\leq i\leq n-1$ then yields
    \begin{align}\label{eq:manifest_polynomiality}
        \tau_n\left(\frac{f}{D_n}\right)=
        \frac{1}{D_{n-1}}\left(-\frac{\rope{1,n}f}{(x_1-x_n)} 
        +
        \sum_{i=2}^{n-1}\frac{\rope{i,n}f}{(x_{i-1}-x_n)}
        -
        \sum_{i=2}^{n-1}\frac{\rope{i,n}f}{(x_i-x_n)}
        +
        \frac{\rope{n,n}f}{(x_{n-1}-x_n)}\right), 
    \end{align}
    which given the definition of the operator $\tope{i,n}$ translates to
    \begin{align}\label{eq:one_step_of_recursion}
        \tau_n\left(\frac{f}{D_n}\right)=\frac{1}{D_{n-1}}\left(\sum_{1\leq i\leq n-1}\tope{i,n}(f)\right).
    \end{align}

    The desired identity \eqref{eq:ds_as_trims} is obtained by immediate recursion using~\eqref{eq:recursion_thanks_to_factorization_without_q}.

    When $\deg(f)\leq n-1$, then $\langle f\rangle_n\in \ZZ$. In particular, setting $x_n=0$ in~\eqref{eq:one_step_of_recursion} does not impact the final result, and so we get
    \begin{align}
        \langle f\rangle_n=\tope{1}(\tope{1}+\tope{2})\cdots (\tope{1}+\cdots + \tope{n-1})f.
    \end{align}
    To conclude, observe that
    \begin{align}
       \tope{1}(\tope{1}+\tope{2})\cdots (\tope{1}+\cdots + \tope{n-1})=\sum_{\substack{(i_1,\dots,i_{n-1})\\ i_j\leq j}} \tope{i_1}\cdots \tope{i_k}.
    \end{align}
    The set of sequences over which the sum ranges is closed under the relation $\tope{a}\tope{b}=\tope{b}\tope{a+1}$ where $b<a$.
    Furthermore the subset of weakly increasing sequences $(i_1,\dots,i_k)$ subject to $i_j\leq j$ contains precisely the trimming sequences for indexed forests $F$ in $\suppfor{n}$ with $|F|=n-1$.
    It then follows
    \begin{equation}\label{eq:product_of_sums_of_Ts}
      \tope{1}(\tope{1}+\tope{2})\cdots (\tope{1}+\cdots + \tope{n-1})  = \sum_{\substack{F\in \suppfor{n}\\|F|=n-1}}|\Trim{F}|\,\tope{F}. \qedhere
    \end{equation}
\end{proof}
\begin{rem}
    Setting $t_i=x_{n-i}$, then the part of the above computation showing
    $$\langle f \rangle_n=\langle (\tope{1,n}+\cdots+\tope{n-1,n})f \rangle_{n-1}$$
    is in fact exactly the computation one would do using the Atiyah--Bott localization formula to show that under the projection map $\pi:X_{\Perm_{n-1}}\to X_{\Perm_{n-2}}$ induced by the map on normal fans obtained by forgetting the first coordinate, we have $\pi_*f=(\tope{1,n}+\cdots+\tope{n-1,n})f$. \footnote{Hai Zhu has independently found a proof of the factorization of DS in degree $n-1$ using our trimming operators.}
\end{rem}
The following corollary generalizes the nonnegativity of $a_w$ to higher degree Schuberts.
\begin{cor}\label{cor:higher_deg_schubert_ds}
For any Schubert polynomial $\schub{w}\in \poly_n$ we have
    $\langle \schub{w}\rangle_n$ is a polynomial with coefficients that are combinatorially nonnegative.
\end{cor}
\begin{proof}
    It suffices to show that $\tope{i,n}$ applied to $\schub{w}$ is a combinatorially nonnegative combination of Schubert polynomials. 
    To do this, note that $\tope{i,n}=\rope{i,n}\partial_i$ so it suffices to show this result for $\rope{i,n}$. 
    But this follows from one of the main results of Bergeron--Sottile \cite[Theorem 5.1]{BS02}, which expresses
    $$\schub{w}(x_1,\ldots,x_{i-1},z,x_i,\ldots)=\sum a_{i,v}z^i\schub{v}$$
    for combinatorially nonnegative coefficients $a_{i,v}$ (computed using chains in $k$-Bruhat order for various $k$).
\end{proof}

\subsection{Strict positivity of $\langle \schub{w}\rangle_n$ via trims}
\label{subsec:strict_positivity_aw}
We use Theorem~\ref{th:ds_as_trims} to give a straightforward combinatorial proof of the \emph{strict} positivity $\langle \schub{w}\rangle_n>0$ for $w\in S_n$,  answering \cite[Problem 6.6]{HHMP}.
This strict positivity was established earlier in \cite{NT20} by expressing $\langle \schub{w}\rangle_n$ as a sum of normalized mixed Eulerian numbers \cite[\S 16]{Pos09} and using positivity of mixed volumes. 
A combinatorial proof of nonnegativity was then given in \cite{NT_forest} via a parking procedure applied to reduced words of $w$, but strict positivity via this method is unclear.

Fix $w\in S_{n}$ with $\ell(w)=n-1$.
By Theorem~\ref{th:ds_as_trims} we have
\begin{equation}
    \langle \schub{w}\rangle_n=\sum_{i_j\le j} \tope{i_1}\cdots \tope{i_{n-1}}\schub{w}.
\end{equation}
It suffices to construct a sequence of positive integers $\mathbf{i}=(i_1,\dots,i_{n-1})$ satisfying $i_j\leq j$ for $1\leq j\leq n-1$ such that $\tope{\mathbf{i}}\schub{w}>0$.
\begin{lem}
\label{lem:tells_you_the_trim}
    For $w\in S_n$, if there is $1\le i\le n-1$ such that $w^{-1}(1) <i \leq w^{-1}(n)$, then the Schubert expansion of $\rope{i}\schub{w}$ contains some $\schub{v}$ with $v\in S_{n-1}$.
\end{lem}
\begin{proof}
By \Cref{thm:pieriRiTi}, it suffices to show there is a decreasing chain in the $(i-1)$-Bruhat order for some $v\in S_{n-1}$ as follows:
\[
w=w_0\prec_{b_1} w_1\prec \cdots \prec_{b_{i-1}} w_{i-1}=\ins_i(v).
\]
 We describe how to create this chain. 
 
 Let $a_1,\ldots,a_{i-1}$ be the numbers $w(1),\ldots,w(i-1)$ written in decreasing order.
 Since $w^{-1}(1)<i$, we must have $a_{i-1}=1$.
 For $1\leq j\leq i-1$, take 
 \begin{align}\label{eq:wsajbj}
    w_j\coloneqq w_{j-1}s_{a_jb_j}    
 \end{align}
 where  $b_j$ is the first number in the list $w_{j-1}(i),\ldots,w_{j-1}(n)$ such that $a_j\le b_j$. 
 The existence of $b_1$ follows because $n\in \{w(i),\ldots,w(n)\}$ and the existence of $b_j$ is clear for $j\ge 2$ because $a_{j-1}\in \{w_{j-1}(i),\ldots,w_{j-1}(n)\}$. 
 We also have $w_{j-1}\prec_{b_j} w_j$ in the $(i-1)$-Bruhat order by construction. 
 The numbers $b_1,b_2,\ldots$ are a decreasing sequence since if $b_j<b_{j+1}$ then $b_{j+1}$ must have appeared later than $b_j$ in the list $w_{j-1}(i),\ldots,w_{j-1}(n)$, so in the list $w_j(i),\ldots,w_j(n)$ we know that $a_j$ appears before $b_{j+1}$ which contradicts that $b_{j+1}$ is the first number in this list larger than $a_{j-1}$.

All transpositions $s_{a_jb_j}$ in~\eqref{eq:wsajbj} lie in $S_n$, so $w_{i-1}\in S_n$. 
Additionally, $a_{i-1}=1$ implies $w_{i-1}=w_{i-2}s_{(1,w_{i-1}(i))}$, which means $w_{i-1}(i)=1$. 
Therefore $w_{i-1}=\ins_i(v)$ for some $v\in S_{n-1}$.
\end{proof}
\begin{thm}
 Let $w\in S_n$ with $\ell(w)=n-1$.
 Then there exists a sequence $\mathbf{i}=(i_1,\dots,i_{n-1})$ of positive integers satisfying $i_j\leq j$ and $\tope{\mathbf{i}}\schub{w}>0$.
 In particular $a_w>0$.
 \end{thm}
 \begin{proof}
 We claim that there is a descent $i\in \des{w}$ such that $(ws_i)^{-1}(1)\le i \le (ws_i)^{-1}(n)$. Indeed, because $\ell(w)=n-1$ we know that $w^{-1}(1)\le w^{-1}(n)-1$. If this is an equality then $i=w^{-1}(1)$ works. Otherwise $w^{-1}(1)<w^{-1}(n)$ and we can take
 \begin{itemize}
 \item $i=w^{-1}(1)-1$ if $w^{-1}(1)\ne 1$
 \item $i=w^{-1}(n)$ if $w^{-1}(n)\ne n$
 \item Any $i\in \des{w}$ if $w^{-1}(1)=1$ and $w^{-1}(n)=n$.
 \end{itemize}
 By \Cref{lem:tells_you_the_trim} this implies that $\tope{i}\schub{w}=\rope{i}\partial_i\schub{w}$ contains a summand $S_v$ with $v\in S_{n-1}$. Iterating this produces the desired sequence.
\end{proof}

\subsection{$q$-divided symmetrization}

Let $\Delta_n\coloneqq \prod_{1\leq i<j\leq n}(x_i-x_j)$ and $\widehat{\Delta}_n^q\coloneqq \prod_{1\leq i+1<j\leq n}(qx_i-x_j)$.
Recall from \cite[Definition 4.1]{NT21} the $q$-DS operator $f\mapsto \langle f\rangle_n^q$ on $\poly_n$ defined by 
\[
\langle f(x_1,\ldots,x_n) \rangle_n^q=\sum_{\sigma\in S_n}\sigma \left(\frac{f(x_{1},\dots,x_{n})\cdot \widehat{\Delta}_n^q}{\Delta_n}\right).
\]
This clearly recovers ordinary DS at $q=1$.
The first and third authors, by leveraging a connection \cite[Theorem 4.11]{NT21} between $q$-DS  and coefficient extraction in the $q$-Klyachko algebra, showed that several known results involving DS $q$-deformed rather nicely.

We then have the following result that helps us completely understand $q$-DS in degree $n-1$. 
Compare this with the analogous statement in Theorem~\ref{th:ds_as_trims}.
\begin{thm}\label{th:qds_as_trims}
    For $f\in \poly_n$ homogenous of degree $n-1$ we have
    \[
    \langle f\rangle_n^q=\tope{1}(\tope{1}+q\tope{2})\cdots (\tope{1}+q\tope{2}+\cdots + q^{n-2}\tope{n-1})f.
    \]
\end{thm}
\begin{proof}
    
    Using~\eqref{eq:important_factorization} again we may write
    \begin{align}\label{eq:recursion_thanks_to_factorization}
        \langle f\rangle_n^q= \tau_2\tau_3\cdots\tau_n\left(\frac{f\cdot \widehat{\Delta}_n^q}{\Delta_n}\right)
    \end{align}
   We have
    \begin{multline}\label{eq:rationality_alone}
        \tau_n\left(\frac{f\cdot \widehat{\Delta}_n^q}{\Delta_n}\right)=
        \frac{\widehat{\Delta}_{n-1}^q}{\Delta_{n-1}} 
        \Big( \frac{\rope{1,n}f}{x_n-x_1}\prod_{i=2}^{n-1}\frac{qx_n-x_i}{x_n-x_i}
        +
        \sum_{j=2}^{n-1}
        \frac{(qx_{j-1}-x_j)\rope{j,n}f}{(x_{j-1}-x_n)(x_n-x_j)}\prod_{i=1}^{j-2}\frac{qx_i-x_n}{x_i-x_n}
        \prod_{i=j+1}^{n-1} \frac{qx_n-x_i}{x_{n}-x_i}
        \\+
        \frac{\rope{n,n}f}{x_{n-1}-x_n}\prod_{i=1}^{n-2}\frac{qx_i-x_{n}}{x_i-x_n}\Big).
    \end{multline}
    Since $\langle f\rangle_n^q\in \ZZ[q]$ given that $\deg(f)=n-1$ we may set $x_n=0$ throughout in the preceding equality without impacting the eventual result.
    We then obtain
    \begin{align}
        \tau_n\left(\frac{f\cdot \widehat{\Delta}_n^q}{\Delta_n}\right)|_{x_n=0}=
        \frac{\widehat{\Delta}_{n-1}^q}{\Delta_{n-1}} 
        \Big( -\frac{\rope{1}f}{x_1}
        +
        \sum_{j=2}^{n-1}
        q^{j-2}\frac{\rope{j}f}{x_{j-1}}
        -
        \sum_{j=2}^{n-1}
        q^{j-1}\frac{\rope{j}f}{x_j}
        +
        q^{n-2}\frac{\rope{n}f}{x_{n-1}}\Big)
    \end{align}
    which upon rearranging terms and using the definition of $\tope{}$ operators becomes
    \begin{align}
        \tau_n\left(\frac{f\cdot \widehat{\Delta}_n^q}{\Delta_n}\right)|_{x_n=0}=\frac{\widehat{\Delta}_{n-1}^q}{\Delta_{n-1}} \left(\sum_{1\leq j\leq n-1}q^{j-1}\tope{j}f\right).    
    \end{align}
    Recursing using~\eqref{eq:recursion_thanks_to_factorization} concludes the proof.
\end{proof}

    Observe that our statement of Theorem~\ref{th:qds_as_trims} only concerns the case $\deg(f)=n-1$, even though it is largely in the same spirit as the proof of Theorem~\ref{th:ds_as_trims}.
    Indeed, in contrast to the fact that the term within parentheses on the right-hand side in~\eqref{eq:manifest_polynomiality} is already a polynomial, we see that the analogous term in~\eqref{eq:rationality_alone} is not necessarily a polynomial. 
    This deficit is precisely what the specialization $x_n=0$ fixes as the resulting expression is then a polynomial.

\begin{rem}
    As in Theorem~\ref{th:ds_as_trims} we may express the operator $\tope{1}(\tope{1}+q\tope{2})\cdots(\tope{1}+q\tope{2}+\cdots+q^{n-2}\tope{n-1})$ as a $q$-weighted sum of $\tope{F}$ for $F\in \suppfor{n}$ with $|F|=n-1$.
    One writes
    \begin{align}
       \tope{1}(\tope{1}+q\tope{2})\cdots(\tope{1}+q\tope{2}+\cdots+q^{n-2}\tope{n-1})=\sum_{\substack{\mathbf{i}=(i_1,\dots,i_{n-1})\\ i_j\leq j}} q^{i_1+\cdots+i_{n-1}-(n-1)}\tope{i_1}\cdots\tope{i_{n-1}}. 
    \end{align}
    Collecting terms according to the indexing sequences $(i_1,\dots,i_k)$ that trim a common indexed forest $F$, we get one term for each decreasing labeling of $F$.
    Let us denote the set of such labelings by $\mathrm{Dec}(F)$.
    By reading such a labeling $\kappa$ in inorder one obtains a permutation in $S_{n-1}$ that we shall continue to call $\kappa$. 
    Tracking the $q$-weight tells us that
    \begin{align}\label{eq:product_of_sums_of_Ts_with_q}
    \sum_{\substack{\mathbf{i}=(i_1,\dots,i_{n-1})\\ i_j\leq j}} q^{i_1+\cdots+i_{n-1}-(n-1)}\tope{i_1}\cdots\tope{i_{n-1}}=\sum_{\substack{F\in \suppfor{n}\\|F|=n-1}}\left(\sum_{\kappa\in \mathrm{Dec}(F)}q^{\mathrm{inv}(\kappa)}\right)\tope{F},   
    \end{align}  
    where $\mathrm{inv}(\kappa)$ is the number of inversions of $\kappa$.
    To complete this remark we note that when $F\in\suppfor{n}$ with $|F|=n-1$ then $|\Trim{F}|$ in~\eqref{eq:product_of_sums_of_Ts} is the specialization at $x_i=1$ for all $i$ \cite[Proposition 3.14]{NT_forest}, whereas  $\sum_{\kappa\in \mathrm{Dec}(F)}q^{\mathrm{inv}(\kappa)}$ in~\eqref{eq:product_of_sums_of_Ts_with_q} is the specialization of $\forestpoly{F}$ at $x_i=q^{i-1}$ for all $i$.
\end{rem}
We conclude this section by speculating a bit more on the case where $\deg(f)$ exceeds $n-1$.
Recall that the $P$-Hall--Littlewood polynomial $P_{\lambda}(x_1,\dots,x_n;q)$ \cite{Mac95} for a partition $(\lambda_1,\dots,\lambda_n)$ is obtained by normalizing  the following symmetric polynomial 
    \[
        \sum_{\sigma\in S_n}\sigma\left(x_1^{\lambda_1}\cdots x_n^{\lambda_n}\prod_{1\leq i<j\leq n}\frac{x_i-qx_j}{x_i-x_j}\right).
    \]
Thus it is a common generalization of monomial symmetric polynomials and Schur polynomials, recovering the former at $q=1$ and the latter at $q=0$.
We conjecture the following generalization of Corollary~\ref{cor:higher_deg_schubert_ds} which we have verified on all permutations up until $S_8$.
\begin{conj}\label{conj:HL_positivity}
    For any Schubert polynomial $\schub{w}\in \poly_n$ we have
    \[
        \langle \schub{w}\rangle_n^q=\sum_{\lambda=(\lambda_1\geq \cdots \geq \lambda_n)}b_{\lambda,w}(q)P_{\lambda}(x_1,\dots,x_n;q^{-1}),
    \]
    where the $b_{\lambda,w}(q)$ are Laurent polynomials in $q$ with nonnegative integer coefficients.
\end{conj}
At $q=1$ this recovers the nonnegative expansion involving monomial symmetric polynomials.
We also note that the appearance of $q^{-1}$ is simply a reflection of the fact that our definition of $q$-DS employed factors $qx_i-x_j$ (keeping with the choice in \cite{NT21}) rather than $x_i-qx_j$ used in the definition of the $P$-Hall--Littlewoods. 
With this latter choice, Conjecture~\ref{conj:HL_positivity} becomes a $P$-Hall--Littlewood positivity statement with coefficients in $\NN[q]$.
The conjecture is already nontrivial in the simplest case where $\schub{w}$ is a dominant monomial $x_{1}^{\lambda_1}\cdots x_{n}^{\lambda_n}$ where $\lambda_1\geq \cdots \geq \lambda_n$, i.e. when $w$ is a 132-avoiding permutation.
The $q=1$ case is due to Postnikov \cite[Theorem 4.3]{Pos09} and is a simple consequence of a famous result of Brion \cite{Bri88} on integer point transforms of integral polytopes.
We are unaware of a generalization that incorporates the parameter $q$.

\section{A geometric perspective on a  formula of Gessel}
\label{sec:Gessel}

We now give a geometric interpretation for a result of Gessel on extracting the coefficients for the expansion of a symmetric polynomial in $\sym{n}$ into fundamental quasisymmetric polynomials in $\qsym{n}$.

In \cite[Corollary 8.6]{NST_a} the present authors showed that for any polynomial $f\in \poly_n$,  there is a decomposition into fundamental quasisymmetric polynomials given by
$$f(x_1,\ldots,x_n)=\sum_k\sum_{a_k,\ldots,a_n\ge 1}(\ct \mathsf{T}_{k}^{a_k}\mathsf{T}_{k+1}^{a_{k+1}}\cdots \mathsf{T}_{n}^{a_n}f)\,\slide{a_k,\ldots,a_n}(x_1,\ldots,x_n)$$
where  
$\slide{a_k,\ldots,a_n}(x_1,\ldots,x_n)$ for $a_k,\ldots,a_n\ge 1$ is the fundamental quasisymmetric polynomial\footnote{The indexing on fundamental quasisymmetric polynomials is different from \cite{NST_a}.} whose reverse lexicographic leading monomial is $x_k^{a_k}\cdots x_n^{a_n}$.

Recall that the Schubert cells $Y^\lambda$ and the opposite Schubert cells $Y_\mu$ of $\gr(n;N)$ are indexed by partitions $\lambda,\mu$ of integers into at most $n$ parts with the largest part of size at most $N-n$. 
If $\lambda\ge \mu$ coordinate-wise, then there is an associated Grassmannian Richardson variety $Y^\lambda_\mu\subset \gr(n;N)$.
We show that certain $\rt$-Richardson varieties with $\rt\in \rtseq_N$ push forward to the Grassmannian Richardson varieties associated to ribbon shapes.
\begin{defn}\label{defnribbon_attached_to_sequence}
    The ribbon shape $\lambda/\mu$ indexed by $(a_k,\ldots,a_n)$ for $a_k,\ldots,a_n\ge 1$ is the one which associated to partitions
    \begin{align*}
    \lambda&=(a_n+\cdots+a_k-(n-k),\ldots, a_n+a_{n-1}-1,a_n)\text{ of length $n-k+1$, and }\\
        \mu&=(a_n+\cdots+a_{k+1}-(n-k), \cdots ,a_n+a_{n-1}-2,a_n-1)\text{ of length $n-k$}
        .
    \end{align*}
\end{defn}
These are the skew shapes whose rows reading from top to bottom are of lengths $a_k,\ldots,a_n$, and such that there is exactly one square in the same column in two consecutive rows.
For instance take $k=2$, $n=5$, and suppose that $(a_2,\dots,a_5)=(2,1,1,3)$. Then the corresponding $\lambda$ and $\mu$ equal $(7-3,5-2,4-1,3)=(4,3,3,3)$ and $(5-3,4-2,3-1)=(2,2,2)$ respectively. 
See Figure~\ref{fig:ribbon} which demonstrates the ribbon $\lambda/\mu$ with shaded cells corresponding to $\mu$.
\begin{figure}[!ht]
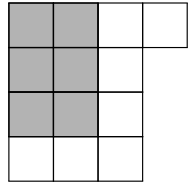

    \centering
 \ytableausetup{centertableaux}
\ytableaushort{}
* {4,3,3,3}
* [*(gray!60)]{2,2,2}
    \caption{The ribbon corresponding to $(a_2,\dots,a_5)=(2,1,1,3)$ according to Definition~\ref{defnribbon_attached_to_sequence}}
    \label{fig:ribbon}
\end{figure}

For fixed $N\ge n$ denote the forgetful map $\pi:\fl{N}\to \gr(n;N)$. 
We will show the following theorem.
\begin{thm}\label{th:geometric_gessel}
    For any $N\ge k+\sum a_i$, let $\rt=\rletter{1}^{N-k-\sum a_i}\tletter{1}^{a_k}\cdots \tletter{n-(k-1)}^{a_n}\rletter{1}^{k-1}\in \rtseq_N.$
    Then
    $$\pi_*(X(\rt))=Y^{\lambda}_{\mu}\subset \gr(n;N)$$
    where $\lambda/\mu$ is the ribbon indexed by $(a_k,\dots,a_n)$. 
\end{thm}
As a corollary of this geometric fact we show how to recover a formula of Gessel for the coefficient of the fundamental quasisymmetric polynomial $\slide{a_k,\ldots,a_n}(x_1,\ldots,x_n)$ in the fundamental quasisymmetric polynomial expansion of $f$.
\begin{cor}[{\cite[Theorem 3]{Ges84}}]
    For $a_k,\ldots,a_n\ge 1$, the coefficient of $\slide{a_k,\ldots,a_n}(x_1,\ldots,x_n)$ in $f\in \sym{n}$ is $\langle s_{\lambda/\mu},f\rangle_{\mathrm{Hall}}$ where $\lambda/\mu$ is the ribbon shape for $(a_k,\ldots,a_n)$ and $\langle ,\rangle_{\mathrm{Hall}}$ is the Hall inner product on $\sym{n}$.
\end{cor}
\begin{proof}
The Hall inner product with a skew Schur function $s_{\lambda/\mu}$ computes the degree of $f$ on the Grassmannian Richardson variety $Y^{\lambda}_\mu\subset \gr(n;N)$. 
The pullback $\pi^*:H^\bullet(\operatorname{Gr}(n;N))\to H^\bullet(\fl{N})$ identifies $H^\bullet(\operatorname{Gr}(n;N))$ with the subalgebra of $H^\bullet(\fl{N})=\coinv{N}$ generated by $\sym{n}$.  
Therefore
     $$\langle s_{\lambda/\mu},f\rangle_{\mathrm{Hall}}=\deg_{Y^\lambda_\mu}f=\deg_{\gr(n;N)} \pi_*[X(\rt)]f=\deg_{\fl{N}}[X(\rt)]f=\prodoperator{\rt}^N\,f=\ct \mathsf{T}_{k}^{a_k}\cdots \mathsf{T}_{n}^{a_n}f,$$
    where in the last step we used that $\mathsf{T}_{i-(k-1)}^{a_i}\rope{1}^{k-1}=\rope{1}^{k-1}\mathsf{T}_{i}^{a_i}$ and $\ct\rope{1}^{N-1-\sum a_i}=\ct$.
\end{proof}

\begin{rem} 
    To emphasize the contrast to the geometric approach here,
   we briefly remark on Gessel's proof\footnote{which while phrased in the language of (quasi)symmetric functions,  restricts to the finite variable case stated earlier} in \cite{Ges84}. 
   Given a symmetric function $f$, the coefficient of a fixed monomial symmetric function $m_{\lambda}$ equals $\langle f,h_{\lambda}\rangle_{\mathrm{Hall}}$ where $h_{\lambda}$ is the complete homogenous symmetric function associated with $\lambda$. Observing that the coefficient of $m_{\lambda}$ is in fact equal to the coefficient of the monomial quasisymmetric function $M_{\alpha}$ for  \emph{any} sequence $\alpha$ of positive integers that rearranges to $\lambda$, we see that the monomial quasisymmetric expansion may thus be easily obtained. 
   The punchline then relies on the observation that the expansion of a monomial quasisymmetric function in terms of fundamental quasisymmetric functions is the same as an expansion (going back to MacMahon \cite{MacM}) of skew ribbon Schurs in terms of homogenous symmetric functions.
\end{rem}

We prove Theorem~\ref{th:geometric_gessel} in the rest of this section. 
For $w\in S_N$, define $g(w)$ to be the $n$-Grassmannian permutation obtained by sorting the first $n$ and the last $N-n$ entries in the one-line notation of $w$. Then we always have $\pi(X^v_u)\subset Y^{g(v)}_{g(u)}$.
Note that $\pi(X_{u}^v)$ is a \emph{projected Richardson variety} in the sense of \cite{KLS14}, or \emph{projection variety} in the sense of \cite{BC12}.
In fact, since $\pi$ is a map into $\gr(n;N)$, we obtain \emph{positroid varieties}.
\begin{prop}
\label{prop:speyer}    $\pi|_{X^v_u}:X^v_u\to Y^{g(v)}_{g(u)}$ is birational if and only if $u\le_n v$ and $\ell(v)-\ell(u)=\ell(g(v))-\ell(g(u))$. (Here $\le_n$ stands for comparison in the $n$-Bruhat order.)
\end{prop}
\begin{proof}
As recalled in {\cite[Proposition 1.24]{speyer2024richardson}},  
the map $\pi|_{X^v_u}:X^v_u\to \pi(X^{v}_u)$ is birational if and only if $u\le_n v$. The result follows because $X^v_u$ and $Y^{g(v)}_{g(u)}$ are irreducible and we have
 $\dim X^v_u=\ell(v)-\ell(u)$ and $\dim Y^{g(v)}_{g(u)}=\ell(g(v))-\ell(g(u))$.
\end{proof}

\begin{lem}\label{lem:special_u_v}
    Letting $\rt=\rletter{1}^{N-k-\sum a_i}\tletter{1}^{a_k}\cdots \tletter{n-(k-1)}^{a_n}\rletter{1}^{k-1}\in \rtseq_N$ as in the statement of Theorem~\ref{th:geometric_gessel} we have
    \begin{align*}
        u(\rt)&=1\cdots (k-1) p_{k+1}\cdots p_{n}k\cdots\\
        v(\rt)&=1\cdots (k-1)p_k\cdots p_{n}k\cdots 
    \end{align*}
    where $p_i=k+\sum_{j=i}^na_j$ and the tails of unlisted numbers are in increasing order. Furthermore we have $u(\Omega)\leq_n v(\Omega)$.
\end{lem}
\begin{proof}
    The first half of the claim rests on the elementary observation that for a sequence $c_2,\ldots,c_m\ge 1$ we have
$$\ins_{m}^{c_m}\cdots \ins_{2}^{c_2}(\idem)=b_2b_3\cdots b_m 1 \cdots$$
where $b_i=1+\sum_{j=i}^mc_i$ and the tail of unlisted numbers is given by the numbers in $\NN\setminus \{b_2,\ldots,b_m\}$ listed in increasing order.

For the second half, it is not hard to see that $u(\rt)\le_n v(\rt)$  from the characterization \cite[Section 3.1]{BS98} that for $u,v$ permutations, we have $u\le_n v$ if
\begin{enumerate}
    \item $i\le n <j$ implies $u(i)\le w(i)$ and $u(j)\ge w(j)$, and
    \item if $i<j$, $u(i)<u(j)$, and $v(i)>v(j)$, then $i\le n < j$.
\end{enumerate}
We leave the details to the reader.
\end{proof}

\begin{proof}[Proof of Theorem~\ref{th:geometric_gessel}]
We set $u\coloneqq u(\Omega)$ and $v\coloneqq v(\Omega)$ for brevity. 
Lemma~\ref{lem:special_u_v} implies that
\begin{align*}
    g(u)&=1\cdots (k-1)kp_n\cdots p_{k+1}\cdots, \\
    g(v)&=1\cdots (k-1)p_n\cdots p_kk\cdots. 
\end{align*}
We have $\ell(g(u))=\ell(u)-(n-k+1)$ and $\ell(g(v))=\ell(v)-(n-k+1)$, so we conclude by Proposition~\ref{prop:speyer} that $\pi|_{X^v_u}:X^v_u\to Y^{g(v)}_{g(u)}$ is birational, and in particular $\pi_*(X^v_u)=Y^{g(v)}_{g(u)}$. It remains to identify $Y^{g(v)}_{g(u)}$.

The Lehmer codes of $g(u)$ and $g(v)$ are
    \begin{align*}
        \lcode{u(\rt)}&=(0^{k-1},0,a_n-1,a_n+a_{n-1}-2,\ldots,a_n+\cdots+a_{k+1}-(n-k))\\
        \lcode{v(\rt)}&=(0^{k-1},a_n,a_n+a_{n-1}-1,\ldots,a_n+\cdots+a_{k}-(n-k),0).
    \end{align*}
This implies that the skew shape $\lambda/\mu$ associated to $u$ and $v$ is the ribbon indexed by $(a_k,a_{k+1},\ldots,a_n)$, and we conclude.
\end{proof}

\begin{eg}\label{eg:gessel}
    Consider the fundamental quasisymmetric polynomial $\slide{2,1,1,3}(x_1,x_2,x_3,x_4)$. Then for $N=9$ we have $\rt=\tletter{1}^2\tletter{2}\tletter{3}\tletter{4}^3\in \rtseq_9$. The permutations $u(\rt),v(\rt)$ are equal to $65412378\cdots$ and $86541237\cdots$.
    On sorting the first four letters in each, we get the $4$-Grassmannian permutations $g(u)$ and $g(v)$ equaling $\textcolor{red}{1456}2378\cdots$ and $\textcolor{blue}{4568}1237\cdots$ with Lehmer codes $(0,2,2,2,0,\dots)$ and $(3,3,3,4,0,\dots)$, which yields the ribbon shape indexed by $(2,1,1,3)$.
\end{eg}

\appendix

\section{Nested Forest combinatorics}
\label{sec:nested_forests_extra}

Here we gather the proofs of various statements of \Cref{sec:compositesTR}.

\begin{lem}
 \label{lem:mnmonoidgeneration}
     The forests of the form $\ul{i}$ and $\ul{i_{\circ}}$ generate the monoid $\mnfor$.
 \end{lem}  

\begin{proof}[Proof of Lemma~\ref{lem:mnmonoidgeneration}]
    Let $\wt{G}$ be a nontrivial forest. If the $i$'th leaf is the trivial marked tree $\otimes$, then  $\wt{G}=\wt{F}\cdot \ul{i_{\circ}}$ where $\wt{F}$ is obtained by deleting $\otimes$. Otherwise, consider a tree in $\wt{G}$ which has no forests nested underneath it, so its support is a contiguous discrete interval $\{a,a+1,\ldots,b\}$, and let $v$ be an internal node of this tree which is farthest from the root. Then its children are two consecutive leaves $i,i+1$, and we can write $\wt{G}=\wt{F}\cdot i$ where $\wt{F}$ is obtained by deleting the two children of $v$ (making $v$ a leaf).
    
    We conclude by induction on the sum of the number of marked nodes and the number of internal nodes of $\wt{G}$.
\end{proof}

We recall the usual code map on plane forests.
 \begin{defn}
 \label{defncodemap}
 Let $\nvect$ denote finite supported sequences of natural numbers.
     For each plane forest $F$, define the \emph{code of $F$}, $c(F)=(c_1,c_2,\ldots)\in\nvect$ to be sequence defined by taking $c_i$ to be the number of internal nodes of $F$ whose leftmost leaf descendent is labeled $i\in \NN$.
 \end{defn}
 For marked nested plane forests, we will need to augment the code map. Let $\wt{\nvect}$ denote the set of infinite sequences indexed by $\NN$ whose elements are pairs $(\ins,i)$ with $\ins\in \{0,1\}$ and $i\in \{0,1,2,\ldots\}$, such that only finitely many terms are not $(0,0)$.
    
    We define the augmented code map $$\wt{c}:\mnfor\to \wt{\nvect}.$$ as follows: $\wt{c}(F)_i=(\ins,c)$ where $c$ is the number of non-leaf nodes such that iteratively taking left children leads to the $i$'th leaf, and $\ins=1$ if the highest such node is the root of its connected component and is marked, and $\ins=0$ otherwise.

We note that by its construction, we have
\begin{align}
\label{eq:special_code}
    \wt{c}(\ul{1_{\circ}}^{\ins_1}\cdot\ul{1}^{a_1}\cdot \ul{2_{\circ}}^{\ins_2}\cdot\ul{2}^{a_2}\cdots)=((\ins_1,a_1),(\ins_2,a_2),\ldots).
\end{align}

\begin{proof}[Proof of Theorem~\ref{thm:presentation_mnfor}]
We directly check the relations \begin{align}
    \ul{i}\cdot \ul{j}=\underbrace{\times\cdots \times}_{j-1} \wedge \underbrace{\times\cdots \times}_{i-j-1} \wedge \times\cdots = &\ul{j}\cdot \ul{i+1}\text{ for }i>j\\ \ul{i}\cdot \ul{j_{\circ}}=\underbrace{\times\cdots \times}_{j-1}\otimes \underbrace{\times \cdots \times}_{i-j}\wedge\times \cdots=&\ul{j_{\circ}}\cdot \ul{i+1}\text{ for }i\geq j,\\
        \ul{i_{\circ}}\cdot \ul{j}=\underbrace{\times\cdots\times}_{j-1}\wedge\underbrace{\times\cdots\times}_{i-j-1}\otimes\times \cdots =&\ul{j}\cdot \ul{{(i+1)}_{\circ}}\text{ for }i> j,\\\ul{i_{\circ}}\cdot \ul{j_{\circ}}=\underbrace{\times\cdots\times}_{j-1}\otimes\underbrace{\times\cdots\times}_{i-j}\otimes\times \cdots =&\ul{j_{\circ}}\cdot \ul{{(i+1)}_{\circ}}\text{ for }i\ge j.
\end{align} so there is a monoid morphism $\wt{\Th}\to \mnfor$ given by $i\mapsto \ul{i}$ and $i_{\circ}\mapsto \ul{i_{\circ}}$. Because $\mnfor$ is generated by the elements $\ul{i}$ and $\ul{i_{\circ}}$, the morphism $\wt{\Th}\to \mnfor$ is surjective.

Using the relations in $\wt{\Th}$, we may write any element as $1_{\circ}^{\ins_1}1^{a_1}2_{\circ}^{\ins_2}2^{a_2}\cdots $ with each $\ins_i\in \{0,1\}$. Indeed, let $m\in\wt{\Th}$, and consider any factorization of $m$ in the generators. The relations  of \Cref{defn:augmented_thompson} for $j=1$ and $i>1$ are $i\cdot 1=1\cdot (i+1)$,  $i\cdot 1_{\circ}=1_{\circ}\cdot (i+1)$, $i_{\circ}\cdot 1=1\cdot (i+1)_{\circ}$ $i_{\circ}\cdot 1_{\circ}=1_{\circ}\cdot (i+1)_{\circ}$.  Using those we can obtain a factorization of $m$ where the occurrences of $1$ and $1_{\circ}$ form a prefix. Using the extra relations  $1\cdot 1_{\circ}=1_{\circ}\cdot 2$, $1_{\circ}\cdot 1_{\circ}=1_{\circ}\cdot 2_{\circ}$, we can in fact have this prefix be of the form $1_{\circ}^{\ins_1}1^{a_1}$. It is then easily concluded by induction that we can have the desired factorization, as the relations  $j>1$ are just shifts of the ones for $j=1$.

Now these special factorizations map to distinct elements of $\mnfor$ via the morphism, because of~\eqref{eq:special_code}. We conclude that the morphism is injective, and thus  $\wt{\Th}\cong \mnfor$.

Finally, the augmented code map is surjective by~\eqref{eq:special_code}, and injective because $\wt{\Th}\to \mnfor$ is surjective and we have just shown that the composite $\wt{\Th}\to \mnfor\to \wt{\nvect}$ is injective.
\end{proof}

\begin{proof}[Proof of forward implications of \Cref{thm:faithfulindexing}]
The first forward implication was already proved. Now note that if $\wt{F}=\mnf{\rt}$, then the forest $\mnf{\rletter{1}^N\rt}$ is obtained by marking the first $N$ unmarked outer trees in  $\wt{F}$. It follows that if $\nf{\rt}=\nf{\rt'}$ then there are $N,M\gg 1$ so that $\mnf{\rletter{1}^N\rt}=\mnf{\rletter{1}^M\rt'}$; by the first part, we get $\rope{1}^N\rtc=\rope{1}^M\prodoperator{\rt'}$. Now for any operator $A$ and polynomial $f$ we have $\rope{1}^NAf=\rope{1}^Mf=\ct Af$ for $N,M$ large enough, so we can conclude $\ct\rtc=\ct\prodoperator{\rt'}$.
\end{proof}

We now prove the reverse implications. For this we make use of the following proposition of independent interest:

\begin{prop}
Let $c\in \nvect$. For each leaf $i$, consider the unique path $P_i$ of length $c_i$ towards the root of the tree containing $i$. If any of these paths does not exist, or if the paths do not partition all internal nodes of $\nf{\rt}$ then $\ct \prodoperator{\wt{F}} x^c=0$. Otherwise, $\ct \prodoperator{\wt{F}} x^c=(-1)^m$ where $m$ is the number of edges connecting a node to a right child in $\bigsqcup P_i$.
\end{prop}

\begin{proof}
We note that $\ct \ul{\emptyset}=\ct$, so the description holds for $\emptyset$. Write $\wt{F}=\mnf{\rt}$ for some $\rt\in \rtseq$, and assume the result holds for all smaller length $\rt\in \rtseq$. We write $\rt=(\rt',X)$, and let $\wt{F'}=\mnf{\rt'}$. There are two cases.

If $X=\rletter{i}$ then $$\ct\rtc=\ct\prodoperator{\rt'}\rope{i}x^c=\begin{cases}\ct\prodoperator{\rt'}x_1^{c_1}\cdots x_{i-1}^{c_{i-1}}x_i^{c_{i+1}}\cdots&c_i= 0\\0&c_i\neq 0.\end{cases}$$
The description holds by induction, since $\wt{F}$ has a trivial tree with support $\{i\}$ in this case and $\wt{F'}$ is obtained by removing this trivial tree.

If $X=\tletter{i}$ then $$\ct\rtc=\ct\prodoperator{\rt'}\tope{i}x^c=\begin{cases}\ct\prodoperator{\rt'}x_1^{c_1}\cdots x_{i-1}^{c_{i-1}} x_i^{c_{i}-1}x_{i+1}^{c_{i+2}}\cdots&c_i\ne 0\text{ and }c_{i+1}=0\\
-\ct\prodoperator{\rt'}x_1^{c_1}\cdots x_{i-1}^{c_{i-1}} x_i^{c_{i+1}-1}x_{i+1}^{c_{i+2}}\cdots&c_i=0\text{ and }c_{i+1}\ne 0\\
0&\text{otherwise.}\end{cases}$$
Here $\wt{F}$ has a terminal node with support $\{i,i+1\}$ and $\wt{F'}$ is obtained by transforming it into a leaf. The description holds by induction.
\end{proof}

\begin{proof}[Proof of the reverse implication of \Cref{thm:faithfulindexing}]
We first show how to reconstruct $\wh{F}$ from $\ct\prodoperator{\wt{F}}$. For each $i$, let $b_i$ be maximal such that $f\mapsto \phi( x_i^{b_i}f)$ is not identically zero; then $b_i$ is the distance in $\nf{\rt}$ from leaf $i$ to the root of the tree it lies in. Now for  $i\neq j$, $f\mapsto  \phi(x_i^{b_i}x_j^{b_j}f)$ is identically zero if and only if $i,j$ lie in the same tree of $\nf{\rt}$. Finally, if $i,j$ lie in the same tree in $\nf{\rt}$ then there is a smallest integer $c_{i,j}$ such that $\ct\rtc x_i^{b_i}x_j^{c_{i,j}}f$ is identically zero, namely  the distance from leaf $j$ to the nearest common parent of $i$ and $j$. This information is enough to reconstruct the nested forest.

    Now, we show how to recover $\widetilde{F}$ from $\prodoperator{\widetilde{F}}$. We already know that $\wh{F}=\wh{F'}$ by what was said above, so we need to show that the markings coincide. Now for any $i$ we have $\prodoperator{i\cdot\wt{F}}=\tope{i}\prodoperator{\wt{F}}= \tope{i}\prodoperator{\wt{F'}}=\prodoperator{i\cdot \wt{F'}}$ so the unmarked forests associated to $i\cdot \wt{F}$ and $i\cdot \wt{F'}$ have to coincide. Because the unmarked forests for $\wt{F}$ and $\wt{F'}$ coincide this implies that $i$'th unmarked outer roots of $\mnf{\rt}$ and $\mnf{\rt'}$ are the same, so $\wt{F}=\wt{F'}$ as desired.
\end{proof}

\begin{proof}[Proof of \Cref{prop:fullsupp}]
The proposition holds for $n=1$ as $\rt=\rletter{1}$ in this case and $\mnf{\rt}=\otimes \times \times \cdots$. For $n>1$, the fact that $\mnf{\rt}$ satisfies the properties is then immediate by induction. For the converse statement, by the proof of \Cref{lem:mnmonoidgeneration} we have that $\wt{F}=\wt{F'}\cdot i$ for some $i<n$ or $\wt{F}=\wt{F'}\cdot i_{\circ}$ for some $i\leq n$. Then $\wt{F'}$ satisfies the properties for $n-1$, so we have $\wt{F'}=\mnf{\rt'}$ for some $\rt'\in \rtseq_{n-1}$ by immediate induction, and thus $\wt{F}=\mnf{\rt'\xletter{}}$ with $\xletter{}=\tletter{i}$ with $i<n$ or $\xletter{}=\rletter{i}$ for $i\geq n$. We have thus $\rt'\xletter{}\in\rtseq_n$ which concludes the proof.
\end{proof}

\section{Combinatorics associated to the permutations \texorpdfstring{$u(\rt),v(\rt)$}{u(Omega),v(Omega}}
\label{sec:uvPerms}
Given a permutation $u\in S_\infty$ and $a\geq 1$, let $\ell_a(u)$ be the number of $b\geq a$ such that $u^{-1}(b)\leq u^{-1}(a)$. 
In words, $\ell_a(u)$ counts numbers larger than $a$ that occur before $a$ in the one-line notation of $u$. 
For instance if $w=426153789\cdots$ then $\ell_a(u)$ is given by $4,2,4,1,2,1,1,1,1,\ldots$ for $a=1,2,\ldots$. 

This gives a bijection between $S_\infty$ and sequences of positive integers eventually constant equal to $1$. 
It is clear that the permutations in $S_n$ are characterized as those with $\ell_i(u)\le n+1-i$ for $1\le i \le n$ and $\ell_i(u)=1$ for $i\ge n+1$.
For any sequence $j_1,\ldots,j_{k}$, the permutation $$w=\ins_{j_{1}}\cdots \ins_{j_k}(1)$$
is then characterized by $\ell_a(w)=j_a$ for $a=1,\ldots,k$ and $\ell_a(w)=1$ for $a\ge k+1$.
\begin{cor}
\label{cor:R1uv}
If $|\rt|\le |\rt'|$, then $(u(\rt),v(\rt))=(u(\rt'),v(\rt'))$ if and only if $\rt'=\rletter{1}^{|\rt'|-|\rt|}\rt$.
\end{cor}
\begin{prop}

For fixed $m$, the map $\rt\to (u(\rt),v(\rt))$ is a bijection between length $m$ sequences in $\rtseq$ and pairs of permutations $(u,v)\in S_\infty\times S_\infty$ such that $\ell_a(v)-\ell_a(u)\in\{0,1\}$ for $1\le a \le m$ and $\ell_a(v)=\ell_a(u)=1$ for $a\ge m+1$.    
\end{prop}
\begin{proof}
    For a length $m$ sequence $\rt=\xletter{1}\cdots \xletter{m}\in \rtseq$, we have $\ell_i(u)=\ell_i(v)=1$ for $i\ge m+1$ and for $1\le i \le m$ we have $$(\ell_i(u),\ell_i(v))=\begin{cases}(k,k)&\xletter{m+1-i}=\rletter{k}\\(k,k+1)&\xletter{m+1-i}=\tletter{k},\end{cases}$$
    which clearly establishes the bijection.
\end{proof}
We now describe what pairs $(u(\rt),v(\rt))$ are produced for $\rt\in \rtseq_n$. This is not particularly a restriction by \Cref{cor:R1uv}, since $\rletter{1}^{N-1-|\rt|}\rt\in \rtseq_{N}$ for $N$ sufficiently large. 

    


\begin{prop}
\label{prop:rtseq_and_permutations}
The map $\rtseq_n\to S_n\times S_n$ given by $\rt\mapsto (u(\rt),v(\rt))$ is an injective map with image those pairs $(u,v)\in S_n\times S_n$ such that $\ell_a(v)-\ell_a(u)\in \{0,1\}$ for $1\le a \le n$.
\end{prop}
\begin{proof}
    For $\rt\in \rtseq_n$, the way that $u(\rt),v(\rt)$ are both of the form $\ins_{j_1}\cdots \ins_{j_{n-1}}(\idem)$ with $j_i\le n+1-i$  which shows that $(u(\rt),v(\rt))\in S_n\times S_n$. Conversely, if $(u,v)\in S_n\times S_n$ with $\ell_a(v)-\ell_a(v)\in \{0,1\}$ for $1\le a \le n$ then setting $\rt=\xletter{1}\cdots \xletter{n}$ where
    $$\xletter{i}=\begin{cases}\rletter{k}&(\ell_{n+1-i}(u),\ell_{n+1-i}(v))=(k,k)\\
    \tletter{k}&(\ell_{n+1-i}(u),\ell_{n+1-i}(v))=(k,k+1)\end{cases}$$
    achieves $(u(\rt),v(\rt))=(u,v)$.
\end{proof}
We have the following special cases with $|\rt|_{\tletter{}}=0$ minimal and $|\rt|_{\tletter{}}=n-1$ maximal.
\begin{prop}
    \label{cor:minmaxrt}\leavevmode
\begin{enumerate}
    \item The pairs of permutations  $(u(\rt),v(\rt))\in S_n\times S_n$ for $\rt\in \rtseq_n$ and $|\rt|_{\tletter{}}=0$ are precisely those with $u=v$.
    \item Let $c_n=n12\cdots (n-1)=s_{n-1}s_{n-2}\cdots s_1$ be the reverse long cycle in $S_n$. The pairs of permutations $(u(\rt),v(\rt))\in S_n\times S_n$ with $\rt\in \rtseq_n$ and $|\rt|_{\tletter{}}=n-1$ are exactly the pairs $(u,uc_n)$ with $u\in S_n$ and $u(n)=n$. 
\end{enumerate}
\end{prop}

We now describe the Bruhat intervals associated to $\rt$.
\begin{prop}For $x,y$ permutations, if $j,k\in \{i,i+1\}$ then $\ins_j(x)\le \ins_k(y)$ if and only if $j\le k$ and $x\le y$.
\end{prop}
\begin{proof}
   Recall the \emph{tableau criterion} \cite[Theorem 2.6.3]{BjBr05}, which says that $a\le b$ in the Bruhat order if for all $t$, the first $t$ entries in the one-line permutation of $a$ when sorted are elementwise less than the first $t$ entries in the one-line permutation of $b$ when sorted.

    In particular, if we have $\ins_r(a)\le \ins_s(b)$, then the tableau criterion applied to the first $s$ entries implies that $1$ appears in the first $r$ entries so $r\le s$. 
    Therefore it remains to show that if $i\le j \le k \le i+1$ then $\ins_j(x)\le \ins_k(y)$ if and only $x\le y$.

    If $i\le j\le k\le i+1$ and $x\le y$ then the tableau criterion shows directly that $\ins_j(x)\le \ins_j(y)\le \ins_k(y)$. Conversely, suppose $\ins_j(x)\le \ins_k(y)$. 
    Then for $1\le r \le j-1$, the first $r$ entries of $x$ and $y$ agree respectively with the first $i$ entries of $\ins_j(x)$ and $\ins_k(y)$ decremented by $1$, while for $r\ge j$ they agree respectively with the first $r+1$ entries of $\ins_j(x)$ and $\ins_k(y)$ after removing $1$ and decrementing the remaining entries by $1$. 
    Hence the tableau criterion for $\ins_j(x)\le \ins_k(y)$ implies the tableau criterion for $x\le y$.
\end{proof}
\begin{cor}
\label{prop:insBruhat}
    If $u\le v$ in the Bruhat order, then the following are true.
    \begin{enumerate}
        \item The map $w\mapsto \ins_i(w)$ is a poset isomorphism between $[u,v]$ and $[\ins_i(u),\ins_i(v)]$.
        \item $[\ins_i(u),\ins_{i+1}(v)]=[\ins_i(u),\ins_i(v)]\sqcup [\ins_{i+1}(u),\ins_{i+1}(v)]$. Furthermore for $w,w'\in [u,v]$ and $j,k\in \{i,i+1\}$ we have
        $\ins_j(w)\le \ins_k(w')$ if and only if $w\le w'$ and $j\le k$.
    \end{enumerate}
\end{cor}
\begin{cor}
\label{prop:Bruhat}
 The Bruhat intervals contained in $[u(\rt),v(\rt)]$ are those of the form $[u(\rt'),v(\rt')]$ with $\rt'$ obtained by replacing some $\tletter{i}$ in $\rt$ with either $\rletter{i}$ or $\rletter{i+1}$.
\end{cor}

\section{Moment polytopes of \texorpdfstring{$\rt$}{Omega}-Richardson varieties}
\label{sec:Moment}
For $\lambda=(\lambda_1>\cdots>\lambda_n)$ a decreasing sequence of integers, let $P^{v(\rt)}_{u(\rt)}(\lambda)$ be the moment polytope of $X(\rt)$ under the generalized Pl\"ucker embedding $\operatorname{Pl}_\lambda:\fl{n}\to \mathbb{P}^{n!-1}$. We now show that
\begin{enumerate}
    \item $P^{v(\rt)}_{u(\rt)}(\lambda)\cong \gz(\lambda;\rt) $ where $\gz(\lambda;\rt)$ is a face associated to $\rt$  of the Gelfand--Zetlin polytope determined by $\lambda$.
    \item The moment polytopes for $\rt\in \rtseq_n$ are exactly the faces of a subdivision of the permutahedron
    $$\Perm(\lambda)\coloneqq \{w\cdot (\lambda_1,\ldots,\lambda_n)\suchthat w\in S_n\}$$
    into combinatorial cubes.
    \item $P^{v(\rt)}_{u(\rt)}(\lambda)\cong \cube(\wh{F};\lambda)$ where $\cube(\wh{F};\lambda)$ is a polytope intrinsically associated to the nested forest $\wh{F}=\nf{\rt}$.
\end{enumerate}
 The first two points slightly generalize \cite[Lemma 6.3]{HHMP} which in our terminology shows the first point for those $\rt\in \rtseq_n$ with $|\rt|_t=n-1$ maximal, and shows that these are the top dimensional faces of a dissection of $\Perm(\lambda)$ into combinatorial cubes. In \Cref{rem:combcubeP1} we explain how we can see the different Bott manifold structures on the toric variety associated to $\cube(\wh{F};\lambda)$ (which is naturally isomorphic to every $X(\rt)$ with $\wh{F}(\rt)=\wh{F}$) from the recursive structure of the combinatorial cube $\cube(\widehat{F};\lambda)$, giving a different perspective to the identifications from \Cref{sec:Bott}.

\subsection{The generalized Pl\"ucker embedding and moment polytopes of Richardson varieties}
For $\lambda_1>\cdots>\lambda_n$ a strictly decreasing sequence, the generalized Pl\"ucker embedding $\operatorname{Pl}_\lambda:\fl{n}\to \mathbb{P}^{n!-1}$ associated to $\lambda$ realizes $\fl{n}$ as a projective variety, and takes $$M\mapsto [\prod_{i=1}^{N} (\det M_{\{\sigma(1),\ldots,\sigma(i)\},\{1,\ldots,i\}})^{\lambda_i-\lambda_{i+1}}]_{\sigma\in S_n}$$
where $M_{A,B}$ is the submatrix of $M$ determined by the rows indexed by $A$ and the columns indexed by $B$, and we set $\lambda_{n+1}=0$ by convention.

The $T$-fixed points of $\fl{n}$ are the permutation matrices $\{P_{\sigma}\suchthat \sigma\in S_n\}$, and we identify $P_{\sigma}$ with the permutation $\sigma$ itself. Recall that $S_n$ acts on $\mathbb{R}^n$ by $\sigma\cdot e_i=e_{\sigma(i)}$, or equivalently
\begin{align}
    \sigma\cdot (\lambda_1,\ldots,\lambda_n)=(\lambda_{\sigma^{-1}(1)},\ldots,\lambda_{\sigma^{-1}(n)}).
\end{align}
Let the $i$'th standard character of $T$ be denoted by $t_i$.
If we take the action of $T$ on $\mathbb{P}^{n!-1}$ to be obtained by scaling the coordinate indexed by $\sigma\in S_n$ by $t_{\sigma(1)}^{\lambda_{1}}t_{\sigma(2)}^{\lambda_{2}}\cdots t_{\sigma(n)}^{\lambda_n}=t^{\sigma\cdot (\lambda_1,\ldots,\lambda_n)}$, then the generalized Pl\"ucker embedding is a $T$-equivariant embedding  $\fl{n}\hookrightarrow \mathbb{P}^{n!-1}$, and $\sigma\in \fl{n}$ is mapped to the standard basis vector $e_\sigma$.

For a torus $T'$ acting on $\mathbb{P}^N$ by scaling the $i$'th coordinate by $t_i^{\mu_i}$,
with $\mu_i\in \ZZ^n$, the moment map is the map
$\mathbb{P}^{N}\to \RR^n$
given by $\mu([\{x_i\}])\coloneqq \sum_{i=1}^n \frac{|x_i|^2}{\sum_{j=0}^N |x_j|^2}\mu_i.$
 By the convexity theorem for moment maps \cite{GS82}, if the characters $\mu_i$ are distinct (so the only $T'$-fixed points are the $e_i$), then for a $T'$-invariant subvariety $Y\subset \mathbb{P}^{N}$ we have
$\mu(Y)=\operatorname{conv}\{\mu(e_i):e_i\in Y\}.$

For $X\subset \fl{n}$ the image of a $T$-invariant subvariety of $\fl{n}$, the moment map $\mu_\lambda:X\to \RR^n$ is obtained by composing the embedding $X\to \mathbb{P}^{n!-1}$ with the moment map on $\mathbb{P}^{n!-1}$ under the torus action mentioned earlier. In particular, this implies
$$\mu_\lambda(X)=\operatorname{conv}\{(\sigma\cdot (\lambda_1,\ldots,\lambda_n)\suchthat \sigma \in X\}.$$

Hence for example the moment polytope of $\fl{n}$ is the permutahedron $$\Perm(\lambda)\coloneqq \operatorname{conv}\{\sigma\cdot (\lambda_1,\ldots,\lambda_n)\suchthat \sigma\in S_n\}\subset \RR^n.$$
For the Richardson variety $X^v_u$, we have $\sigma\in X^v_u$ if and only if $u\le \sigma\le v$ in the Bruhat order.
\begin{defn}
    If $u\le v$ in Bruhat order on $S_n$, then we define the \emph{twisted inverse Bruhat interval polytope} by
    $$P^v_u(\lambda)=\operatorname{conv}(\{w\cdot (\lambda_1,\ldots,\lambda_n)\suchthat w\in [u,v]\}).$$
\end{defn}

We therefore have the following theorem of Tsukerman--Williams.
\begin{thm}[{\cite[Proposition 2.9]{TW15}}]
    The moment polytope of $X^u_v$ under the generalized Pl\"ucker embedding $\operatorname{Pl}_\lambda$ is the twisted Bruhat interval polytope
   $P^{u}_v(\lambda)$.
\end{thm}
\begin{rem}
We note the relation to the \emph{Bruhat interval polytope} \cite{TW15} $Q^v_u=\operatorname{conv}(\{(w(1),\ldots,w(n)): w\in [u,v]\})$ is that $$P^{v}_{u}(n,\ldots,1)=Q^{w_{0,n}u^{-1}}_{w_{0,n}v^{-1}}.$$
\end{rem}

\subsection{Recalling GZ polytopes}

We now introduce the Gelfand--Zetlin polytopes. 
To keep exposition brief we quickly recall material that has appeared elsewhere (cf. \cite[Section 7.1]{NTremixed} for instance, which builds upon \cite{HHMP, KST12}).
\begin{defn}
For $\lambda=(\lambda_1> \cdots > \lambda_n)\in \mathbb{R}^n$
the \emph{Gelfand--Zetlin polytope} $\gz(\lambda)$ is the polytope in $\mathbb{R}^{n(n-1)/2}$ containing points $(p_{i,j})_{1\leq i\leq j\leq n}$ such that $p_{i,1}=\lambda_i$ for $1\leq i\leq n$ and
\begin{equation*}
    p_{i,j}\geq p_{i+1,j+1}\geq p_{i+1,j}.
\end{equation*}
\end{defn}
We shall think of points in $\gz(\lambda)$ as fillings of triangular/staircase shape as shown in Figure~\ref{fig:gz_pattern} where the bottom row reads $\lambda_1$ through $\lambda_n$ left to right.
Such a filling is often called a \emph{GZ pattern}.

\begin{figure}[!ht]
    \centering
    \includegraphics[scale=1.1]{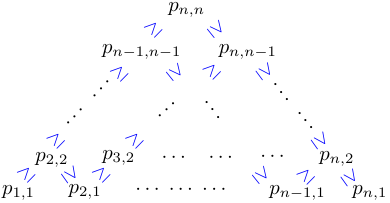}
    \caption{A GZ pattern determining a point in $\gz(\lambda)$}
    \label{fig:gz_pattern}
\end{figure}
We will specify a face of $\gz(\lambda)$ by a ``face diagram'', consisting of a graph whose underlying vertex set is the positions of entries in a GZ pattern, and whose edges record defining hyperplanes whose intersection is the face.
We represent the facet $p_{i,j}=p_{i+1,j+1}$ by a \emph{left edge} connecting vertices $(i,j)$ and $(i+1,j+1)$, and the facet $p_{i,j}=p_{i,j-1}$ by a \emph{right edge} between vertices $(i,j)$ and $(i,j-1)$. 
The face diagram associated to a face is not necessarily unique. 
We remark that it is typical in the literature, after \cite{Kog00},  to call a \emph{Kogan} (respectively \emph{dual Kogan}) face of $\gz(\lambda)$ to be one determined by a collection of left (respectively right) edges, so that every face is the intersection of a Kogan and a dual Kogan face.

The following linear projection (cf. \cite[\S 5]{HHMP}) will be used to relate certain faces of $\gz(\lambda)$ and Bruhat interval polytopes.

\begin{defn}
The linear projection
   $\mu:\gz(\lambda)\to \Perm(\lambda)$ is defined by
\begin{equation}
    \mu(p)=(y_1-y_2,y_2-y_3,\dots,y_{n}-y_{n+1}),
\end{equation}
where $y_i\coloneqq \sum_{i\leq k\leq n} p_{k,i}$ for $1\leq j\leq n+1$ (so $y_{n+1}=0$).
\end{defn}
Referring to Figure~\ref{fig:gz_pattern} the $y_i$s  are produced by summing the entries in each row, going bottom to top, and then $\mu(p)$ records the differences in successive rows.

\subsection{$\rt$-Richardson varieties and GZ-polytopes}

Recall the notion of trimming diagram given $\rt \in \rtseq_n$ that was introduced at the end of \Cref{sec:compositesTR}.
We now derive face diagrams from these trimming diagrams.
\begin{defn}
    For $\rt\in \rtseq_n$ we define a face $\gz(\lambda;\rt)$ of $\gz(\lambda)$ by removing the blue edges from the trimming diagram associated to $\rt$.
\end{defn}
See Figure~\ref{fig:g_omega_and_face_diagram} for the face diagram obtained from the trimming diagram in Figure~\ref{fig:Trimming_diagram_and_forest}.

\begin{figure}[!ht]
    \centering
    \includegraphics[scale=1]{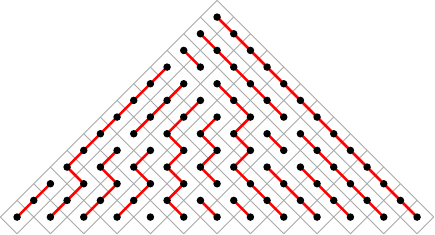}
    \caption{The sequence $\rletter{1}^4\tletter{2}\rletter{3}\rletter{4}\tletter{6}\tletter{2}\tletter{7}\rletter{1}\tletter{6}\rletter{5}\in \rtseq_{13}$}
    \label{fig:g_omega_and_face_diagram}
\end{figure}
\begin{prop}
\label{prop:annoyingGZ}
    The polytope $\gz(\lambda;\rt)$ is cut out by the $p_{i,1}=\lambda_i$, and $\gz$ equalities along the red edges and the inequalities corresponding to the blue edges. Furthermore, in $\gz(\lambda;\rt)$ we have $p_{i,i}>\cdots >p_{n,i}$ for all $i$.
\end{prop}
\begin{proof}
    First, note that every GZ inequality is clearly cut out by the red equalities, the blue inequalities, and the inequalities $p_{i,j}\ge p_{i+1,j}$ for all $i,j$. For $i=1$ the totality of inequalities $p_{1,1}> \cdots > p_{n,1}$ follow because $\lambda_1>\cdots > \lambda_n$. It is straightforward to see that the totality of inequalities $p_{i,i}> \cdots > p_{n,i}$ together with the red/blue restrictions between $p_{i,j}$ and $p_{i+1,k}$ for various $i$ and $k$ then imply the totality of inequalities $p_{i+1,1}>\cdots > p_{n,i+1}$ and we conclude.
\end{proof}
\begin{cor}
\label{cor:annoyingGZ}
    The faces of $\gz(\lambda;\rt)$ are exactly those of the form $\gz(\lambda;\rt')$ where $\rt'$ is obtained by changing some subset of the $\tletter{i}$ to $\rletter{i}$ or $\rletter{i+1}$.
\end{cor}
\begin{proof}
    The faces of $\gz(\lambda;\rt)$ are obtained by setting some of the defining inequalities to equalities. Since there are only blue edge inequalities, we can either replace two blue edges at an internal vertex with a right red edge or a left red edge, and these correspond to replacing a $\tletter{i}$ with either $\rletter{i}$ or $\rletter{i+1}$ respectively (noting that we cannot add in both edges since then that would force an equality $p_{i,j}=p_{i,j+1}$).
\end{proof}
Figure~\ref{fig:squares_in_gz} depicts the three-dimensional $\gz(\lambda_1,\lambda_2,\lambda_3)$. The two-dimensional faces $\gz(\lambda;\rletter{1}\tletter{1}\tletter{1})$ and $\gz(\lambda;\rletter{1}\tletter{1}\tletter{2})$ are shaded. 
In accordance with the preceding corollary we note, for instance, that the vertices of these faces may be obtained by replacing the $\tletter{i}$ with $\rletter{i}$ or $\rletter{i+1}$.
Observe that the unique vertex of degree $4$ does not appear as a vertex of either $\gz(\lambda;\rletter{1}\tletter{1}\tletter{1})$ or $\gz(\lambda;\rletter{1}\tletter{1}\tletter{2})$.
\begin{figure}[!ht]
    \centering
    \includegraphics[scale=0.9]{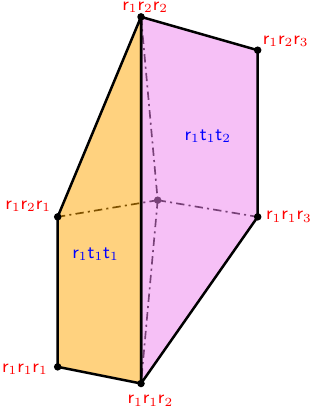}
    \caption{$\gz(\lambda_1,\lambda_2,\lambda_3) \subset \mathbb{R}^3$ with maximal dimension faces $\gz(\lambda;\rt)$ shaded}
    \label{fig:squares_in_gz}
\end{figure}

\begin{defn}[{\cite[Section 2.2.2]{Kog00}}]
    A \emph{nondegenerate (or simple) vertex}  of $\gz(\lambda)$ is a vertex of the form $\gz(\lambda;\rt) $ where $\rt=\rletter{i_1}\cdots \rletter{i_{n}}\in \rtseq_n$. We define the associated permutation $\pi(v)\in S_n$ to be obtained by setting $\pi(v)(i)$ to be the number of vertices in the path for $v$ starting at $(i,1)$.
\end{defn}
\begin{prop}
\label{prop:simplevertex}
If $\rt=\rletter{i_1}\cdots \rletter{i_{n}}\in \rtseq_n$ then for $v=\gz(\lambda;\rt) $ the associated vertex we have $\pi(v)=\ins_{i_{n}}\cdots \ins_{i_1}(\idem)$. 
Furthermore, $\mu_\lambda(v)=\pi(v)\cdot (\lambda_1,\ldots,\lambda_n)$. 
\end{prop}
\begin{proof}
    We induct on $n$. 
    If $v'=\gz(\lambda;\rletter{i_1}\cdots \rletter{i_{n-1}})$ then we can check directly that
    $$\pi(v)_j=\begin{cases}\pi(v')_j&j<i_{n}\\1&j=i_{n}\\\pi(v')_{j-1}&j\ge i_{n}+1\end{cases}$$
    which shows that $\pi(v)=\ins_{i_{n}}(\pi(v'))$. Finally, the index $\alpha_i$ of the path which stop at the $i$'th row from the bottom is given by $\alpha_i=\pi(v)^{-1}(i)$ so 
    \begin{equation*}
    \mu(v)=(\lambda_{\pi(v)^{-1}(1)},\ldots,\lambda_{\pi(v)^{-1}(n)})=\pi(v)\cdot (\lambda_1,\ldots,\lambda_n).\qedhere
    \end{equation*}
\end{proof}

The similarity between Figure~\ref{fig:squares_in_gz} and Figures~\ref{fig:cube_subdivision} is explained by our next result.

\begin{thm}
    For $\rt\in \rtseq_n$, the map $\mu$ linearly identifies $$\gz(\lambda;\rt) \cong P^{u(\rt)}_{v(\rt)}(\lambda).$$
\end{thm}
\begin{proof}
All vertices of $\gz(\lambda;\rt) $ are nondegenerate, so letting $$S(\rt)=\{\pi(v)\suchthat v\text{ vertex of }\gz(\lambda;\rt)\},$$ it suffices by \Cref{prop:simplevertex} to check that $S(\rt)=[u(\rt),v(\rt)].$ 

By \Cref{prop:simplevertex} again we have
$$S(\rt\, \xletter{})=\begin{cases}\{\ins_i(w)\suchthat w\in S(\rt)\}&\xletter{}=\rletter{i}\\\{\ins_i(w)\suchthat w\in S(\rt)\}\sqcup \{\ins_{i+1}(w)\suchthat w\in S(\rt)\}&\xletter{}=\tletter{i}.\end{cases}$$
On the other hand, by \Cref{prop:insBruhat} we have $$[u(\rt \,\xletter{}),v(\rt\, \xletter{})]=\begin{cases}[\ins_i(u),\ins_i(v)]=\{(\ins_i(w)\suchthat w\in [u,v]\}&\xletter{}=\rletter{i}\\ [\ins_i(u),\ins_{i+1}(v)]=\{\ins_i(w)\suchthat w\in [u,v]\}\sqcup \{\ins_{i+1}(w)\suchthat w\in [u,v]\}&\xletter{}=\tletter{i},\end{cases}$$
so we conclude by induction.
\end{proof}
\subsection{The HHMP subdivision of the permutahedron}
It was shown in \cite{HHMP} that the moment polytopes of the top-dimensional $X(\rt)$ give a subdivision of the permutahedron into $(n-1)!$ combinatorial cubes \cite{HHMP}, which we call the $\hhmp$-subdivision. We reprove this here for the convenience of the reader.
\begin{thm}[{\cite{HHMP}}]
    The moment polytopes $P^{v(\rt)}_{u(\rt)}(\lambda)$ are the faces of a subdivision (the $\hhmp$-subdivision) of $\Perm(\lambda)$ into combinatorial cubes.
\end{thm}
\begin{proof}
The faces of $P^{v(\rt)}_{u(\rt)}(\lambda)$ are of the form $P^{v(\rt')}_{u(\rt')}(\lambda)$, so it remains to show that the relative interiors $P^{v(\rt')}_{u(\rt')}(\lambda)^{\circ}$ partition $\Perm(\lambda)$.
    For $z\in[\lambda_i,\lambda_{i+1}]$ we have (see \cite[Proposition 3.7]{Liu16}) $$\Perm(\lambda_1,\ldots,\lambda_n)\cap \{x_1=z\}=\{z\}\times \Perm(\lambda_1,\ldots,\lambda_{i-1},\lambda_i+\lambda_{i+1}-z,\lambda_{i+2},\ldots,\lambda_n).$$
    For such a $z$, we let $\lambda(z)\coloneqq (\lambda_1,\ldots,\lambda_{i-1},\lambda_i+\lambda_{i+1}-z,\lambda_{i+2},\ldots,\lambda_n)$.
    We recursively record a sequence in $\rtseq_n$ by
    $$S(z_1,\ldots,z_n;\lambda)=\begin{cases}S(z_2,\ldots,z_n;\lambda(z_1))\rletter{i}&z_1=\lambda_i\\
    S(z_2,\ldots,z_n;\lambda(z_1))\tletter{i}&z_1\in (\lambda_i,\lambda_{i+1}).\end{cases}$$
    Then it is straightforward to check inductively that $S(z_1,\ldots,z_n;\lambda)$ records the unique relative interior $\mu(\gz(\lambda;\rt) ^{\circ})=P^{v(\rt)}_{u(\rt)}(\lambda)^{\circ}$ that $z$ belongs to.
\end{proof}
\subsection{$\rt$-Richardson varieties and Nested Forest polytopes}
\begin{defn}
    For $\wh{F}\in \nfor$ we define the nested forest polytope $\cube({\wh F;\lambda})$ as the polytope of functions $\phi\in \RR^{\internal{\widehat{F}}}$ such that for the extension $\phi_\lambda:\internal{\widehat{F}}\sqcup \NN\to \RR$ taking $i\mapsto \lambda_i$, we have for each $v\in \internal{\widehat{F}}$ the inequalities $$\phi_\lambda(v_L)\ge \phi_\lambda(v)\ge \phi_\lambda(v_R).$$
\end{defn}
Figure~\ref{fig:nested_forest_cube} demonstrates a nested forest $\wh{F}$ as well as the inequalities cutting out the associated polytope (in $\mathbb{R}^3$).

\begin{figure}[!ht]
    \centering
    \includegraphics[width=0.75\linewidth]{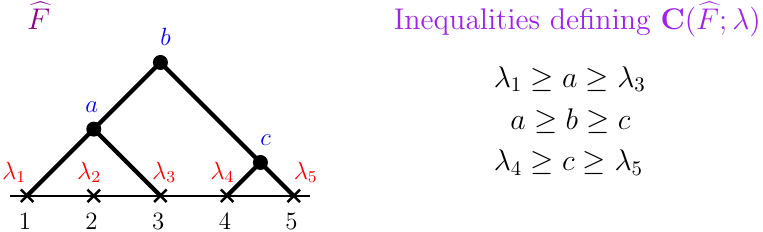}
    \caption{A nested forest $\wh{F}$ and the corresponding $\mathbf{C}(\wh{F};\lambda)$}
    \label{fig:nested_forest_cube}
\end{figure}

\begin{thm}
There is a linear isomorphism
    $\gz(\lambda;\rt) \cong \cube(\nf{\rt};\lambda)$ given by assigning a GZ pattern to the function which takes an internal node $v$ to the corresponding $p_{i,j}$ in the trimming diagram of $\wh{F}$.
\end{thm}
\begin{proof}
    This immediately follows from \Cref{prop:annoyingGZ} after identifying equal $p_{i,j}$.
\end{proof}
\begin{cor}
For a strictly decreasing $\lambda$, 
    $\cube(\widehat{F};\lambda)$ is a combinatorial cube.
\end{cor}
\begin{cor}
    $X(\wh{F})$ is the toric variety associated to the polytope $\cube{(\wh{F};\lambda)}$.
\end{cor}
\begin{proof}
By definition,
    $X(\wh{F})$ is isomorphic to $X(\rt)$ for any $\rt\in \Trim{\wh{F}}$, and we have just shown that $X(\rt)$ is the toric variety associated to $\gz(\lambda;\rt)\cong \cube{(\wh{F}(\rt);\lambda)}=\cube{(\wh{F};\lambda)}$.
\end{proof}
\begin{rem}
    A tedious verification which we omit shows that the isomorphism $\pi_{\rt'}^{-1}\pi_{\rt}:X(\rt)\to X(\rt')$ is induced by the composite isomorphism on moment polytopes
    $\gz(\lambda;\rt)\cong \cube{(\wh{F};\lambda)}\cong \gz(\lambda;\rt')$, which shows that the identification of $X(\wh{F})$ with the toric variety associated to $\cube{(\wh{F};\lambda)}$ can be done in such a way that the isomorphism $\pi_{\rt'}:X(\rt')\to X(\wh{F})$ for any $\rt'\in \Trim{\wh{F}}$ is induced by the linear isomorphism $\gz(\lambda;\rt')\cong \cube{(\wh{F};\lambda)}$.
\end{rem}

One way of creating a combinatorial cube is to take a linear family of combinatorial cubes $C(\lambda)\subset \RR^n$ (i.e. for  two strongly equivalent combinatorial cubes $C(a)$ and $C(b)$ we define $C(\lambda)$ for $\lambda\in [a,b]$ by $C(ta+(1-t)b)=tC(a)+(1-t)C(b)$) and take $$C'=\{(x,z)\in \RR^n\times [a,b]\suchthat x\in C(z)\}.$$
It turns out that $\cube(\widehat{F};\lambda)$ has this recursive structure that makes it into a combinatorial cube. Furthermore different ways of expressing $\widehat{F}$ as a product of $\underline{i}$ and $\underline{i_{\circ}}$ give different realizations of $\cube(\widehat{F};\lambda)$ as a combinatorial cube.
\begin{thm}
Let $\lambda$ be a decreasing sequence, and let $\lambda'=(\lambda_1,\ldots,\lambda_{i-1},\lambda_{i+1},\ldots)$. Then we have
$$\cube(\wh{F}\cdot i_{\circ};\lambda)=\cube(\wh{F};\lambda').$$
    For $z\in [\lambda_{i-1},\lambda_i]$, let $\lambda(z)=(\lambda_1,\ldots,\lambda_{i-1},z,\lambda_{i+2},\ldots)$. Then
    $$\cube(\wh{F}\cdot i;\lambda)=\{(y,z)\suchthat z\in [\lambda_{i},\lambda_{i+1}],\text{ and }y\in \cube{(\wh{F};\lambda(z))}\}\subset  \RR^{\internal{\widehat{F}}}\times \RR\cong \RR^{\internal{\widehat{F}\cdot i}}.$$
\end{thm}
\begin{rem}
\label{rem:combcubeP1}
    The toric variety associated to a $P'$ arising as the total family of a linear family of polytopes $P(z)$ strongly equivalent to a fixed polytope $P$ realizes the toric variety of $P'$ as $\PP(\ul \CC \oplus \mathcal{L})$ for a toric line bundle $\mathcal{L}$ on the toric variety associated to $P$. The different ways of realizing the nested forest polytope therefore correspond to different Bott manifold structures on $X(\rt)$, giving an alternate perspective on the computations in \Cref{sec:Bott}.

\end{rem}
\bibliographystyle{hplain}
\bibliography{main_2nd_part.bib}

\begin{thebibliography}{10}

\bibitem{And10}
D.~Anderson and J.~Tymoczko.
\newblock Schubert polynomials and classes of {H}essenberg varieties.
\newblock {\em J. Algebra}, 323(10):2605--2623, 2010.

\bibitem{AS17}
S.~Assaf and D.~Searles.
\newblock Schubert polynomials, slide polynomials, {S}tanley symmetric
  functions and quasi-{Y}amanouchi pipe dreams.
\newblock {\em Adv. Math.}, 306:89--122, 2017.

\bibitem{ABB04}
J.-C. Aval, F.~Bergeron, and N.~Bergeron.
\newblock Ideals of quasi-symmetric functions and super-covariant polynomials
  for {$S_n$}.
\newblock {\em Adv. Math.}, 181(2):353--367, 2004.

\bibitem{BelkBrown05}
J.~M. Belk and K.~S. Brown.
\newblock Forest diagrams for elements of {T}hompson's group {$F$}.
\newblock {\em Internat. J. Algebra Comput.}, 15(5-6):815--850, 2005.

\bibitem{BeGa23}
N.~Bergeron and L.~Gagnon.
\newblock The excedance quotient of the {B}ruhat order, {Q}uasisymmetric
  {V}arieties and {T}emperley-{L}ieb algebras, 2023, 2302.10814.

\bibitem{BS98}
N.~Bergeron and F.~Sottile.
\newblock Schubert polynomials, the {B}ruhat order, and the geometry of flag
  manifolds.
\newblock {\em Duke Math. J.}, 95(2):373--423, 1998.

\bibitem{BS02}
N.~Bergeron and F.~Sottile.
\newblock Skew {S}chubert functions and the {P}ieri formula for flag manifolds.
\newblock {\em Trans. Amer. Math. Soc.}, 354(2):651--673, 2002.

\bibitem{BGG73}
I.~N. Bern\v{s}te\u{\i}n, I.~M. Gelfand, and S.~I. Gelfand.
\newblock Schubert cells, and the cohomology of the spaces {$G/P$}.
\newblock {\em Uspehi Mat. Nauk}, 28(3(171)):3--26, 1973.

\bibitem{BC12}
S.~Billey and I.~Coskun.
\newblock Singularities of generalized {R}ichardson varieties.
\newblock {\em Comm. Algebra}, 40(4):1466--1495, 2012.

\bibitem{BJS93}
S.~C. Billey, W.~Jockusch, and R.~P. Stanley.
\newblock Some combinatorial properties of {S}chubert polynomials.
\newblock {\em J. Algebraic Combin.}, 2(4):345--374, 1993.

\bibitem{BjBr05}
A.~Bj\"{o}rner and F.~Brenti.
\newblock {\em Combinatorics of {C}oxeter groups}, volume 231 of {\em Graduate
  Texts in Mathematics}.
\newblock Springer, New York, 2005.

\bibitem{Bor53}
A.~Borel.
\newblock Sur la cohomologie des espaces fibr\'{e}s principaux et des espaces
  homog\`enes de groupes de {L}ie compacts.
\newblock {\em Ann. of Math. (2)}, 57:115--207, 1953.

\bibitem{BS58}
R.~Bott and H.~Samelson.
\newblock Applications of the theory of {M}orse to symmetric spaces.
\newblock {\em Amer. J. Math.}, 80:964--1029, 1958.

\bibitem{Bri88}
M.~Brion.
\newblock Points entiers dans les poly\`edres convexes.
\newblock {\em Ann. Sci. \'{E}cole Norm. Sup. (4)}, 21(4):653--663, 1988.

\bibitem{CFPnotes96}
J.~W. Cannon, W.~J. Floyd, and W.~R. Parry.
\newblock Introductory notes on {R}ichard {T}hompson's groups.
\newblock {\em Enseign. Math. (2)}, 42(3-4):215--256, 1996.

\bibitem{DehTes19}
P.~Dehornoy and E.~Tesson.
\newblock Garside combinatorics for {T}hompson's monoid {$F^+$} and a hybrid
  with the braid monoid {$B^+_\infty$}.
\newblock {\em Algebr. Comb.}, 2(4):683--709, 2019.

\bibitem{Dem74}
M.~Demazure.
\newblock D\'{e}singularisation des vari\'{e}t\'{e}s de {S}chubert
  g\'{e}n\'{e}ralis\'{e}es.
\newblock {\em Ann. Sci. \'{E}cole Norm. Sup. (4)}, 7:53--88, 1974.

\bibitem{FuSt97}
W.~Fulton and B.~Sturmfels.
\newblock Intersection theory on toric varieties.
\newblock {\em Topology}, 36(2):335--353, 1997.

\bibitem{GaoZhu24}
Y.~Gao and H.~Zhu.
\newblock Boolean structure constants, 2024, 2405.05527.

\bibitem{Ges84}
I.~M. Gessel.
\newblock Multipartite {$P$}-partitions and inner products of skew {S}chur
  functions.
\newblock In {\em Combinatorics and algebra ({B}oulder, {C}olo., 1983)},
  volume~34 of {\em Contemp. Math.}, pages 289--317. Amer. Math. Soc.,
  Providence, RI, 1984.

\bibitem{GK94}
M.~Grossberg and Y.~Karshon.
\newblock Bott towers, complete integrability, and the extended character of
  representations.
\newblock {\em Duke Math. J.}, 76(1):23--58, 1994.

\bibitem{GS82}
V.~Guillemin and S.~Sternberg.
\newblock Convexity properties of the moment mapping.
\newblock {\em Invent. Math.}, 67(3):491--513, 1982.

\bibitem{Han73}
H.~C. Hansen.
\newblock On cycles in flag manifolds.
\newblock {\em Math. Scand.}, 33:269--274 (1974), 1973.

\bibitem{HHMP}
M.~Harada, T.~Horiguchi, M.~Masuda, and S.~Park.
\newblock The volume polynomial of regular semisimple {H}essenberg varieties
  and the {G}elfand-{Z}etlin polytope.
\newblock {\em Proc. Steklov Inst. Math.}, 305:318--344, 2019.

\bibitem{Hu23}
D.~Huang.
\newblock Schubert products for permutations with separated descents.
\newblock {\em Int. Math. Res. Not. IMRN}, (20):17461--17493, 2023.

\bibitem{huang2022bumpless}
D.~Huang and P.~Pylyavskyy.
\newblock Bumpless pipe dream {R}{S}{K}, growth diagrams, and {S}chubert
  structure constants, 2022, 2206.14351.

\bibitem{KST12}
V.~A. Kirichenko, E.~Yu. Smirnov, and V.~A. Timorin.
\newblock Schubert calculus and {G}elfand-{T}setlin polytopes.
\newblock {\em Uspekhi Mat. Nauk}, 67(4(406)):89--128, 2012.

\bibitem{KLS14}
A.~Knutson, T.~Lam, and D.~E. Speyer.
\newblock Projections of {R}ichardson varieties.
\newblock {\em J. Reine Angew. Math.}, 687:133--157, 2014.

\bibitem{KZJ1}
A.~Knutson and P.~Zinn-Justin.
\newblock Schubert puzzles and integrability i: invariant trilinear forms,
  2020, 1706.10019.

\bibitem{KZJ3}
A.~Knutson and P.~Zinn-Justin.
\newblock Schubert puzzles and integrability iii: separated descents, 2023,
  2306.13855.

\bibitem{Kog00}
M.~Kogan.
\newblock {\em Schubert geometry of flag varieties and {G}elfand-{C}etlin
  theory}.
\newblock ProQuest LLC, Ann Arbor, MI, 2000.
\newblock Thesis (Ph.D.)--Massachusetts Institute of Technology.

\bibitem{Ko01}
M.~Kogan.
\newblock R{C}-graphs and a generalized {L}ittlewood-{R}ichardson rule.
\newblock {\em Internat. Math. Res. Notices}, (15):765--782, 2001.

\bibitem{LS82}
A.~Lascoux and M.-P. Sch\"{u}tzenberger.
\newblock Polyn\^{o}mes de {S}chubert.
\newblock {\em C. R. Acad. Sci. Paris S\'{e}r. I Math.}, 294(13):447--450,
  1982.

\bibitem{Lee21}
E.~Lee, M.~Masuda, and S.~Park.
\newblock Toric {B}ruhat interval polytopes.
\newblock {\em Journal of Combinatorial Theory, Series A}, 179:105387, 2021.

\bibitem{LSS06}
C.~Lenart, S.~Robinson, and F.~Sottile.
\newblock Grothendieck polynomials via permutation patterns and chains in the
  {B}ruhat order.
\newblock {\em Amer. J. Math.}, 128(4):805--848, 2006.

\bibitem{lian2023hhmp}
C.~Lian.
\newblock The {H}{H}{M}{P} decomposition of the permutohedron and degenerations
  of torus orbits in flag varieties, 2023, 2309.01747.

\bibitem{Liu16}
G.~Liu.
\newblock Mixed volumes of hypersimplices.
\newblock {\em Electron. J. Combin.}, 23(3):Paper 3.19, 19pp, 2016.

\bibitem{Mac95}
I.~G. Macdonald.
\newblock {\em Symmetric functions and {H}all polynomials}.
\newblock Oxford Mathematical Monographs. The Clarendon Press, Oxford
  University Press, New York, second edition, 1995.
\newblock With contributions by A. Zelevinsky, Oxford Science Publications.

\bibitem{MacM}
P.~A. MacMahon.
\newblock {\em Combinatory analysis}.
\newblock Chelsea Publishing Co., New York, 1960.
\newblock Two volumes (bound as one).

\bibitem{MP08}
M.~Masuda and T.~E. Panov.
\newblock Semi-free circle actions, {B}ott towers, and quasitoric manifolds.
\newblock {\em Mat. Sb.}, 199(8):95--122, 2008.

\bibitem{NST_a}
P.~Nadeau, H.~Spink, and V.~Tewari.
\newblock Quasisymmetric divided differences, 2024, 2406.01510.

\bibitem{NST_2}
P.~Nadeau, H.~Spink, and V.~Tewari.
\newblock Schubert polynomial expansions revisited, 2024, 2407.02375.

\bibitem{DS}
P.~Nadeau and V.~Tewari.
\newblock Divided symmetrization and quasisymmetric functions.
\newblock {\em Selecta Math. (N.S.)}, 27(4):Paper No. 76, 24, 2021.

\bibitem{NT21}
P.~Nadeau and V.~Tewari.
\newblock A $q$-analogue of an algebra of {K}lyachko and {M}acdonald's reduced
  word identity, 2021, arXiv:2106.03828.

\bibitem{NT23}
P.~Nadeau and V.~Tewari.
\newblock ${P}$-partitions with flags and back stable quasisymmetric functions,
  2023, 2303.09019.

\bibitem{NT20}
P.~Nadeau and V.~Tewari.
\newblock The permutahedral variety, mixed {E}ulerian numbers, and principal
  specializations of {S}chubert polynomials.
\newblock {\em Int. Math. Res. Not. IMRN}, (5):3615--3670, 2023.

\bibitem{NTremixed}
P.~Nadeau and V.~Tewari.
\newblock Remixed {E}ulerian numbers.
\newblock {\em Forum Math. Sigma}, 11:Paper No. e65, 26pp, 2023.

\bibitem{NT_forest}
P.~Nadeau and V.~Tewari.
\newblock Forest polynomials and the class of the permutahedral variety.
\newblock {\em Adv. Math.}, 453:Paper No. 109834, 33pp, 2024.

\bibitem{PeSa22}
O.~Pechenik and M.~Satriano.
\newblock Quasisymmetric {S}chubert calculus, 2023, 2205.12415.

\bibitem{PeSa23}
O.~Pechenik and M.~Satriano.
\newblock James reduced product schemes and double quasisymmetric functions.
\newblock {\em Adv. Math.}, 449:Paper No. 109737, 28pp, 2024.

\bibitem{PW22}
O.~Pechenik and A.~Weigandt.
\newblock A dual {L}ittlewood-{R}ichardson rule and extensions, 2022,
  2202.11185.

\bibitem{Pos09}
A.~Postnikov.
\newblock Permutohedra, associahedra, and beyond.
\newblock {\em Int. Math. Res. Not. IMRN}, (6):1026--1106, 2009.

\bibitem{Sot96}
F.~Sottile.
\newblock Pieri's formula for flag manifolds and {S}chubert polynomials.
\newblock {\em Ann. Inst. Fourier (Grenoble)}, 46(1):89--110, 1996.

\bibitem{speyer2024richardson}
D.~E. Speyer.
\newblock Richardson varieties, projected {R}ichardson varieties and positroid
  varieties, 2024, 2303.04831.

\bibitem{StThesis}
R.~P. Stanley.
\newblock {\em Ordered structures and partitions}.
\newblock Memoirs of the American Mathematical Society, No. 119. American
  Mathematical Society, Providence, R.I., 1972.

\bibitem{TW15}
E.~Tsukerman and L.~Williams.
\newblock Bruhat interval polytopes.
\newblock {\em Adv. Math.}, 285:766--810, 2015.

\bibitem{Sunic07}
Z.~\v{S}uni\'{c}.
\newblock Tamari lattices, forests and {T}hompson monoids.
\newblock {\em European J. Combin.}, 28(4):1216--1238, 2007.

\end{thebibliography}
\end{document}